\documentclass[a4paper,11pt,reqno]{amsart}
\usepackage{amsmath,amssymb,amsthm,amsfonts}
\usepackage[dvips]{graphicx}

\newdimen\margin   
\def\textno#1&#2\par{%
   \margin=\hsize
   \advance\margin by -4\parindent
          \setbox1=\hbox{\sl#1}%
   \ifdim\wd1 < \margin
      $$\box1\eqno#2$$%
   \else
      \bigbreak
      \hbox to \hsize{\indent$\vcenter{\advance\hsize by -3\parindent
      \sl\noindent#1}\hfil#2$}%
      \bigbreak
   \fi}

\theoremstyle{plain}

\newtheorem{theorem}{Theorem}

\newtheorem*{thm*}{Theorem}
\newtheorem*{theorem*}{Theorem}
\newtheorem{prop}[theorem]{Proposition}

\newtheorem*{claim*}{Claim}
\newtheorem{lemma}[theorem]{Lemma}

\newtheorem{conjecture}[theorem]{Conjecture}

\theoremstyle{definition}

\theoremstyle{remark}

\renewcommand{\leq}{\leqslant}
\renewcommand{\le}{\leqslant}
\renewcommand{\geq}{\geqslant}
\renewcommand{\ge}{\geqslant}

\def\Chvatal{Chv\'atal}

\def\Kuhn{K{\"u}hn}
\def\Haggkvist{H{\"a}ggkvist}

\def\Posa{P\'osa}

\def\Szemeredi{Szemer\'edi}

\def\COMMENT#1{}

\newcommand{\eps}{\varepsilon}

\newcommand{\mc}[1]{\mathcal{#1}}
\newcommand{\mb}[1]{\mathbb{#1}}

\newcommand{\lra}{\leftrightarrow}
\newcommand{\sm}{\setminus}

\newcommand{\sub}{\subseteq}

\newcommand{\wt}{\widetilde}
\newcommand{\exit}{{\rm exit}}
\newcommand{\entry}{{\rm entry}}
\newcommand{\Exit}{{\rm Exit}}
\newcommand{\Entry}{{\rm Entry}}
\newcommand{\V}{\mathcal{V}}
\newcommand{\Z}{\mathcal{Z}}
\newcommand{\cE}{\mathcal{E}}
\newcommand{\C}{\mathcal{C}}
\newcommand{\match}{{\rm Match}}
\newcommand{\nd}{{\rm 2nd}}

\begin{document}

\vspace*{-1.5em}

\title{A semi-exact degree condition for Hamilton cycles in digraphs}
\author{Demetres Christofides, Peter Keevash, Daniela \Kuhn\ and Deryk Osthus}
\thanks {D.~Christofides was supported by the EPSRC, grant no.~EP/E02162X/1.
D.~K\"uhn was supported by the EPSRC, grant no.~EP/F008406/1.
D.~Osthus was supported by the EPSRC, grant no.~EP/E02162X/1 and~EP/F008406/1.}
\begin{abstract}
We show that for each $\beta > 0$, every digraph $G$ of sufficiently
large order $n$ whose outdegree and indegree sequences $d_1^+ \leq
\dots \leq d_n^+$ and $d_1^- \leq \dots \leq d_n^-$ satisfy
$d_i^+, d_i^- \geq \min{\{i + \beta n, n/2\}}$ is Hamiltonian.
In fact, we can weaken these assumptions to
\begin{itemize}
\item[(i)] $d_i^+ \geq \min{\{i + \beta n, n/2\}}$ or $d^-_{n - i - \beta n} \geq n-i$;
\item[(ii)] $d_i^- \geq \min{\{i + \beta n, n/2\}}$ or $d^+_{n - i - \beta n} \geq n-i$;
\end{itemize}
and still deduce that $G$ is Hamiltonian.
This provides an approximate version of a conjecture of
Nash-Williams from 1975 and improves a previous result
of \Kuhn, Osthus and Treglown.
\end{abstract}
\maketitle

\section{Introduction}

The decision problem of whether a graph contains a Hamilton cycle is
one of the most famous NP-complete problems, and so it
is unlikely that there exists a good characterization of all
Hamiltonian graphs. For this reason, it is natural to ask for
sufficient conditions which ensure Hamiltonicity. The most basic result of this kind is
Dirac's theorem~\cite{Dirac52}, which states that
every graph of order $n \geq 3$ and minimum degree at least $n/2$ is
Hamiltonian.

Dirac's theorem was followed by a series of results by various
authors giving even weaker conditions which still guarantee
Hamiltonicity. An appealing example is a theorem of
\Posa~\cite{Posa62} which implies%
   \COMMENT{If $n$ is odd then Posa is stronger than what we state (similar to NW-conjecture)}
that every graph of order $n \geq 3$ whose degree sequence $d_1 \leq
d_2 \leq \dots \leq d_n$ satisfies $d_i \geq i+1$ for all $i<n/2$ is
Hamiltonian. Finally, \Chvatal~\cite{Chvatal72} showed that if the
degree sequence of a graph~$G$ satisfies $d_i \geq i+1$ or  $d_{n-i}
\geq n-i$ whenever $i < n/2$, then $G$ is Hamiltonian. \Chvatal's
condition is best possible in the sense that for every sequence not
satisfying this condition, there is a non-Hamiltonian graph whose
degree sequence majorises the given sequence.

It is natural to seek analogues of these theorems for digraphs. For
basic terminology on digraphs, we refer the reader to the monograph
of Bang-Jensen and Gutin~\cite{Bang-Jensen&Gutin01}.
Ghouila-Houri~\cite{Ghouila-Houri60} proved that every digraph of
order $n$ and minimum indegree and outdegree at least $n/2$ is Hamiltonian,
thus providing such an analogue of Dirac's theorem for  digraphs.
Thomassen~\cite{Thomassen81} asked the corresponding question for
oriented graphs (digraphs with no $2$-cycles). One might expect that
a weaker minimum semidegree (i.e.\ indegree and outdegree) condition
would suffice in this case. \Haggkvist~\cite{Haggkvist93} gave a construction
showing that a minimum semidegree of $\frac{3n-4}{8}$ is necessary and
conjectured that it is also sufficient to guarantee a Hamilton cycle
in any oriented graph of order $n$. This conjecture was recently proved
in \cite{Keevash&Kuhn&Osthus+}, following an asymptotic solution in \cite{Kelly&Kuhn&Osthus+}.
In \cite{Christofides&Keevash&Kuhn&Osthus+}
we gave an NC algorithm for finding Hamilton cycles in digraphs
with a certain robust expansion property which captures several
previously known criteria for finding Hamilton cycles. These and other
results are also discussed in the recent survey~\cite{cyclesurvey}.

However, no digraph analogue of \Chvatal's theorem is known. For a digraph $G$ of
order $n$, let us write $d_1^+(G) \leq \dots \leq d_n^+(G)$ for its
outdegree sequence, and $d_1^-(G) \leq \dots \leq d_n^-(G)$ for its
indegree sequence. We will usually write $d_i^+$ and $d_i^-$ instead of $d_i^+(G)$ and $d_i^-(G)$ if this is unambiguous.

The following conjecture of
Nash-Williams~\cite{Nash-Williams75} would provide such an analogue.

\begin{conjecture}\label{Nash-Williams Conjecture - Chvatal}
Let $G$ be a strongly connected digraph of order $n \geq 3$ and
suppose that for all $i < n/2$
\begin{itemize}
\item[(i)] $d_i^+ \geq i+1$ or $d^-_{n-i} \geq n-i$;
\item[(ii)] $d_i^- \geq i+1$ or $d^+_{n-i} \geq n-i$.
\end{itemize}
Then $G$ contains a Hamilton cycle.
\end{conjecture}

Nash-Williams also highlighted the following conjectural analogue
of \Posa's theorem, which would follow from
Conjecture~\ref{Nash-Williams Conjecture - Chvatal}.%
    \COMMENT{Why is the last assumption necessary? It gives a weaker
condition if $n$ is odd (have $d_i^+,d_i^- \geq i+1$ only for all $i < (n-1)/2$
instead of for all $i<n/2$ as stated previously)}

\begin{conjecture}\label{Nash-Williams Conjecture - Posa}
Let $G$ be a digraph of order $n \geq 3$ such that
$d_i^+,d_i^- \geq i+1$ for all $i < (n-1)/2$ and
$d^+_{\lceil n/2 \rceil}, d^-_{\lceil n/2 \rceil} \geq \lceil n/2 \rceil$.
Then $G$ contains a Hamilton cycle.
\end{conjecture}

Note that in Conjecture~\ref{Nash-Williams Conjecture - Posa}
the degree condition implies that $G$ is strongly connected.
It is not even known whether the above conditions guarantee the existence
of a cycle though any given pair of vertices (see \cite{bt}).
We will prove the following semi-exact form of Conjecture~\ref{Nash-Williams Conjecture - Posa}.
It is `semi-exact' in the sense that for half of the vertex degrees, we obtain the conjectured bound,
whereas for the other half, we need an additional error term.

\begin{theorem}\label{CKKO - Approximate Posa}
For every $\beta > 0$ there exists an integer $n_0 = n_0(\beta)$
such that the following holds. Suppose $G$ is a digraph on $n \geq
n_0$ vertices such that $d_i^+,d_i^- \geq \min \{i + \beta n,n/2\}$
whenever $i < n/2$. Then $G$ contains a Hamilton cycle.
\end{theorem}

Recently, the following approximate version of
Conjecture~\ref{Nash-Williams Conjecture - Chvatal} for large
digraphs was proved by \Kuhn, Osthus and
Treglown~\cite{Kuhn&Osthus&Treglown+}.

\begin{theorem}\label{Kuhn-Osthus-Treglown}
For every $\beta > 0$ there exists an integer $n_0 = n_0(\beta)$
such that the following holds.
Suppose $G$ is a digraph on $n \geq n_0$ vertices such that%
   \COMMENT{We'll remove the $i<n/2$ condition later.}
for all $i < n/2$
\begin{itemize}
\item[(i)] $d_i^+ \geq i + \beta n$ or $d^-_{n-i - \beta n} \geq n-i$;
\item[(ii)] $d_i^- \geq i + \beta n$ or $d^+_{n-i - \beta n} \geq n-i$.
\end{itemize}
Then $G$ contains a Hamilton cycle.
\end{theorem}

We will extend this to the following theorem, which implies
Theorem~\ref{CKKO - Approximate Posa}.

\begin{theorem}\label{CKKO - Approximate Chvatal}
For every $\beta > 0$ there exists an integer $n_0 = n_0(\beta)$
such that the following holds. Suppose $G$ is a digraph on $n \geq
n_0$ vertices such that for all $i < n/2$
\begin{itemize}
\item[(i)] $d_i^+ \geq \min{ \{i + \beta n,n/2\} }$ or $d^-_{n-i - \beta n} \geq n-i$;
\item[(ii)] $d_i^- \geq \min{ \{i + \beta n, n/2 \} }$ or $d^+_{n-i - \beta n} \geq n-i$.
\end{itemize}
Then $G$ contains a Hamilton cycle.
\end{theorem}

(For the purposes of our arguments it turns out that there is
no significant difference in the use of the assumptions,
so for simplicity the reader could just read our proof as it applies to
Theorem~\ref{CKKO - Approximate Posa}.)

The improvement in the degree condition may at first appear minor,
so we should stress that capping the degrees at $n/2$
makes the problem substantially more difficult,
and we need to develop several new techniques in our solution.
This point cannot be fully explained until we have given several
definitions, but for the expert reader we make the following comment.
Speaking very roughly, the general idea used in
\cite{Kelly&Kuhn&Osthus+,Keevash&Kuhn&Osthus+,Christofides&Keevash&Kuhn&Osthus+}
is to apply Szemer\'edi's Regularity Lemma, cover most of the reduced digraph
by directed cycles, and then use the expansion property guaranteed by the degree conditions
on~$G$ to link these cycles up into
a Hamilton cycle while absorbing any exceptional vertices.
When the degrees are capped at $n/2$ two additional difficulties arise:
(i) the expansion property is no longer sufficient to link up the cycles,
and (ii) failure of a previously used technique for reducing the size
of the exceptional set. Our techniques for circumventing these difficulties
seem instructive and potentially useful in attacking Conjectures~\ref{Nash-Williams Conjecture - Chvatal}
and~\ref{Nash-Williams Conjecture - Posa}
in full generality.

Our paper is organized as follows. The next section contains some notation
and Section 3 some preliminary observations and examples.
Our proof will use the machinery of Szemer\'edi's Regularity Lemma, which we describe in Section~4.
(Unlike~\cite{Kelly&Kuhn&Osthus+,Keevash&Kuhn&Osthus+,Christofides&Keevash&Kuhn&Osthus+}, we do not require
the Blow-up lemma.)
Section~5 contains an overview of the proof in a special case that
illustrates the new methods that we introduce in this paper.
The cycle covering result is proved in Section~6 and
the proof of the special case completed in Section~7.
In Section~8 we describe the structures that arise in the general case.
We establish some bounds for these structures in Section~9.
Our main theorem is proved in Section~10.
The final section contains a concluding remark.

\section{Notation}

Given two vertices $x$ and $y$ of a digraph $G$, we write $xy$ for
the edge directed from $x$ to~$y$. The {\em order} $|G|$ of $G$ is
the number of its vertices. We write $N^+_G(x)$ and $N^-_G(x)$ for
the outneighbourhood and inneighbourhood of $x$ and $d^+_G(x)$ and
$d^-_G(x)$ for its outdegree and indegree.
The {\em degree} of $x$ is $d_G(x)=d^+_G(x)+d^-_G(x)$.
We usually drop the subscript $G$ if this is unambiguous.
The {\em minimum degree} and {\em maximum
degree} of $G$ are defined to be $\delta(G) = \min{\{d(x):x \in
V(G)\}}$ and $\Delta(G) = \max{\{d(x):x \in V(G)\}}$ respectively.
We define the {\em minimum indegree} $\delta^-(G)$ and
{\em minimum outdegree} $\delta^+(G)$ similarly.
The {\em minimum semidegree} is $\delta^0(G) = \min\{\delta^+(G),\delta^-(G)\}$.
Given $S \sub V(G)$ we write $d^+_S(x) = |N^+(x) \cap S|$
for the outdegree of $x$ in the set $S$.
We define $d^-_S(x)$ and $d_S(x)$ similarly.
Given a set $A$ of vertices of $G$, we write $N^+_G(A)$ for the set of all
outneighbours of vertices of $A$, i.e.~for the union of $N^+_G(x)$
over all $x \in A$. We define $N^-_G(A)$ analogously.

Given vertex sets $A$ and $B$ in a graph or digraph $G$, we write
$E_G(A,B)$ for the set of all edges
$ab$ with $a \in A$ and $b \in B$ and put $e_G(A,B) = |E_G(A,B)|$.
As usual, we drop the subscripts when this is unambiguous.
If $A\cap B=\emptyset$ we write
$(A,B)_G$ for the bipartite subgraph of $G$ with vertex
classes $A$ and $B$ whose set of edges is $E_G(A,B)$.
The restriction $G[A]$ of $G$ to $A$ is the digraph with vertex set $A$
and edge set all those edges of $G$ with both endpoints in $A$.
We also write $G \sm A$ for the digraph obtained by deleting $A$
and all edges incident to it, i.e.\ $G \sm A = G[V(G) \sm A]$.

Cycles and paths will always be understood as directed cycles and directed
paths, even if this is not explicitly stated.
Given two vertices $x$ and $y$ on a directed cycle $C$ we write
$xCy$ for the subpath of $C$ from $x$ to $y$. Similarly, given two
vertices $x$ and $y$ on a directed path $P$ such that $x$ precedes
$y$, we write $xPy$ for the subpath of $P$ from $x$ to $y$. A {\em
walk of length $\ell$} in a digraph $G$ is a sequence
$v_0,v_1,\ldots,v_{\ell}$ of vertices of $G$ such that $v_{i}v_{i+1}
\in E(G)$ for all $0 \leq i \leq \ell - 1$. The walk is {\em closed}
if $v_0 = v_{\ell}$. A {\em 1-factor} of $G$ is a collection of
disjoint cycles which cover all vertices of $G$. Given a 1-factor
$F$ of $G$ and a vertex $x$ of $G$, we write $x^+_F$ and $x^-_F$ for
the successor and predecessor of $x$ on the cycle in $F$ containing
$x$. We usually drop the subscript $F$ if this is unambiguous.
We say that $x$ and $y$ are at {\em distance} $d$ on $F$ if they
belong to the same directed cycle $C$ in $F$ and the distance from~$x$ to~$y$
or from $y$ to $x$ on~$C$ is~$d$. Note in particular that with this definition,
$x$ and $y$ could be at distance $d$ and $d'$ on $F$ with $d \neq d'$.

A digraph $G$ is {\em strongly connected} if for any ordered pair of
vertices $(x,y)$ there is a directed walk from $x$ to $y$. A {\em
separator} of $G$ is a set $S$ of vertices such that $G \sm S$ is
not strongly connected. We say $G$ is {\em strongly $k$-connected}
if $|G|>k$ and if it has no separator of size less than $k$. By
Menger's theorem, this is equivalent to the property that for any
ordered pair of vertices $(x,y)$ there are $k$ internally disjoint
paths from $x$ to $y$.

We write $a = b \pm c$ to mean that the real numbers $a,b,c$ satisfy $|a-b| \le c$.
We sometimes also write an expression such as $d^\pm(x) \ge t$ to mean
$d^+(x) \ge t$ and $d^-(x) \ge t$. The use of the $\pm$ sign will be clear from
the context.

To avoid unnecessarily complicated calculations we will sometimes
omit floor and ceiling signs and treat large numbers as if they were integers.

\section{Preliminaries}\label{prelims}

In this section we record some simple consequences of our degree assumptions
and describe the examples showing that Conjectures~\ref{Nash-Williams Conjecture - Chvatal}
and~\ref{Nash-Williams Conjecture - Posa} would be best possible.
We also recall two results on graph matchings and a standard large
deviation inequality (the Chernoff bound).

Our degree assumptions are that for all $i < n/2$ we have
\begin{itemize}
\item[(i)] $d_i^+ \geq \min{ \{i + \beta n,n/2\} }$ or $d^-_{n-i - \beta n} \geq n-i$;
\item[(ii)] $d_i^- \geq \min{ \{i + \beta n, n/2 \} }$ or $d^+_{n-i - \beta n} \geq n-i$.
\end{itemize}

We claim that $\delta^+(G)=d^+_1 \ge \beta n$. For if this were false our assumptions
would give $d^-_{n-1-\beta n} \ge n-1$, i.e.\ $G$ contains at least $\beta n + 1$ vertices
of indegree $n-1$. But a vertex of indegree $n-1$ is an outneighbour of all other vertices,
so this also implies that $\delta^+(G) \ge \beta n$. Similarly we have $\delta^-(G) \ge \beta n$.

To avoid complications with boundary cases it will be convenient to drop the condition $i<n/2$.
We note that this does not change our assumptions.
For if $n/2 \le i < n - \beta n$ we can apply our assumption (i) to $i'=n-i-\beta n$
and get $d_{i'}^+ \geq \min{ \{i' + \beta n,n/2\} }$ or $d^-_{n-i' - \beta n} \geq n-i'$,
i.e.\ $d_{n-i-\beta n}^+ \ge n-i$ or $d^-_i \ge i+\beta n$, which implies assumption (ii) for $i$.
Similarly assumption (ii) for $i'$ implies assumption (i) for $i$.
The assumptions do not make sense for $i \ge n - \beta n$, but if we consider any  statement about $d^\pm_j$ with $j \notin [1,n]$ as being vacuous (i.e.\ always true), then we do not have to impose any conditions when $i \ge n - \beta n$.

For an extremal example for Conjectures~\ref{Nash-Williams Conjecture - Chvatal}
and~\ref{Nash-Williams Conjecture - Posa}, consider a digraph $G$ on $n$ vertices constructed as follows.
The vertex set is partitioned as $I \cup K$ with $|I|=k<n/2$ and $|K|=n-k$.
We make $I$ independent and $K$ complete. Then we pick a set $X$ of $k$
vertices of $K$ and add all possible edges in both directions between $I$ and $X$.
This gives a strongly connected non-Hamiltonian digraph~$G$ in which both
the indegree and outdegree sequence are
$$\underbrace{k,\ldots,k}_{k \text{ times}},
\underbrace{n-1-k,\ldots,n-1-k}_{n-2k \text{ times}},
\underbrace{n-1,\ldots,n-1}_{k \text{ times}}.$$
$G$ fails conditions~(i) and~(ii) in Conjecture~\ref{Nash-Williams Conjecture - Chvatal} for $i=k$
and also one of the conditions in Conjecture~\ref{Nash-Williams Conjecture - Posa}.%
    \COMMENT{If $k\neq (n-1)/2$ then it violates the Posa condition in the $k$th
position. If $k= (n-1)/2$ then
$d^+_{\lceil n/2 \rceil}=d^-_{\lceil n/2 \rceil}=\lfloor n/2 \rfloor$}
In fact, a more complicated example is given in \cite{Kuhn&Osthus&Treglown+}
where only one condition in Conjecture~\ref{Nash-Williams Conjecture - Chvatal} fails. So, if true,
Conjecture~\ref{Nash-Williams Conjecture - Chvatal} would be best possible
in the same sense as Chv\'atal's theorem.%

A {\em matching} in a graph or digraph $G$ is a set of pairwise disjoint edges.
A {\em cover} is a set $C$ of vertices such that every edge of $G$ is incident
to at least one vertex in $C$. For bipartite graphs these concepts are related
by the following classical result of K\"onig.

\begin{prop}\label{konig}
In any bipartite graph, a maximum matching and a minimum cover have equal size.
\end{prop}

The following result, known as the `defect Hall theorem', may be easily deduced
from Proposition~\ref{konig}, using the observation that if $C$ is a cover
then $N(A \sm C) \sub C \cap B$.

\begin{prop}\label{dhall}
Suppose $G$ is a bipartite graph with vertex classes $A$ and $B$ and there is some number
$D$ such that for any $S \sub A$ we have $|N(S)| \ge |S|-D$.
Then $G$ contains a matching of size at least $|A|-D$.
\end{prop}

We will also need the following well-known fact.\COMMENT{This was
Proposition 16 in the previous version. Check also that you are
happy with the paragraph before Lemma 17.}

\begin{prop} \label{1-factor}
Suppose that $J$ is a digraph such that $|N^+(S)| \ge |S|$ for every
$S \sub V(J)$. Then $J$ has a $1$-factor.
\end{prop}

\begin{proof}
The result follows immediately by applying Proposition~\ref{dhall}
(with $D=0$) to the following bipartite graph~$\Gamma$: both vertex
classes $A,B$ of~$\Gamma$ are copies of the vertex set of the
original digraph $J$ and we connect a vertex $a \in A$ to $b \in B$
in~$\Gamma$ if there is a directed edge from $a$ to $b$ in~$J$. A
perfect matching in~$\Gamma$ corresponds to a 1-factor in~$J$.
\end{proof}

We conclude by recording the Chernoff bounds for binomial and hypergeometric
distributions (see e.g.~\cite[Corollary 2.3 and Theorem 2.10]{Janson&Luczak&Rucinski00}).
Recall that the binomial random variable with parameters $(n,p)$ is the sum
of $n$ independent Bernoulli variables, each taking value $1$ with probability $p$
or $0$ with probability $1-p$.
The hypergeometric random variable $X$ with parameters $(n,m,k)$ is
defined as follows. We let $N$ be a set of size $n$, fix $S \subset N$ of size
$|S|=m$, pick a uniformly random $T \subset N$ of size $|T|=k$,
then define $X=|T \cap S|$. Note that $\mb{E}X = km/n$.

\begin{prop}\label{chernoff}
Suppose $X$ has binomial or hypergeometric distribution and $0<a<3/2$. Then
$\mb{P}(|X - \mb{E}X| \ge a\mb{E}X) \le 2 e^{-\frac{a^2}{3}\mb{E}X}$.
\end{prop}

\section{Regularity}

The proof of Theorem~\ref{CKKO - Approximate Chvatal} will use the
directed version of \Szemeredi's Regularity Lemma.
In this section, we state a digraph form of this lemma
and establish some additional useful properties.
For surveys on the Regularity Lemma and its applications
we refer the reader to~\cite{Komlos&Simonovits96,Komlos99,Kuhn&Osthus09}.

\subsection{The Regularity Lemma}

The {\em density} of a bipartite graph $G = (A,B)$ with vertex
classes $A$ and $B$ is defined to be $d_G(A,B) =
\frac{e_G(A,B)}{|A||B|}$. We often write $d(A,B)$ if this is
unambiguous. Given $\eps > 0$, we say that $G$ is $\eps$-{\em
regular} if for all subsets $X \sub A$ and $Y \sub B$ with
$|X| \geq \eps|A|$ and $|Y| \geq \eps |B|$ we have that $|d(X,Y) -
d(A,B)| < \eps$. Given $d \in [0,1]$, we say that $G$ is
$(\eps,d)$-{\em regular} if it is $\eps$-regular of density at least
$d$. We also say that $G$ is $(\eps,d)$-{\em super-regular} if it is
$\eps$-regular and furthermore $d_G(a) \geq d|B|$ for all $a \in A$
and $d_G(b) \geq d|A|$ for all $b \in B$.

Given a digraph $G$, and disjoint subsets $A,B$ of $V(G)$, we say
that the ordered pair $(A,B)_G$ is $\eps$-regular, if the corresponding
undirected bipartite graph induced by the edges of $G$ from $A$ to
$B$ is $\eps$-regular. We use a similar convention for
super-regularity. The Diregularity Lemma is a version of the
Regularity Lemma for digraphs due to Alon and
Shapira~\cite{Alon&Shapira04}. We will use the degree form of the
Diregularity Lemma, which can be easily derived from the standard
version, in exactly the same manner as the undirected degree form.
(See e.g.~\cite{Kuhn&Osthus09} for a sketch proof.)

\begin{lemma}[Diregularity Lemma; Degree form]\label{Diregularity}
For every $\eps \in (0,1)$ and each positive integer $M'$, there
are positive integers $M$ and $n_0$ such that if $G$ is a digraph on
$n \geq n_0$ vertices, $d \in [0,1]$ is any real number, then
there is a partition of the vertices of $G$ into
$V_0,V_1,\ldots,V_k$ and a spanning subdigraph $G'$ of $G$ with the
following properties:
\begin{itemize}
\item $M' \leq k \leq M$;
\item $|V_0| \leq \eps n, |V_1| = \dots = |V_k| =: m$ and $G'[V_i]$
is empty for all $0 \leq i \leq k$;
\item $d^+_{G'}(x) > d^+_{G}(x) - (d + \eps)n$
and $d^-_{G'}(x) > d^-_{G}(x) - (d + \eps)n$ for all $x \in V(G)$;
\item all pairs $(V_i,V_j)_{G'}$ with $1 \leq i,j \leq k$ and $i\neq j$ are
$\eps$-regular with density either 0 or at least $d$.
\end{itemize}
\end{lemma}
Note that we do not require the densities of $(V_i,V_j)_{G'}$ and
$(V_j,V_i)_{G'}$ to be the same.
We call $V_1,\ldots,V_k$ the {\em clusters} of the partition, $V_0$
the {\em exceptional set} and the vertices of $G$ in $V_0$ the {\em
exceptional vertices}. The {\em reduced digraph} $R=R_{G'}$ of $G$ with
parameters $\eps,d,M'$ (with respect to the above partition) is the
digraph whose vertices are the clusters $V_1,\ldots,V_k$ and in
which $V_iV_j$ is an edge precisely when $(V_i,V_j)_{G'}$ has
density at least $d$.%

In various stages of our proof of Theorem~\ref{CKKO - Approximate
Chvatal}, we will want to make some pairs of clusters super-regular,
while retaining the regularity of all other pairs. This can be
achieved by the following folklore lemma. Here and later on we write $0<a_1 \ll a_2$ to mean that
we can choose the constants $a_2$ and $a_1$ from right to left. More
precisely, there is an increasing function $f$ such that, given
$a_2$, whenever we choose some $a_1 \leq f(a_2)$ all calculations in the proof
of Lemma~\ref{Super-regular} are valid. Hierarchies with more constants are
to be understood in a similar way.

\begin{lemma}\label{Super-regular}
Let $0<\eps\ll d,1/\Delta$ and let $R$ be a reduced digraph of $G$ as
given by Lemma~\ref{Diregularity}. Let $H$ be a subdigraph of $R$ of
maximum degree $\Delta$. Then, we can move exactly $\Delta \eps m$
vertices from each cluster $V_i$ into $V_0$ such that each pair of
clusters corresponding to an edge of $H$ becomes $(2 \eps ,
\frac{d}{2})$-super-regular, while each pair of clusters
corresponding to an edge of $R$ becomes $(2\eps,d-\eps)$-regular.
\end{lemma}

\begin{proof}
For each cluster $V \in V(R)$, let
\[ A(V) = \left\{x \in V:
\begin{matrix}
|N_{G'}^+(x) \cap W| < (d - \eps)m \text{ for some out-neighbour $W$ of $V$
in
$H$} \\
\text{ or } |N_{G'}^-(x) \cap W| < (d - \eps)m \text{ for some in-neighbour
$W$ of $V$ in $H$}
\end{matrix}
\right\}. \]
The definition of regularity implies that $|A(V)| \leq \Delta \eps m$.
Remove from each cluster $V$ a set of size exactly $\Delta \eps m$
containing $A(V)$. Since $\Delta \eps \leq \frac{1}{2}$, it follows
easily that all pairs corresponding to edges of $R$ become
$(2\eps,d-\eps)$-regular. Moreover, the minimum
degree of each pair corresponding to an edge of $H$ is at least $(d
- (\Delta + 1)\eps)m \geq \frac{d}{2}m$, as required.
\end{proof}

Next we note the easy fact that regular pairs have nearly perfect
matchings and super-regular pairs have perfect matchings.

\begin{lemma}\label{regmatch}
Suppose $\eps > 0$ and $G=(A,B)$ is an $(\eps,2\eps)$-regular pair with $|A|=|B|=n$.
Then $G$ contains a matching of size at least $(1-\eps)n$. Furthermore, if
$G$ is $(\eps,2\eps)$-super-regular then $G$ has a perfect matching.
\end{lemma}

\begin{proof}
For the first statement we verify the conditions of
the defect Hall theorem (Proposition~\ref{dhall}) with $D=\eps n$.
We need to show that $|N(S)| \ge |S|-D$ for $S \sub A$.
We can assume that $|S| \ge D=\eps n$. Then by $\eps$-regularity, all but
at most $\eps n$ vertices in $B$ have at least $\eps|S| > 0$ neighbours
in $S$. Therefore $|N(S)| \ge (1-\eps)n \ge |S|-\eps n$, as required.
For the second statement we need to show that $|N(S)| \ge |S|$ for $S \sub A$.
For any $x \in S$ we have $d(x) \ge 2\eps n$ by super-regularity,
so we can assume that $|S| \ge 2\eps n$. Then as before we have
$|N(S)| \ge (1-\eps)n$, so we can assume that $|S|>(1-\eps)n$.
But we also have $d(y) \ge 2\eps n$ for any $y \in B$, so
$N(y) \cap S \ne \emptyset$, i.e. $N(S)=B$ and $|N(S)|=n \ge |S|$.
\end{proof}

We will also need the following regularity criterion for finding a
Hamilton cycle in a non-bipartite digraph. We say that a general
digraph $G$ on $n$ vertices is {\em $\eps$-regular} of density $d$
if $\frac{e_G(X,Y)}{|X||Y|} = d \pm \eps$ for all (not necessarily
disjoint) subsets $X,Y$ of $V(G)$ of size at least $\eps n$, and
$(\eps,d)$-{\em super-regular} if it is $\eps$-regular and
$\delta^\pm(G) \ge dn$.

\begin{lemma} \label{fk}
Suppose $0<\eps \ll d \ll 1$, $n$ is sufficiently large and $G$ is
an $(\eps,d)$-super-regular digraph on $n$ vertices. Then $G$ is
Hamiltonian.
\end{lemma}

In fact, Frieze and Krivelevich \cite[Theorem 4]{FK} proved that an $(\eps,d)$-super-regular
digraph has $(d-4\eps^{1/2})n$ edge-disjoint Hamilton cycles, which is a substantial
strengthening of Lemma~\ref{fk}. Lemma~\ref{fk} can also be deduced from Lemma 10 in \cite{Keevash&Kuhn&Osthus+}.

Next we need a construction that we will use to preserve super-regularity of a pair
when certain specified vertices are excluded.

\begin{lemma}\label{tw}
Suppose $0<\eps \ll d \ll 1$, $G=(A,B)$ is an $(\eps,d)$-super-regular pair
with $|A|=|B|=n$ sufficiently large and $X \sub A$ with $|X|\le n/3$.
Then there is a set $Y \sub B$ with $|Y|=|X|$ such that
$(A \sm X, B \sm Y)_G$ is $(2\eps,d/2)$-super-regular.
\end{lemma}

\begin{proof}
If $|X| \le 2\eps n$ then we choose $Y$ arbitrarily with $|Y|=|X|$.
Next suppose that $|X| > 2\eps n$. We let $B_1$ be the set of
vertices in $B$ that have less than $\frac{1}{2}d|A \sm X|$
neighbours in $A \sm X$. Then $|B_1| \le \eps n$ by
$\eps$-regularity of $G$. Consider choosing $B_2 \sub B \sm B_1$ of
size $|X|-|B_1|$ uniformly at random. For any $x$ in $A$ its degree
in $B_2$ is $d_{B_2}(x)=|N_G(x) \cap B_2|$, which has hypergeometric
distribution with parameters $(|B \sm B_1|,d_{B \sm B_1}(x),|B_2|)$.
Super-regularity gives $\mb{E}[d_{B_2}(x)]= d_{B \sm B_1}(x)|B_2|/|B
\sm B_1| > \eps dn/2$, and the Chernoff bound (Proposition~\ref{chernoff})
applied with $a=n^{2/3}/\mb{E}d_{B_2}(x)>n^{-1/3}$ gives
$\mb{P}(|d_{B_2}(x)-\mb{E}d_{B_2}(x)| > n^{2/3})
< 2e^{-an^{2/3}/3} < 2e^{-n^{1/3}/3}$.%
\COMMENT{** recalculated}
By a union bound, there is some choice of $B_2$ so that every $x$ in
$A$ has $d_{B_2}(x) = d_{B \sm B_1}(x)|B_2|/|B \sm B_1| \pm n^{2/3}
< 0.4d_B(x)$ (say). Let $Y=B_1 \cup B_2$. Then $(A \sm X, B \sm
Y)_G$ is $2\eps$-regular, by $\eps$-regularity of $G$. Furthermore,
every $y \in B \sm Y$ has $d_{A \sm X}(y) \ge \frac{1}{2}d|A \sm X|$
by definition of $B_1$, and every $x$ in $A \sm X \sub A$ has $d_{B
\sm Y}(x) \ge d_B(x) - |B_1| - d_{B_2}(x) \ge \frac{1}{2}d|B \sm
Y|$.
\end{proof}

Finally, given an $(\eps,d)$-super-regular pair $G=(A,B)$, we will often need
to isolate a small subpair that maintains super-regularity in any
subpair that contains it. For $A^* \sub A$ and $B^* \sub B$
we say that $(A^*,B^*)$ is an \emph{$(\eps^*,d^*)$-ideal for $(A,B)$}
if for any $A^* \sub A' \sub A$ and $B^* \sub B' \sub B$
the pair $(A',B')$ is $(\eps^*,d^*)$-super-regular.
The following lemma shows that ideals exist, and moreover randomly chosen
sets $A^*$ and $B^*$ form an ideal with high probability.

\begin{lemma} \label{ideal}
Suppose $0<\eps \ll \theta,d < 1/2$, $n$ is sufficiently large
and $G=(A,B)$ is $(\eps,d)$-super-regular with $n/2 \le |A|,|B| \le n$.
Let $A^* \sub A$ and $B^* \sub B$ be independent uniformly random
subsets of size $\theta n$. Then with high probability
$(A^*,B^*)$ is an $(\eps/\theta,\theta d/4)$-ideal for $(A,B)$.
\end{lemma}

\begin{proof}
First we note that $\eps$-regularity of $G$ implies that $(A',B')$
is $\eps/\theta$-regular for any $A' \sub A$ and $B' \sub B$ with
$|A'|,|B'|\ge \theta n$. For each $x\in A$, the degree of $x$ in
$B^*$ is $d_{B^*}(x)=|N_G(x) \cap B^*|$, which has hypergeometric
distribution with parameters $(|B|,d_G(x),\theta n)$ and expectation
$\mb{E}[d_{B^*}(x)] \ge \theta d_G(x)$. By super-regularity we have
$d_G(x) \ge d n/2$, so by the Chernoff bound
(Proposition~\ref{chernoff}) applied with $a=\theta
dn/4\mb{E}[d_{B^*}(x)]\ge d/4$, we have $\mb{P}(d_{B^*}(x) < \theta
dn/4) < 2e^{-\theta d^2n/48}$. By a union bound, there is some
choice of $B^*$ so that every $x \in A$ has at least $\theta dn/4$
neighbours in $B^*$, and so at least $(\theta d/4)|B'|$ neighbours
in $B'$ for any $B^* \sub B' \sub B$. Arguing similarly for $A^*$
gives the result.
\end{proof}


\section{Overview of the proof} \label{techniques}

We will first prove a special case of Theorem~\ref{CKKO - Approximate Chvatal}.
Although it would be possible to give a single argument that covers all cases,
we believe it is instructive to understand the methods
in a simplified setting before introducing additional complications.
This section gives an overview of our techniques.%
   \COMMENT{If H is highly connected then consider a random partition into R and L.
Then with high probability (since we know many short paths survive),
both H[L] and H[R] are still highly connected}
We begin by defining additional constants such that%
  \COMMENT{Some of these parameters are not needed for the highly connected case.}
$$\frac{1}{n_0} \ll \eps \ll d \ll \gamma \ll d' \ll \eta \ll \eta' \ll \beta \leq 1.$$
Note that this hierarchy of parameters will be used throughout the paper.
By applying the Diregularity Lemma to $G$ with parameters $\eps,d$
and $M' = 1/\eps$, we obtain a reduced digraph $R_{G'}$ on $k$
clusters of size $m$ and an exceptional set $V_0$. We will see that the
degree sequences of $R_{G'}$ inherit many of the properties of the
degree sequences of $G$. Then it will follow that $R_{G'}$
contains a union of cycles $F$ which covers all but at most
$O(d^{1/2}k)$ of the clusters of $R_{G'}$. We move the vertices of all
clusters not covered by $F$ into $V_0$. By moving some further
vertices into $V_0$ we can assume that all edges of $F$
correspond to super-regular pairs.

Let $R_{G'}^*$ be the digraph obtained
from $R_{G'}$ by adding the set $V_0$ of exceptional vertices and for
each $x \in V_0$ and each $V \in R_{G'}$ adding the edge $xV$ if $x$
has an outneighbour in $V$ and the edge $Vx$ if $x$ has an
inneighbour in $V$. We would like to find a closed walk $W$ in
$R_{G'}^*$ such that
\begin{itemize}
\item[(a)] For each cycle $C$ of $F$, $W$ visits every cluster of $C$
the same number of times, say $m_C$;
\item[(b)] We have $1 \le m_C \le m$, i.e.\
$W$ visits every cluster at least once but not too many times;
\item[(c)] $W$ visits every vertex of $V_0$ exactly once;
\item[(d)] For each $x_i \in V_0$ we can choose
an inneighbour $x_i^-$ in the cluster preceding $x_i$ on $W$ and
an outneighbour $x_i^+$ in the cluster following $x_i$ on $W$,
so that as $x_i$ ranges over $V_0$
all vertices $x_i^+$, $x_i^-$ are distinct.
\end{itemize}
If we could find such a walk $W$ then by properties (a) and (b)
we can arrange that $m_C=m$ for each cycle $C$ of $F$
by going round $C$ an extra $m-m_C$ times on one particular
visit of $W$ to $C$. Then we could apply properties (c) and (d) to choose
inneighbours and outneighbours for every vertex of $V_0$ such that
all choices are distinct. Finally, we could apply a powerful tool known
as the Blow-up Lemma (see \cite{Komlos&Sarkozy&Szemeredi97})
to find a Hamilton cycle $C_{Ham}$ in $G$ corresponding to $W$,
where $C_{Ham}$ has the property that whenever $W$ visits a vertex of $V_0$, $C_{Ham}$ visits the
same vertex, and whenever $W$ visits a cluster $V_i$ of $R_{G'}$, then
$C_{Ham}$ visits a vertex $x \in V_i$.
(We will not discuss the Blow-up Lemma further,
as in fact we will take a different approach that does not need it.)

To achieve property (a), we will build up $W$ from certain `shifted'
walks, each of them satisfying property (a).
Suppose $R$ is a digraph, $R'$ is a subdigraph of $R$,
$F$ is a $1$-factor in $R$ and $a,b$ are vertices.
A {\em shifted walk} (with respect to $R'$ and $F$) from $a$ to $b$
is a walk $W(a,b)$ of the form
\[ W(a,b) = X_1 C_1 X^-_1 X_2 C_2 X^-_2 \ldots X_t C_t X^-_t X_{t+1},\]
where $X_1=a$, $X_{t+1} = b$, $C_i$ is the  cycle of $F$ containing $X_i$,
and for each $1 \leq i \leq t$,
$X^-_i$ is the predecessor of $X_i$ on $C_i$
and the edge $X^-_i X_{i+1}$ belongs to $R'$.
We say that $W(a,b)$
\emph{traverses the cycles $C_1,\ldots,C_t$}.
Note that even if the cycles
$C_1,\ldots,C_t$ are not distinct we say that $W$ traverses $t$
cycles. Note also that, for every cycle $C$ of $F$, the walk $W(a,b)\sm b$
visits the vertices of $C$ an equal number of times.

Given a shifted walk $W = W(a,b)$ as above
we say that $W$ {\em uses}  $X$ if $X$ appears in the list
$\{X_1^-, \ldots , X_t,X_t^-,X_{t+1} \}$.%
\COMMENT{originally, we included $X_1$ in this list, but then $X_1$ may be used but not used $a=1$ times,
which is rather odd. Omitting $X_1$ doesn't matter, as we only consider concatenations of walks anyway.}
More generally, we say that $X$ is {\em used $s$ times} by $W$ if it appears $s$ times in the above list (counting multiplicities). Thus $W$ uses $2t$ clusters, counting multiplicities.
We say that $W$ \emph{internally uses} $X$ if
$X \in \{X_2,X_2^-, \ldots , X_t,X_t^- \}$
(i.e.\ we do not count the uses of $X_1^-$ or $X_{t+1}$).
We also refer to the uses of $X_2,\ldots,X_{t+1}$ as {\em entrance uses}
and $X_1^-,\ldots,X_t^-$ as {\em exit uses}.
If $X$ is used as both $X_i$ and $X_j$ for some $2 \le i < j \le t+1$
then we can obtain a shorter shifted walk from $a$ to $b$ by deleting the
segment of $W$ between $X_i$ and $X_j$ (retaining one of them).
Similarly, we can obtain a shorter shifted walk if $X$ is used
as both $X_i^-$ and $X_j^-$ for some $1 \le i < j \le t$.
Thus we can always choose shifted walks so that any cluster
is used at most once as an entrance and at most once as an exit,
and so is used at most twice in total.

We say that a cluster $V$ is \emph{entered $a$ times} by $W$
if $W$ contains $a$ edges whose final vertex is $V$ and which do not lie in~$F$
(where the edges of $W$ are counted with multiplicities).
We have a similar definition for \emph{exiting $V$ $a$ times}.

Next we define an auxiliary digraph $H$ that plays a crucial role in our argument.
Let $R_{G''}$ be the spanning subdigraph of $R_{G'}$ obtained by deleting all those edges
corresponding to a pair of clusters whose density is less than $d'$.
Let $F$ be the $1$-factor of $R_{G'}$ mentioned above.
The vertices of $H$ are the clusters of $R_{G''}$.
We have an edge from $a$ to $b$ in $H$ if there is a shifted walk with respect to $R_{G''}$ and $F$
from $a$ to $b$ which traverses exactly one cycle.
One can view $H$ as a `shifted version' of $R_{G''}$.

For now we will only consider the special case in which $H$ is highly connected.
Even then, the fact that the exceptional set $V_0$ can be much bigger than the cluster sizes
creates a difficulty in ensuring property (b), that $W$ does not visit a cluster too many times.
A natural attempt to overcome this difficulty is the technique
from~\cite{Christofides&Keevash&Kuhn&Osthus+}.
In that paper we split each cluster $V_i$ of $R_{G'}$ into two equal pieces
$V_i^1$ and $V_i^2$. If the splitting is done
at random, then with high probability, the super-regularity between
pairs of clusters corresponding to the edges of $F$ is
preserved. We then applied the Diregularity Lemma to the subdigraph of
$G$ induced by $V_0 \cup V_1^2 \cup \dots \cup V_{k}^2$ with
parameters $\eps_2,d_2$ and $M_2' = 1/\eps_2$ to obtain a reduced
graph $R_2$ and an exceptional set $V_0^2$. The advantage gained is that by choosing
$\eps_2 \ll d_2 \ll \eps$ the exceptional set $V_0^2$ becomes much smaller than
the original cluster sizes and there is no difficulty with property (b) above.
However, the catch is that in our present case the degrees are capped at $n/2$, and
in the course of constructing the union of cycles in $R_2$ we would have to enlarge
$V_0^2$ to such an extent that this approach breaks down.

Our solution is to replace condition (b) by the following property for $W$:
\begin{itemize}
\item[(b$'$)] $W$ visits every cluster of $R_{G'}$ at least once but does not {\em use}
any cluster of $R_{G'}$ too many times.
\end{itemize}

This condition can be guaranteed by the high connectivity property of $H$.
However, we now have to deal with the fact that $W$  may  `wind around' each cycle of $F$ too many times.
This will be addressed by a shortcutting technique, where for each cycle $C$
in $F$ we consider the required uses of $C$ en masse and reassign routes
so as not to overload any part of $C$. A side-effect of this procedure is that
we may obtain a union of cycles, rather than a single Hamilton cycle.
However, using a judicious choice of $W$ and a switching procedure for matchings,
we will be able to arrange that these shortcuts do produce a single Hamilton cycle.
In particular, this approach does not rely on the Blow-up Lemma.%


\section{Structure I: Covering the reduced digraph by cycles}

We start the proof by applying the Diregularity Lemma (Lemma~\ref{Diregularity})
to $G$ with parameters $\eps,d$ and $M' = 1/\eps$,
obtaining a reduced digraph $R_{G'}$ on $k$ clusters of size $m$ and an exceptional set $V_0$.
Initially we have $|V_0| \le \eps n$, although we will add vertices to $V_0$ during the argument.
Note also that $n = km +|V_0|$.

\subsection{Properties of  $R_{G'}$}

Our main aim in this section is to show that $R_{G'}$ contains an
almost $1$-factor $F$, more specifically, a disjoint union of
directed cycles covering all but at most $7d^{1/2}k$ vertices of
$R_{G'}$. To begin with, we show that the degree sequences of
$R_{G'}$ have similar properties to the degree sequences of $G$.

\begin{lemma}\label{Degree sequence of R_{G'}} $ $
\begin{itemize}
\item[(i)] $d_i^+(R_{G'}) \geq \frac{1}{m}d_{im}^+(G) - 2dk$;
\item[(ii)] $d_i^-(R_{G'}) \geq \frac{1}{m}d_{im}^-(G) - 2dk$;
\item[(iii)] $\delta^+(R_{G'}) \geq \frac{\beta}{2}k$;
\item[(iv)] $\delta^-(R_{G'}) \geq \frac{\beta}{2}k$;
\item[(v)] $d_i^+(R_{G'}) \geq \min{\left\{i + \frac{\beta}{2}k, \left(\frac{1}{2} -
2d \right) k\right\}}$ or $d^-_{\left(1 -
\frac{\beta}{2}\right)k - i}(R_{G'}) \geq k - i - 2dk$;
\item[(vi)] $d_i^-(R_{G'}) \geq \min{\left\{i + \frac{\beta}{2}k, \left(\frac{1}{2} -
2d \right) k\right\}}$ or $d^+_{\left(1 -
\frac{\beta}{2}\right)k - i}(R_{G'}) \geq k - i - 2dk$.
\end{itemize}
\end{lemma}

\begin{proof}
We will only prove parts (i),(iii) and (v). Parts (ii), (iv) and
(vi) can be obtained in exactly the same way as parts (i), (iii) and
(v) respectively, by interchanging $+$ and $-$ signs. Consider $i$
clusters with outdegrees at most $d_i^+(R_{G'})$ in
$R_{G'}$. These clusters contain $im$ vertices of
$G$, so must include a vertex~$x$ of outdegree at least
$d_{im}^+(G)$. Lemma~\ref{Diregularity} implies that the cluster~$V$
containing~$x$ satisfies
$$
d_G^+(x) \leq d_{G'}^+(x) + (d + \eps)n \le d_{R_{G'}}^+(V)m + (d + \eps)n +|V_0| \le
d_{R_{G'}}^+(V)m + \frac{3}{2}dn.
$$
Therefore
\[ d_i^+(R_{G'}) \ge d^+_{R_{G'}}(V)\geq \frac{1}{m}d_{im}^+(G) - \frac{3}{2}d\frac{n}{m}
\geq \frac{1}{m}d_{im}^+(G) - 2dk,\] which proves (i).
Next, (iii) follows from (i), since $\delta^+(G) \geq \beta n$.
To prove (v), suppose
that $d_i^+(R_{G'}) < \min{\left\{i + \frac{\beta}{2}k,
\left(\frac{1}{2} - 2d \right) k\right\}}$. It follows from (i) that
\[ d^+_{im}(G) < \min{\left\{ m(i + \beta k/2 + 2dk),
mk/2 \right\}} \leq \min{\left\{ im + \beta n , n/2 \right\}}.\]
Using our degree assumptions gives
$d^-_{n - im - \beta n}(G) \geq n - im$.
Then by (ii) we have
\begin{align*}
d^-_{\left(1 - \frac{\beta}{2} \right)k - i}(R_{G'})
\geq \frac{1}{m} d^-_{\left(1 - \frac{\beta}{2} \right)km -im}(G) - 2dk
\geq \frac{1}{m} d^-_{\left(1 - \beta \right)n - im}(G) - 2dk
\geq k - i - 2dk,
\end{align*}
as required.
\end{proof}

Unfortunately, $R_{G'}$ need not satisfy the hypothesis of
Proposition~\ref{1-factor}, so we cannot use it to deduce the
existence of a 1-factor in $R_{G'}$. The next lemma shows that a
problem can only occur for subsets $S$ of $V(R_{G'})$ of size close
to $k/2$.

\begin{lemma}\label{Outexpansion of R_{G'}}
Let $S$ be a subset of $V(R_{G'})$ such that either
$|S| \leq (1/2 - 2d)k$ or $|S| > (1/2 + 2d)k$.
Then $|N^+(S)|,|N^-(S)| \geq |S|$.
\end{lemma}

\begin{proof}
Suppose firstly that $|S| \leq (1/2 - 2d)k$ but $|N^+(S)| < |S|$. By
the minimum outdegree condition of $R_{G'}$ (Lemma~\ref{Degree
sequence of R_{G'}}~(iii)) we must have $|S| \geq \beta k/2$. Also
$d^+_{|S| - 2dk - 1} (R_{G'}) \leq d_{|S|}^+ (R_{G'}) < |S| \leq
(1/2 - 2d)k$, so Lemma~\ref{Degree sequence of R_{G'}}~(v) gives
$d^-_{(1 -\beta/2)k - |S| + 2dk + 1} (R_{G'})\geq k - |S| + 1$. Thus
there are at least $\beta k/2 + |S| - 2dk\geq |S|$ vertices of
indegree at least $k - |S| + 1$. Now if $x$ has indegree at least $k
- |S| + 1$ then $N^-(x)$ intersects $S$, so $x$ belongs to $N^+(S)$.
We deduce that $|N^+(S)| \geq |S|$. A similar argument shows that
$|N^-(S)| \geq |S|$ as well. Now suppose that $|S| > (1/2 + 2d)k$
but $|N^+(S)| < |S|$, and consider $T = V(R_{G'}) \sm N^+(S)$. Since
$N^-(T) \cap S = \emptyset$, we have $|N^-(T)| < |T|$, and so $|T| >
(1/2 - 2d)k$ by the first case. But now we can consider a subset
$T'$ of $T$ of size $|T'|=(1/2-2d)k$ to see that $|N^-(T)| \geq
|N^-(T')| \geq |T'| = (1/2 - 2d)k$, and so $|S| \leq (1/2 + 2d)k$, a
contradiction. The claim for $|N^-(S)|$ follows by a similar
argument.
\end{proof}

Applying Hall's theorem as in Proposition~\ref{1-factor},
one can use Lemma~\ref{Outexpansion of R_{G'}} to partition the vertex set
of $R_{G'}$ into a union of cycles and at most $4dk$ paths.
However, for our approach we need to find a disjoint union of cycles
covering almost all the vertices. The first step towards this goal will
be to arrange that for each path its initial vertex has large indegree
and its final vertex has large outdegree.
To prepare the ground, we show in the next lemma that if $R_{G'}$
does not have a $1$-factor, then it has many vertices of large
outdegree and many vertices of large indegree.

\begin{lemma}\label{No 1-factor in R_{G'} ==> many vertices of large degree}
If $R_{G'}$ does not have a $1$-factor, then it contains more than
$(1/2 + 2d)k$ vertices of outdegree at least $(1/2 - 2d)k$
and more than $(1/2 + 2d)k$ vertices of indegree at least $(1/2 - 2d)k$.
\end{lemma}

\begin{proof}
Since $R_{G'}$ does not have a $1$-factor, by Proposition~\ref{1-factor}
it contains a set $S$ with $|N^+(S)| < |S|$. Then by
Lemma~\ref{Degree sequence of R_{G'}} (i) we have
\[ |S| > |N^+(S)| \ge d^+_{|S|}(R_{G'}) \geq \frac{1}{m}d^+_{m|S|}(G) - 2dk,\]
and so
\[ d^+_{m|S| - \frac{\beta}{2}n}(G) \leq
m(|S| + 2dk) \leq m|S| + \frac{\beta}{2}n. \]
Moreover, $(1/2-2d)k <|S| \le (1/2+2d)k$ by
Lemma~\ref{Outexpansion of R_{G'}}.
So if it were also the case that
$d^+_{m|S| - \frac{\beta}{2}n}(G) < n/2=\min \{m|S|+ \beta n/2,n/2\}$,
then $d^-_{(1 -\beta/2)n - m|S|}(G) \geq (1 + \beta/2)n - m|S|$ and so by
Lemma~\ref{Degree sequence of R_{G'}} (ii) we would have
\[ d^-_{\left(1 - \frac{\beta}{4} \right)k - |S|}(R_{G'}) \geq \left(1
+ \frac{\beta}{2} \right)k - |S| - 2dk \geq k - |S| + 1. \]
Then $R_{G'}$ contains at least $\beta k/4 + |S|$ vertices of
indegree at least $k - |S| + 1$, and these must all belong to
$N^+(S)$, a contradiction. It follows that
$d^+_{m|S| - \frac{\beta}{2}n}(G) \geq n/2$.
So Lemma~\ref{Degree sequence of R_{G'}} (i) gives
\[ d^+_{|S| - \frac{\beta}{2}k}(R_{G'}) \geq d^+_{|S| -
\frac{\beta}{2}\frac{n}{m}}(R_{G'}) \geq \frac{n}{2m} - 2dk
\geq \left( \frac{1}{2} - 2d \right)k,\]
i.e.~$R_{G'}$ contains at least $(1 + \beta/2)k - |S|\ge (1/2+2d)k$ vertices of
outdegree at least $(1/2 - 2d)k$, which proves the first part of
the lemma. The second part can be proved in exactly the same way.
\end{proof}

Now we can show how to arrange the degree property for the paths.

\begin{lemma}\label{R_{G'} is a union of cycles and not too many paths}
The vertex set of $R_{G'}$ can be partitioned into a union of cycles
and at most $4dk$ paths such that the initial vertices of the
paths each have indegree at least $(1/2 - 2d)k$ and the final vertices
of the paths each have outdegree at least $(1/2 - 2d)k$.
\end{lemma}

\begin{proof}
We may assume that $R_{G'}$ does not have a $1$-factor and so the
consequences of Lemma~\ref{No 1-factor in R_{G'} ==> many vertices of
large degree} hold. We define an auxiliary digraph $R_{G'}'$ by
adding $4dk$ new vertices $v_1,v_2,\ldots,v_{4dk}$ to $R_{G'}$,
adding all possible edges between
these vertices (in both directions), adding all edges of the form
$vv_i$, where $1 \leq i \leq 4dk$ and $v$ is a vertex of $R_{G'}$
of outdegree at least $(1/2 - 2d)k$ and finally adding all
edges of the form $v_iv$ where $1 \leq i \leq 4dk$ and $v$ is a
vertex of $R_{G'}$ of indegree at least $(1/2 - 2d)k$.
Then any vertex that previously had indegree at least $(1/2-2d)k$
now has indegree at least $(1/2+2d)k$, and similarly for outdegree.
Also, Lemma~\ref{No 1-factor in R_{G'} ==> many vertices of
large degree} implies that every new vertex $v_i$ has
indegree and outdegree more than $(1/2 + 2d)k$.
We claim that $R_{G'}'$ has a $1$-factor. Having proved this, the result will follow
by removing $v_1,\ldots,v_{4dk}$ from the cycles in the 1-factor. To prove the claim, let us take
$S \sub V(R_{G'}')$. By Proposition~\ref{1-factor}
we need to show that $|N^+(S)| \geq |S|$. We consider cases according
to the size of $S$.
If $|S| \leq (1/2 - 2d)k$, then either $S \sub V(R_{G'})$, in
which case $|N^+(S)| \geq |S|$ by Lemma~\ref{Outexpansion of R_{G'}},
or $S$ contains some new vertex $v_i$, in which case
$|N^+(S)| \geq d^+(v_i) \geq (1/2 + 2d)k \geq |S|$.
Next suppose that $(1/2 - 2d)k < |S| \leq (1/2 + 2d)k$.
As before, if $S$ contains a new vertex $v_i$ we have
$|N^+(S)| \geq d^+(v_i) \geq (1/2 + 2d)k \geq |S|$,
so we can assume $S \sub V(R_{G'})$.
Now by Lemma~\ref{No 1-factor in R_{G'} ==> many vertices of large degree}
each new vertex $v_i$ has at least $(1/2+2d)k>k-|S|$ inneighbours in $V(R_{G'})$ and
so $v_i$ has an inneighbour in~$S$, i.e.\ $v_i \in N^+(S)$.
Also, $S$ has at least $(1/2 - 2d)k$ outneighbours in $R_{G'}$
by Lemma~\ref{Outexpansion of R_{G'}}, so in $R_{G'}'$ we have
$|N^+(S)| \geq 4dk + (1/2 - 2d)k \ge |S|$.
Finally suppose that $|S| > (1/2 + 2d)k$.
Let $T = V(R_{G'}') \sm N^+(S)$. Considering a subset $S'\sub S$ of size
$(1/2+2d)k$ shows that $|T|\le k-|N^+(S')|\le k-|S'|=(1/2-2d)k$.
However, $N^-(T)$ is disjoint from $S$, so
if $|N^+(S)| < |S|$ we have $|T| > |N^-(T)|$.
Now similar arguments to before give $|T| > (1/2 + 2d)k$, a contradiction.
\end{proof}

\subsection{The almost 1-factor}

We now come to the main result of this section.

\begin{lemma}\label{Existence of F}
$R_{G'}$ contains a disjoint union $F$ of cycles covering all but at most
$7d^{1/2}k$ of its vertices.
\end{lemma}

\begin{proof}
We implement the following algorithm.
At each stage, the vertex set of $R_{G'}$ will be partitioned into some
cycles and paths and a waste set $W$. In every path the initial vertex
will have indegree at least $(1/2 - 2d)k$ and the final vertex
will have outdegree at least $(1/2 - 2d)k$. One of the paths
will be designated as the `active path'.

In the initial step, we begin with the partition guaranteed by
Lemma~\ref{R_{G'} is a union of cycles and not too many paths}.
We have $W = \emptyset$ and choose an arbitrary path to be active.

In each iterative step we have some active path $P$.
Let $u$ be the initial vertex of $P$ and $v$ its final vertex.
Let $S$ be the sum of the numbers of vertices in all the paths. If at any
point $S \leq 5d^{1/2}k$, then we move the vertices of all these
paths into $W$ and stop. Otherwise we define
$\alpha = 5dk/S$ and for each path $P_r$, we let $\ell_r = \alpha|P_r|$.
Note that the parameters $S$, $\alpha$ and $\{\ell_r\}$ are
recalculated at each step.
By our assumption on $S$ we have $\alpha \leq d^{1/2}$.
Also $\sum_r \ell_r = \alpha S = 5dk$.

For each cycle $C=w_1\dots w_t w_1$ and $X \sub V(C)$ we write $X^+ = \{w_{i+1}: w_i \in X\}$
for the set of successors of vertices of $X$.
For each path $P_r = w_1 \ldots w_t$, $X \sub P_r$ and $1 \le s \le t$
we let $X^{+s} = \{w_j: \exists w_i \in X, i<j \le i+s \}$.
Also, for each path
$P_r=w_1 \ldots w_t$ which contains at least one outneighbour of $v$ we let $i^v_r \ge 0$
be minimal such that $w_{i^v_r+1} \in N^+(v)\cap P_r$. Similarly, for
each path $P_r=w_1 \ldots w_t$ which contains at least one inneighbour of $u$ we let
$i^u_r \ge 0$ be minimal such that $w_{t-i^u_r} \in N^-(u)\cap P_r$.
We claim that at least one of the following conditions holds:
\begin{itemize}
\item[(1)] There is a $w \in W$ such that $wu,vw \in E(R_{G'})$.
\item[(2)] There is a cycle $C = w_1 \ldots w_i w_{i+1} \ldots w_t w_1$
such that $w_iu,vw_{i+1} \in E(R_{G'})$.
\item[(3)] There is a path $P_r = w_1 \ldots w_t$ and $1 \leq i < j \leq t$
with $j-i \leq \ell_r+1$ such that $w_iu,vw_j \in E(R_{G'})$.
\item[(4)] There is a path $P_r = w_1 \ldots w_t$ with $i^u_r\le \ell_r$ or $i^v_r \le \ell_r$.
\end{itemize}
To see this, suppose to the contrary that all these conditions fail. Since (1) fails, then $(N^-(u) \cap W) \cap (N^+(v) \cap W) = \emptyset$ and so
\begin{itemize}
\item $|N^-(u) \cap W| + |N^+(v) \cap W| \le |W|$.
\end{itemize}
Since (2) fails, then for each cycle $C$ we have $(N^-(u) \cap C)^+ \cap (N^+(v) \cap C) = \emptyset$ and so
\begin{itemize}
\item $|N^-(u) \cap C| + |N^+(v) \cap C| \le |C|$ for each $C$.
\end{itemize} 
Since (4) fails then for each path $P_r$ we have $|N^-(u)\cap P_r|\le |P_r|-\ell_r$ and $|N^+(v)\cap P_r|\le |P_r|-\ell_r$. In particular
\begin{itemize}
\item for each path $P_r$, if $P_r$ does not meet both $N^-(u)$ and $N^+(v)$ then $|N^-(u) \cap P_r|+|N^+(v) \cap P_r| \le |P_r| - \ell_r$.
\end{itemize} 
On the other hand if a path $P_r$ meets both $N^-(u)$ and $N^+(v)$ then, since (3) fails we have $(N^-(u) \cap P_r)^{+(\ell_r+1)} \cap (N^+(v) \cap P_r) = \emptyset$. Moreover, since $i^u_r> \ell_r$ and since also (4) fails, we also have that $|(N^-(u) \cap P_r)^{+(\ell_r+1)}| \ge |N^-(u) \cap P_r|+\ell_r$. Altogether this gives that
\begin{itemize}
\item for each path $P_r$, if $P_r$ meets both $N^-(u)$ and $N^+(v)$ then $|N^-(u) \cap P_r|+|N^+(v) \cap P_r| \le |P_r| - \ell_r$.
\end{itemize} 
Summing these inequalities gives
\[d^-(u)+d^+(v) \le |W| + \sum_C |C| + \sum_r (|P_r| - \ell_r) = k - \sum_r \ell_r.\]
But we also have $\sum_r \ell_r = \alpha S = 5dk$
and $d^-(u), d^+(v) \ge (1/2-2d)k$ by the degree property of the paths.
This contradiction shows that at least one of the conditions (1)--(4) holds.

According to the above conditions we take one of the following actions.
\begin{itemize}
\item[(1)] Suppose there is a $w \in W$ such that $wu,vw \in E(R_{G'})$.
Then we replace the path $P$ by the cycle $C = wuPvw$, replace $W$ by
$W \sm \{w\}$, choose a new active path, and repeat.
\item[(2)] Suppose there is a cycle $C = w_1 \ldots w_i w_{i+1} \ldots w_t w_1$
such that $w_iu,vw_{i+1} \in E(R_{G'})$. Then we replace the path
$P$ and the cycle $C$ by the cycle $C' = w_1 \ldots w_iuPvw_{i+1}
\ldots w_tw_1$, choose a new active path, and repeat.
\item[(3)] Suppose there is a path $P_r = w_1 \ldots w_t$
and $1 \leq i < j \leq t$ with $j-i \le  \ell_r+1$ such that $w_iu,vw_j\in E(R_{G'})$.

(i) If $P_r\neq P$ then we replace the paths $P$ and $P_r$ with the path
$P_r' = w_1 \ldots w_i u P v w_j \ldots w_t$, replace $W$ with $W \cup
\{w_{i+1}, \ldots , w_{j-1} \}$, make $P_r'$ the new active path, and repeat.

(ii) If $P_r=P$ (so $w_1 = u$ and $w_t =v$) then we replace $P$ with the cycles
$C_u = uw_2 \ldots w_{i-1}w_iu$ and $C_v = vw_j \ldots w_{t-1}v$, replace $W$ with
$W \cup \{w_{i+1}, \ldots , w_{j-1} \}$, choose a new active path, and repeat.
\item[(4)] Suppose there is a path $P_r = w_1 \ldots w_t$ with
$i^u_r\le \ell_r$ or $i^v_r \le \ell_r$.

(i) If $P_r\neq P$ and $i^u_r\le \ell_r$ then we replace the paths $P$ and $P_r$ with the path
$P'_r = w_1\dots w_{t-i^u_r}uPv$, replace $W$ with
$W \cup \{w_{t-i^u_r+1},\ldots,w_t\}$, make $P'_r$ the new active path, and repeat.

(ii) If $P_r\neq P$ and $i^v_r \le \ell_r$ then we replace the paths $P$ and $P_r$ with the path
$P'_r = uPvw_{i^v_r+1}\dots w_t$, replace $W$ with
$W\cup \{w_1, \ldots, w_{i^v_r}\}$, make $P'_r$ the new active path, and repeat.

(iii) If $P_r=P$ (so $w_1 = u$ and $w_t =v$) and $i^u_r\le \ell_r$ then we replace $P$ with the cycle
$C = uPw_{t-i^u_r}u$, replace $W$ with $W \cup \{w_{t-i^u_r+1}, \dots, w_t\}$,
choose a new active path, and repeat.

(iv) If $P_r=P$ and $i^v_r\le \ell_r$ then we replace $P$ with the cycle
$C = vw_{i^v_r+1}Pv$, replace $W$ with $W \cup \{w_1, \dots, w_{i^v_r}\}$,
choose a new active path, and repeat.
\end{itemize}

At each step the number of paths is reduced by at least~1, so the
algorithm will terminate. It remains to show that $|W| \leq
7d^{1/2}k$. Recall that at every step we have $\ell_r = \alpha|P_r|
\leq d^{1/2}|P_r|$ for each path $P_r$. For every vertex $w$ added
to $W$ we charge its contribution to the path that $w$ {\em
initially} belonged to. To calculate the total contribution we break
it down by the above cases and by initial paths. Cases (1) and (2)
do not increase the size of $W$. In Case~3(i), every initial path
$P_r$ is merged with an active path $P$ at most once, and then its
remaining vertices stay in the active path until a new active path
is chosen, so this gives a contribution to~$W$ of at most $\ell_r
\leq d^{1/2}|P_r|$ from $P_r$. In Case~3(ii), the vertices of the
active path $P_r=P$ are contained in a union $\cup_{i\in I} V(P_i)$
of some subset of the initial paths (excluding some vertices already
moved into $W$). These paths collectively contribute at most $\alpha
|P| \le d^{1/2} \sum_{i \in I} |P_i|$, and each initial path is
merged at most once into such a path $P$. In Cases~4(i) and~4(ii),
as in Case 3(i), an initial path $P_r$ contributes at most
$\alpha|P_r|$. In Cases 4(iii) and 4(iv), as in Case 3(ii), the
vertices of the active path $P_r=P$ are contained in a union
$\cup_{i\in I} V(P_i)$ of some subset of the initial paths and
contribute at most $\alpha |P| \le d^{1/2} \sum_{i \in I} |P_i|$. So
each initial path contributes to $W$ at most twice: once when it is
merged into the active path (in Cases~3(i), 4(i) or 4(ii)) and once
when this active path is turned into one or two cycles (in Cases
3(ii), 4(iii) or 4(iv)). Therefore we get a total contribution from
the paths of at most $2d^{1/2}k$ to~$W$. Finally, there is another
contribution of at most $5d^{1/2}k$ if at any step we have $S \leq
5d^{1/2}k$. In total we have $|W| \le 7d^{1/2}k$.
\end{proof}

\subsection{Further properties of $F$ }

Now we have an almost $1$-factor $F$ in $R_{G'}$, i.e.\ a disjoint union of cycles
covering all but at most $7d^{1/2}k$ clusters of $R_{G'}$.
We move all vertices of these uncovered clusters into $V_0$,
which now has size at most $8d^{1/2}n$.
During the proof of Theorem~\ref{CKKO - Approximate Chvatal} it will be helpful to
arrange that each cycle of $F$ has length at least~4 (say) and moreover,
all pairs of clusters corresponding to edges of $F$ correspond to
super-regular pairs. (This assumption on the lengths is not actually necessary
but does make some of the arguments in the final section more transparent.)\COMMENT{Changed the length so that the $10\eps$-regular later on becomes correct.}

We will now show that we may assume this. Indeed, if $F$ contains
cycles of lengths less than 4, we arbitrarily partition each cluster
of $R_{G'}$ into~2 parts of equal size. (If the sizes of the
clusters are not divisible by 2, then before the partitioning we
move at most 1 vertex from each cluster into $V_0$ in order to
achieve this.) Consider the digraph $R'_{G'}$ whose vertices
correspond to the parts and where two vertices are joined by an edge
if the corresponding bipartite subdigraph of $G'$ is
$(2\eps,\frac{2d}{3})$-regular. It is easy to check that this
digraph contains the $2$-fold `blowup' of $R_{G'}$, i.e.~each
original vertex is replaced by an independent set of $2$ new
vertices and there is an edge from a new vertex $x$ to a new vertex
$y$ if there was such an edge between the original vertices. Each
cycle of length $\ell$ of $F$ induces an $2$-fold blowup of
$C_{\ell}$ in $R'_{G'}$, which contains a cycle of length $2\ell
\ge 4$. So $R'_{G'}$ contains a $1$-factor $F'$ all of whose
cycles have length at least~4. Note that the size of $V_0$ is now at
most $9d^{1/2}n$.

Secondly, we apply Lemma~\ref{Super-regular} to make the pairs of
clusters corresponding to edges of $F'$
$(4\eps,\frac{d}{3})$-super-regular by moving exactly $4\eps |V_i|$
vertices from each cluster $V_i$ into $V_0$ and thus increasing
the size of $V_0$ to at most $10d^{1/2}n$.
For convenience, having made these alterations,
we will still denote the reduced digraph by $R_{G'}$, the order of $R_{G'}$ by $k$,
its vertices (the clusters) by $V_1,\ldots,V_{k}$ and their sizes by $m$.
We also rename $F'$ as $F$. We sometimes refer to the cycles in~$F$ as \emph{$F$-cycles}.

\subsection{A modified reduced digraph}

Let $R_{G''}$ be the spanning subdigraph obtained from $R_{G'}$ by deleting all those edges
which correspond to pairs of density at most $d'$.
Recalling that $d \ll d'$, we note that the density of pairs corresponding to edges in $R_{G''}$
is much larger than the proportion $10\sqrt{d}$ of vertices lying in $V_0$.
The purpose of $R_{G'}$ was to construct $F$ so that this property would hold.
Now we have no further use for $R_{G'}$ and will work only with $R_{G''}$.
(Actually, we could use either $R_{G''}$ or $R_{G'}$ for the special case in the next section,
but we need to work with $R_{G''}$ in general.)

Let $G''$ be the digraph obtained from $G'$ obtained by deleting all edges belonging to
pairs $(X,Y)$ of clusters so that $(X,Y)_{G'}$ has density at most $d'$.
We say that a vertex $x \in X$ is \emph{typical} if
\begin{itemize}
\item $d^\pm_{G''}(x) \ge d^\pm_{G}(x) - 4d'n$;
\item there are at most $\sqrt{\eps}k$ clusters $Y$ such that $x$ does not have
$(1 \pm 1/2)d_{XY}m$ outneighbours in $Y$, where $d_{XY}$ denotes the density of the pair $(X,Y)_{G''}$.
The analogous statement also holds for the inneighbourhood of $x$.
\end{itemize}
\begin{lemma}\label{Degree sequence of R_{G'} - Modified}
By moving exactly $16\sqrt{\eps} m$ vertices from each cluster into $V_0$, we can arrange that each
vertex in each cluster of $R_{G''}$ is typical. We still denote the cluster sizes by $m$. Then we have
\begin{itemize}
\item[(i)] $d_i^+(R_{G''}) \geq \frac{1}{m}d_{im}^+(G) - 5d'k$;
\item[(ii)] $d_i^-(R_{G''}) \geq \frac{1}{m}d_{im}^-(G) - 5d'k$;
\item[(iii)] $\delta^+(R_{G''}) \geq \frac{\beta}{2}k$;
\item[(iv)] $\delta^-(R_{G''}) \geq \frac{\beta}{2}k$;
\item[(v)] $d_i^+(R_{G''}) \geq \min{\left\{i + \frac{\beta}{2}k, \left(\frac{1}{2} -
5d' \right) k\right\}}$
or $d^-_{\left(1 -\frac{\beta}{2}\right)k - i}(R_{G''}) \geq k - i - 5d'k$;
\item[(vi)] $d_i^-(R_{G''}) \geq \min{\left\{i + \frac{\beta}{2}k, \left(\frac{1}{2} -
5d' \right) k\right\}}$
or $d^+_{\left(1 -\frac{\beta}{2}\right)k - i}(R_{G''}) \geq k - i - 5d'k$.
\end{itemize}
\end{lemma}
\proof
Suppose that we are given clusters $X,Y$ such that $XY$ is an edge of $R_{G'}$.
Write $d_{XY}$ for the density of $(X,Y)_{G'}$.
We say that $x \in X$ is \emph{out-typical for $Y$} if (in $G'$) $x$ has  $(1 \pm 1/3)d_{XY}m$ outneighbours
in $Y$. Since the pair $(X,Y)_{G'}$ is $4\eps$-regular, it follows that at most $8 \eps m$
vertices of $X$ are not out-typical for $Y$.
Then on average, a vertex of $X$ is not out-typical for at most $8 \eps k$ clusters.
It follows that there are at most $8\sqrt{\eps}m$ vertices $x$ in $X$ for which
there are more than $\sqrt{\eps}k$ clusters $Y$ such that $x$ is not out-typical for $Y$.
Therefore we can remove a set of exactly $8\sqrt{\eps} m$ vertices from each cluster so that
all of the remaining vertices are out-typical for at least $(1-\sqrt{\eps})k$ clusters.
We proceed similarly for the inneighbourhood of each cluster.
Altogether, we have removed exactly $16 \sqrt{\eps} m$ vertices from each cluster.
These vertices are added to $V_0$, which now has size $|V_0| \le 11 \sqrt{d}n$.
Now consider some cluster $X$ and a vertex $x \in X$.
Since $x$ is out-typical for all but at most $\sqrt{\eps}k$ clusters,
it sends at most $\sqrt{\eps}k\cdot m + k\cdot 2d'm \le 3d'n$ edges into clusters $Y$
such that $(X,Y)_{G'}$ has density at most $d'$.
Then the following estimate shows that $x$ is typical:
\COMMENT{Need to account for more vertices added to $V_0$ since applying the Reglem}
$$
d_{G''}(x) \ge d_{G'}(x) - 3d'n - |V_0| \ge d_G(x) - (d+\eps)n -3d'n - 11 \sqrt{d}n \ge d_{G}(x)-4d'n.
$$

For (i)--(vi), we proceed similarly as in the proof of
Lemma~\ref{Degree sequence of R_{G'}}. For (i), consider $i$
clusters with outdegrees at most $d_i^+(R_{G''})$ in
$R_{G''}$. These clusters contain
$im$ vertices of $G$, so must include a typical vertex~$x$ of
outdegree at least $d_{im}^+(G)$. As in the previous estimate, the
cluster~$V$ containing~$x$ satisfies
$$
d_{im}^+(G) \le d_G^+(x) \leq d_{R_{G''}}^+(V)m+ 4d'n +|V_0|\le d_{R_{G''}}^+(V)m + \frac{9}{2}d' n.
$$
Therefore
\[
d_i^+(R_{G''}) \ge d^+_{R_{G''}}(V) \geq \frac{1}{m}d_{im}^+(G) - \frac{9}{2}d'\frac{n}{m}
\geq \frac{1}{m}d_{im}^+(G) - 5d'k,
\]
which proves (i).
Next, (iii) follows from (i), since $\delta^+(G) \geq \beta n$.
The proof of (v) is the same as that of (v) in Lemma~\ref{Degree sequence of R_{G'}},
with $2d$ replaced by $5d'$.
\COMMENT{
To prove (v), suppose
that $d_i^+(R_{G'}) < \min{\left\{i + \frac{\beta}{2}k,
\left(\frac{1}{2} - 5d' \right) k\right\}}$.
It follows from (i) and $d' \ll \beta$ that
$d^+_{im}(G) < \min{\left\{ m(i + \beta k/2 + 5d'k),
mk/2 \right\}} \leq \min{\left\{ im + \beta n , n/2 \right\}}.
$
Using our degree assumptions gives
$d^-_{n - im - \beta n}(G)  \ge n-im$.
Then by (ii) we have
$
d^-_{\left(1 - \frac{\beta}{2} \right)k - i}(R_{G''})
\geq \frac{1}{m} d^-_{\left(1 - \frac{\beta}{2} \right)km -im}(G) - 5d'k
\geq k - i - 5d'k,
$
as required.}
The proofs of the other three assertions are similar.
\endproof

By removing at most one extra vertex from each cluster we may assume that the size of each cluster is even.\COMMENT{This is needed in Lemma~\ref{split}} We continue to denote the sizes of the modified clusters by $m$ and
the set of exceptional vertices by $V_0$. The large-scale structure
of our decomposition will not undergo any significant further changes:
there will be no further changes to the cluster sizes,
although in some subsequent cases we may add a small number of clusters
to $V_0$ in their entirety. For future reference we note the following properties:
\begin{itemize}
\item $|V_0| \le 11\sqrt{d}n$,
\item all edges of $R_{G''}$ correspond to $(10\eps,d'/2)$-regular pairs
(the deletion of atypical vertices may have reduced the densities slightly),
\item all edges of $F$ correspond to $(10\eps,d/4)$-super-regular pairs.
\end{itemize}


\section{The highly connected case} \label{highconnect}

In this section we illustrate our methods by proving Theorem~\ref{CKKO - Approximate Chvatal}
in the case when the auxiliary graph $H$ is strongly $\eta k$-connected.
We recall that $d' \ll \eta \ll \beta$, and that $H$ was defined in Section~\ref{techniques}
as a `shifted version' of $R_{G''}$, i.e.\ there is an edge in $H$ from a cluster $V_i$ to a cluster $V_j$
if there is a shifted walk (with respect to $R_{G''}$ and $F$) from $V_i$ to $V_j$ which traverses
exactly one cycle. We refer to that section for the definitions of when a cluster is `used'
or `internally used' by a shifted walk, and recall that we can assume that any cluster is
used at most once as an entrance and at most once as an exit.
\begin{lemma} \label{disjointpaths}
Suppose $H$ is strongly $c k$-connected for some $c>0$ and $a,b$ are vertices of $H$ (i.e. clusters).
Then there is a collection of at least $c^2 k/16$ shifted walks (with respect to $R_{G''}$ and $F$)
from $a$ to $b$ such that each walk traverses at most $2/c$ cycles
and each cluster is internally used by at most one of the walks.
\end{lemma}
\proof Since $H$ is strongly $ck$-connected we can find $ck$
internally disjoint paths $P_1,\cdots,P_{ck}$ from $a$ to $b$. There
cannot be $ck/2$ of these paths each having at least $2/c$ internal
vertices, as $H$ has $k$ vertices. Therefore $H$ contains at least
$\ell:=c k/2$ internally disjoint paths $P_1,\dots,P_\ell$ (say,
after relabelling) from $a$ to $b$ which have length at most $2/c$.
Note that each of these corresponds to a shifted walk from $a$ to
$b$ which traverses at most $2/c$ cycles. Let $W_1,\dots,W_\ell$
denote these shifted walks. Since the $P_i$ are internally disjoint,
each cluster $x$ is internally used by at most 2 of the shifted
walks $W_j$ (either as an entrance or as an exit).
Each shifted walk $W_i$ internally uses at most $4/c$ clusters, so
there are at most $8/c - 1$ other shifted walks $W_j$ which
internally use a cluster that $W_i$ also uses
internally. Thus we can greedily choose a subset of the walks
$W_1,\dots,W_\ell$ having the required properties.
\endproof

Given any cluster $X$, recall that we write $X^+$ for the successor of $X$ on $F$
and $X^-$ for its predecessor. For every $X$, we apply Lemma~\ref{ideal} with $\theta=16d$
to the $(10\eps,d/4)$-super-regular pair $(X,X^+)_{G'}$ to obtain an $(\sqrt{\eps},d^2)$-ideal $(X_1,X^+_2)$.
Set $X^*:=X_1 \cup X_2$ (where $(X_1^-,X_2)$ is the ideal chosen for $(X^-,X)$).%
\COMMENT{Complement of previous definition for consistency with later notation.}
Then, by Lemma~\ref{ideal}, we have $|X^*| \le 32dm$ and for any $X^* \sub X' \sub X$ and
$(X^+)^* \sub (X^+)^\prime \sub X^+$ the pair $(X',(X^+)^\prime)$ is $(\sqrt{\eps},d^2)$-super-regular.

First we construct the walk $W$ described in the overview.
List the elements of the exceptional set as $V_0 = \{x_1,\dots,x_r\}$.
We go through the list sequentially, and for each $x_i$ we pick clusters
$X_i$ and $Y_i$ of $R_{G'}$ and vertices $x_i^- \in N^-_G(x_i) \cap (X_i \setminus X_i^*)$ and $x_i^+ \in N^+_G(x_i) \cap (Y_i \setminus Y_i^*)$ such that $x_1^-,x_1^+,\ldots,x_r^-,x_r^+$ are distinct and moreover no cluster of~$R_{G''}$
appears more than $m/60$ times as a cluster of the form $X_i, Y_i$
(and thus no cluster appears more than $m/20$ times as a cluster of the form $X_i,X_i^+,Y_i,Y_i^-$).
To see that this is possible, recall that $|V_0| \le 11d^{1/2} n$ and $|X^*| \le 32dm$ for all $X$.
At most $3|V_0|$ vertices belong to $V_0$ or to the set $\{x_1^-,x_1^+,\ldots,x_{i-1}^-,x_{i-1}^+\}$,
and at most $120|V_0|$ belong%
  \COMMENT{$=(2|V_0|/(m/60))m$}
to clusters that appear at least $m/60$ times as $X_j$ or $Y_j$ for $x_j$ with $j<i$.
Therefore at most $1500d^{1/2}n \le \beta n \le \delta^\pm(G)$ vertices are unavailable at stage $i$,
so we can choose $x_i^-$ and $X_i$ as required. A similar argument applies for $x_i^+$ and $Y_i$.
Note that by construction each cluster contains at most $m/20$ of the vertices $x^\pm_i$.%
\COMMENT{** changed again}

Next we sequentially define shifted walks $W(Y_i,X^+_{i+1})$ with respect to  $R_{G''}$ and $F$
from $Y_i$ to $X^+_{i+1}$ for $1 \le i \le r-1$.
We want each $W(Y_i,X^+_{i+1})$ to traverse
at most $2/\eta$ cycles and each cluster to be internally used at most $m/30$ times
by the collection of all the walks $W(Y_i,X^+_{i+1})$.
To see that this is possible, suppose we are about to find $W(Y_i,X^+_{i+1})$
and let $A$ be the set of clusters internally used at least $m/40$ times
by the walks $W(Y_j,X^+_{j+1})$ with $j<i$. Since each of our walks internally uses at most
$4/\eta$ clusters (although it visits many more)
we have $|A| < \frac{11 d^{1/2} n \cdot 4/\eta}{m/40} < \eta^2 k/16$ (since $d \ll \eta$).
Now Lemma~\ref{disjointpaths} implies that
we can find a shifted walk $W(Y_i,X^+_{i+1})$ from $Y_i$ to $X^+_{i+1}$ that
traverses at most $2/\eta$ cycles and does not internally use any cluster in $A$.
We may assume that $W(Y_i,X^+_{i+1})$ uses each cluster at most once as an entrance
and at most once as an exit, and then
no cluster is internally used more than $2+m/40\le m/30$ times by the collection of all the
walks $W(Y_j,X^+_{j+1})$ for all $j\le i$, as required.

We conclude this step by choosing a shifted walk $W(Y_r,X_1^+)$
from $Y_r$ to $X_1^+$. Since there may be clusters in $R_{G''}$ that
we have not yet used, we construct this walk as a sequence of at
most $k$ shifted walks each traversing at most $2/\eta$ cycles,
in such a way that every cluster is used at least once by $W(Y_r,X_1^+)$.

This leads us to define a closed walk $W$
with vertex set $V_0 \cup V(R_{G''})$ as follows.
Let $W(Y_i,X_{i+1})$
be the walk from $Y_i$ to $X_{i+1}$ which is obtained from $W(Y_i,X^+_{i+1})$
by adding the path from $X^+_{i+1}$ to $X_{i+1}$ in~$F$.
We now define
$$W = x_1 W(Y_1,X_2)x_2 \ldots x_r W(Y_r,X_1) x_1.$$
Using the choice of the clusters $X_i$ and $Y_i$ it is easy to see that $W$ uses
every cluster of~$R_{G''}$ at most $m/20+m/30+k\cdot 8/\eta\le m/10$ times.
Thus $W$ has the properties mentioned in the overview, namely:
\begin{itemize}
\item[(a)] For each cycle $C$ of $F$, $W$ visits every vertex of $C$
the same number of times;
\item[(b$'$)] $W$ visits every cluster of $R_{G''}$ at least once
and uses every cluster of $R_{G''}$ at most $m/10$ times;
\item[(c)] $W$ visits every vertex of $V_0$ exactly once;
\item[(d)] For each $x_i \in V_0$ we have chosen
an inneighbour $x_i^-$ in the cluster $X_i$ preceding $x_i$ on $W$ and
an outneighbour $x_i^+$ in the cluster $Y_i$ following $x_i$ on $W$,
so that as $x_i$ ranges over $V_0$ all the vertices $x_i^+$, $x_i^-$ are distinct.
\end{itemize}
Now we fix edges in $G$ corresponding to all edges of $W$ that do not lie within a cycle of $F$.
We have already fixed the edges incident to vertices of $V_0$ (properties (c) and (d)).
Then we note that the remaining edges of $W$ not in $E(F)$ are precisely those of the form $AB$ where
$A$ is used as an exit by $W$ and $B$ is used as an entrance by $W$. To see this, note that we cannot
have $A=B^-$, as then $B$ would be used twice as an entrance in one of the shifted walks constructed
above, which is contrary to our assumption.
Next we proceed through the clusters $V_1,\ldots,V_{k}$ sequentially choosing edges as follows.
When we come to $V_i$, we consider each $j<i$ in turn. If $V_iV_j\notin E(F)$ we let
$w_{ij}$ be the number of times that $W$ uses $V_iV_j$. Similarly, if $V_jV_i\notin E(F)$
we let $w_{ji}$ be the number of times that $W$ uses $V_jV_i$. We aim to choose a matching in $G$
that avoids all previously chosen vertices and uses $w_{ij}$ edges from $V_i \sm V_i^*$
to $V_j \sm V_j^*$ and $w_{ji}$ edges from $V_j \sm V_j^*$ to $V_i \sm V_i^*$.
This can be achieved greedily as follows. Suppose for example that $w_{ij}>0$
and that when we come to $V_j$ the available vertices are $V'_i \sub V_i$ and $V'_j \sub V_j$.
Since every cluster is used at most $m/10$ times we have $w_{ij} \le m/10$,
and we have $|V'_i|, |V'_j| \ge m/2$ (say, taking account of at most $m/10$ uses,
$m/20$ vertices $x^\pm_i$ and $32dm$ vertices in $V_i^*$ or $V_j^*$).
Then $(V'_i,V'_j)_{G''}$ induces a $(20\eps,d'/3)$-regular pair,
so by Lemma~\ref{regmatch} has a matching of size at least $(1-20\eps)m/2 > w_{ij}$.
The same argument can be used if we also have $w_{ji}>0$. After considering all such pairs
$(i,j)$ we have found edges in $G$ corresponding to all edges of $W$ that do not
lie within a cycle of $F$.

Now let $\Entry$ denote the set of all those vertices which do not lie in the exceptional set
and which are the final vertex of an edge of $G$ that we have fixed
(i.e.~the edges incident to the vertices in $V_0$ and the edges chosen in the previous paragraph).
Similarly, let $\Exit$ denote the set of all those vertices which do not lie in the exceptional set
and which are the initial vertex of an edge of $G$ that we have fixed.
Note that $\Entry \cap \Exit= \emptyset$.

For every cluster $U$, let $U_{Exit}:=U \cap \Exit$ and $U_{Entry}= U \cap \Entry$.
Since $W$ was built up by shifted walks, it follows that $|U_{Exit}|=|U^+_{Entry}|$.
Moreover, since we chose $\Entry$ and $\Exit$ to avoid $U^*$,
we know that $(U \sm U_{Exit}, U^+ \sm U^+_{Entry})_{G'}$ is $(\sqrt{\eps},d^2)$-super-regular,
so contains a perfect matching by Lemma~\ref{regmatch}. Now the edges of these perfect matchings
together with the edges of $W$ that we fixed in the previous step form a $1$-factor $\C$ of $G$.
It remains to modify $\C$ into a Hamilton cycle of $G$.

The following statement provides us with the tool we need.
For any cluster $U$, let $G_U:=(U^-\sm U^-_{Exit},U\sm U_{Entry})_{G'}$
and let Old$_U$ be the perfect matching in~$G_U$ which is contained in $\C$.

 \textno
For any cluster $U$, we can find a perfect matching New$_U$ in $G_U$
so that if we replace Old$_U$ in $\C$ with New$_U$, then all
vertices of $G_U$ will lie on a common cycle in the new
$1$-factor~$\C$. In particular, all vertices in $U \sm U_{Entry}$
will lie on a common cycle $C_U$ in~$\C$ and moreover any pair of
vertices of $G$ that were formerly on a common cycle are still on a
common cycle after we replace Old$_U$ by New$_U$. &(\dagger)

To prove this statement we proceed as follows. For every $u \in U
\sm U_{Entry}$, we move along the cycle $C_u$ of $\C$ containing $u$
(starting at $u$) and let $f(u)$ be the first vertex on $C_u$ in
$U^-  \sm U^-_{Exit}$. Define an auxiliary digraph $J$ on $U\sm
U_{Entry}$ such that $N^+_J(u):=N^+_{G_U}(f(u))$.\COMMENT{The other
equation was redundant.} So $J$ is obtained by identifying each pair
$(u,f(u))$ into one vertex with an edge from $(u,f(u))$ to
$(v,f(v))$ if $G_U$ has an edge from $f(u)$ to $v$. Now $G_U$ is
$(\sqrt{\eps},d^2)$-super-regular by the definition of the sets
$U^*$, so $J$ is also $(\sqrt{\eps},d^2)$-super-regular (according
to the definition for non-bipartite digraphs).\COMMENT{Actually,
$G_U$ may have loops but we are still ok since it has at most one
loop on every vertex.} By Lemma~\ref{fk}, $J$ has a Hamilton cycle,
which clearly corresponds to a perfect matching New$_U$ in~$G'$ with
the desired property.

Now we apply ($\dagger$) to every cluster $U$ sequentially. We
continue to denote the resulting $1$-factor by $\C$ and we write
$C_U$ for the cycle that now contains all vertices in $U \sm
U_{Entry}$. Since $U_{Entry}$ and $U_{Exit}$ have size at most $m/4$
(say) for any $U$, we have $V(G_U) \cap V(G_{U^+}) \ne \emptyset$,
so $C_U = C_{U^+}$. Then $C_{U^-}=C_U=C_{U^+}$, and since $U_{Entry}
\cap U_{Exit} =\emptyset$, we deduce that $C_U$ actually contains
all vertices of $U$. Then $C_U=C_{U^+}$ implies that $C_U$ contains
all vertices lying in clusters belonging to the cycle of $F$
containing $U$.

We now claim that $\C$ is in fact a Hamilton cycle. For this, recall that $W(Y_r,X_1^+)$ used every cluster.
Write $W(Y_r,X_1^+)= U_1 C_1 U_1^- U_2 C_2 U_2^- \ldots U_t C_t U^-_tU_{t+1}$,
where each cluster appears at least once in $U_1,\dots,U_{t+1}$.
Let $u^-_iu_{i+1}$ be the edge that we have chosen for the edge $U^-_iU_{i+1}$ on~$W(Y_r,X_1^+)$.
Note that for each $i=1,\dots, t$ the vertices $u_{i+1}$ and $u^-_i$ lie on a common cycle of~$\C$,
as this holds by construction of $\C$, whatever matchings we use to create $\C$.
Since $u_i,u^-_i\in U_i$ also lie on a common cycle, this means that all of
$u_1,\dots,u_{t}$ (and thus also $u_{t+1}$) lie on the same cycle $C$ of~$\C$,
which completes the proof.


\section{Structure II: Shifted components, transitions and the exceptional set}
\label{structure2}

Having illustrated our techniques in the case when $H$ is strongly $\eta k$-connected,
we now turn to the case when this does not hold. In this section we impose further structure on $G$
by introducing `shifted components' of $H$ and various matchings linking these
components and the vertices of the exceptional set~$V_0$. In the first subsection we construct
the shifted components. We describe some of their properties in the second subsection.
The third subsection describes a process by which our shifted walk $W$ will make transitions
between the shifted components. In the fourth subsection we partition $V_0$ into $4$ parts
according to the existence of certain matchings between $V_0$ and the remainder of the digraph.
Then we complete the description of the transitions in the fifth subsection.
Since we need to introduce a large amount of notation in this section,
we conclude with a summary of the important points.

We recall that $\eps \ll d \ll \gamma \ll d' \ll \eta \ll \eta' \ll \beta$
and $|V_0| \le 11d^{1/2}n$.

\subsection{Shifted components of $H$}

Note that the in- and outdegrees of $H$ are obtained by permuting those of $R_{G''}$,
so $H$ has the same in- and outdegree sequences as $R_{G''}$,
and the bounds in Lemma~\ref{Degree sequence of R_{G'} - Modified} also apply to $H$.
We start by establishing an expansion property for subsets of $V(H)$.

\begin{lemma} \label{boundexpand}
If $X \sub V(H)$ with $|X| \le (1-\beta)k/2$ then
$$|N^{\pm}_H(X)| \ge |X| + \frac{\beta}{2}k - 5d'k - 1 \ge |X|+\frac{\beta}{4}k.$$
\end{lemma}
\proof
The argument is similar to that for Lemma~\ref{Outexpansion of R_{G'}}.
By symmetry it suffices to obtain the bound for $|N^+_H(X)|$.
Suppose for a contradiction that
$|N^+_H(X)| < |X| + \frac{\beta}{2}k - 5d'k - 1$.
By Lemma~\ref{Degree sequence of R_{G'} - Modified}(iii) we have $|X| > 5d'k + 1$.
Also $d^+_{|X| - 5d'k-1}(H) \leq |N^+_H(X)| < (1/2- 5d')k$,
so by Lemma~\ref{Degree sequence of R_{G'} - Modified}(v) we have
$d^-_{(1 - \beta/2)k - |X| + 5d'k + 1}(H) \geq k - |X| + 1$.
Then $H$ contains at least $|X| + \beta k/2 - 5d'k$ vertices
of indegree at least $k - |X| + 1$, and these all belong to $N^+_H(X)$, a contradiction.
\endproof

\noindent
We are assuming that $H$ is not strongly $\eta k$-connected, so we can choose
a separator $S$ of $H$ of size $|S| < \eta k$. Thus we have a partition
of the vertices of $H$ into sets $S$, $C$, and $D$ such that $H\sm S$ does not contain an
edge from $C$ to $D$ (although it might contain edges from $D$ to $C$).

\begin{lemma} \label{CDsize}
$|C|,|D| = k/2 \pm 2\eta k.$
\end{lemma}
\proof
Suppose for a contradiction that $|D|<k/2-2\eta k$.
If the stronger inequality $|D| \le (1-\beta)k/2$ holds then Lemma~\ref{boundexpand}
implies that $|N^-_H(D)|\ge |D|+\frac{\beta}{4}k>|D|+|S|$, a contradiction.
So we may assume that $|D| \ge (1-\beta)k/2$. Let $D'$ be a subset of $D$
of size $(1-\beta)k/2$. Now the first inequality of Lemma~\ref{boundexpand} implies that
$$
|N^-_H(D)| \ge |N^-_H(D')| \ge k/2  - 5d'k -1> (|D|+2\eta k)- 5d'k-1
> |D|+\eta k  \ge |D|+|S|,
$$
a contradiction. The bound $|C| \ge k/2-2\eta k$ is obtained in a similar way,
which proves the lemma.
\endproof
Let $C_{\rm small}$ be the set of vertices in $C$ which (in the digraph $H$)
have at most $\beta k/10$ inneighbours in $C$.
Let $D_{\rm small}$ be the set of vertices in $D$ which (in the digraph $H$)
have at most $\beta k/10$ outneighbours in $D$.
\begin{lemma}\label{small}
$|C_{\rm small}|,|D_{\rm small}| \le 8\eta k$.
\end{lemma}
\proof Let $C_{\rm big}$ be the set of vertices in $C$ which have at
least $k/2-\eta k$ outneighbours in $H$. We claim that $|C_{\rm
big}|\ge \beta k/5$. To see this, first note that Lemma~\ref{CDsize}
and the fact that there are no edges from $C$ to $D$ imply that $D$
contains no vertex of indegree greater than $|D|+|S| \le k/2+3\eta
k$. So again by Lemma~\ref{CDsize}, the number of vertices of
indegree greater than $k/2+3\eta k$ in $H$ is at most $k/2+3\eta k$,
which gives $d^-_{k/2-3\eta k} \le k/2 + 3\eta k$. Now
Lemma~\ref{Degree sequence of R_{G'} - Modified}(v) with
$i=k/2-\beta k/4$ says that $d^+_{k/2-\beta k/4} \ge (1/2-5d')k$ or
$d^-_{k/2-\beta k/4} \ge k/2 + \beta k/4 - 5d'k$. The latter option
cannot hold, as it would contradict our previous inequality for
$d^-_{k/2-3\eta k}$, so the former option holds, and $H$ has at
least $k/2+\beta k/4$ vertices of outdegree at least $k/2-5d'k \ge
k/2-\eta k$. By Lemma~\ref{CDsize}, $C$ has to contain at least
$\beta k/5$ of these vertices of high outdegree, which proves the
claim.

Now note that yet another application of Lemma~\ref{CDsize} shows that
every vertex in $C_{\rm big}$ has at least
$k/2-\eta k -|S| \ge |C|-4\eta k$ outneighbours in $H[C]$.
Suppose that $|C_{\rm small}| > 8\eta k$.
Then every vertex in $C_{\rm big}$ has more than half of the vertices of $C_{\rm small}$
as outneighbours. This in turn implies that there is a vertex in $C_{\rm small}$
with more than half the vertices in $C_{\rm big}$ as inneighbours.
In particular, it has more than $\beta k/10$ inneighbours in $C$.
This contradicts the definition of $C_{\rm small}$, so in fact $|C_{\rm small}| \le 8\eta k$.
The argument for $D_{\rm small}$ is similar.
\endproof
Let $C':=C \sm C_{\rm small}$ and $D':=D \sm D_{\rm small}$.

\begin{lemma} \label{CDconnectivity}
$H[C']$ and $H[D']$ are strongly $\eta' k$-connected.
\end{lemma}
\proof
By symmetry it suffices to consider $H[C']$.
The definition of $C_{\rm small}$ and Lemma~\ref{small} give
$\delta^-(H[C'])  \ge \beta k/10-|C_{\rm small}| \ge \beta k/11$.
Suppose that $H[C']$ is not strongly $\eta' k$-connected.
Then there is a separator $T$ of size at most $\eta' k$ and a partition  $U$, $W$ of $C' \setminus T$ such that
$H[C']\sm T$ contains no edge from $U$ to $W$.
Note that $|W| \ge \delta^-(H[C'])-|T| \ge \beta k/12$.
So
\begin{equation} \label{eqU}
|U| \le |C'|-|W| \le (k/2+2\eta k)- \beta k/12 \le k/2- \beta k/13.
\end{equation}
If the stronger inequality $|U| \le (1-\beta)k/2$ holds then Lemma~\ref{boundexpand}
implies that $|N^+(U)|\ge |U|+\beta k/4 >|U|+|S|+|T|+|C_{\rm small}|$, a contradiction.
So we may assume that $|U| \ge (1-\beta)k/2$. Let $U'$ be a subset of $U$
of size $(1-\beta)k/2$. Now the first inequality in Lemma~\ref{boundexpand} implies that
$$
|N^+(U)| \ge |N^+(U')| \ge k/2  - 5d'k - 1 \stackrel{(\ref{eqU})}{\ge}
(|U|+\beta k/13)- 5d'k -1> |U|+|S|+|T|+|C_{\rm small}|,
$$
a contradiction again.
\endproof
Let $S'$ be the set obtained from $S$ by adding $C_{\rm small}$ and $D_{\rm small}$.
So $|S'|\le 17\eta k$ and $S'$, $C'$, $D'$ is a vertex partition of $H$.

Now let $L$ (for `left') be the set obtained from $C'$ by adding all those vertices $v$ from
$S'$ which satisfy $|N_H^{+}(v) \cap C'| \ge \eta' k$ and $|N_H^{-}(v) \cap C'| \ge \eta' k$.
Next, let $R$ (for `right') be the set obtained from $D'$ by adding all those remaining vertices $v$
from $S'$ which satisfy $|N_H^{+}(v) \cap D'| \ge \eta' k$ and $|N_H^{-}(v) \cap D'| \ge \eta' k$.
Then $H[L]$ and $H[R]$ are both still $\eta' k$-connected.
We write $M_V$ (for `vertical middle') for the remaining vertices in $S'$
(i.e.~those which were not added to $C'$ or $D'$). Then $|M_V| \leq |S'| \le 17 \eta k$.
Moreover, $L$, $M_V$ and $R$ partition the vertex set of $R_{G''}$.

We also define another partition of $V(R_{G''})$ into three sets
which we call $T,M_H$ and $B$ (for `top', `horizontal middle', and `bottom') as follows:
\begin{itemize}
\item a cluster belongs to $T$ if and only if its successor in
$F$ belongs to $L$;
\item a cluster belongs to $M_H$ if and only if its
successor in $F$ belongs to $M_V$;
\item a cluster belongs to $B$ if and
only if its successor in $F$ belongs to $R$.
\end{itemize}
The general picture (including a partition of $V_0$ defined below) is
illustrated in Figure~1.
For each of the above subsets of $V(H)=V(R_{G''})$ we use a `tilde' notation to denote
the subset of $V(G)$ consisting of the union of the corresponding clusters,
thus $\wt{L} = \cup_{U \in L} U \sub V(G)$, etc.
Note that
\begin{equation}\label{eq:middle}
|M_H|=|M_V| \le 17 \eta k.
\end{equation}

\begin{figure}\label{structure}
\includegraphics[scale=0.4]{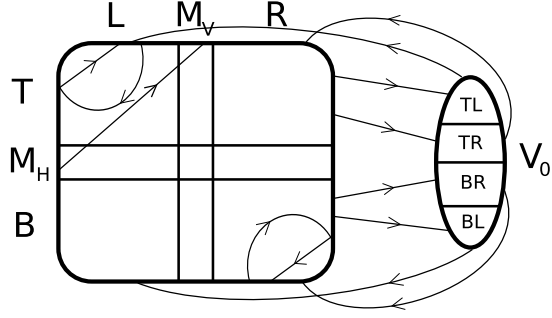}
\caption{Shifted components and the exceptional set}
\end{figure}

We need to remove certain cycles from $F$ that would create difficulties later on.
Let $M:=M_V \cup M_H$.
We say that a cycle $C$ of $F$ \emph{significantly intersects $M$} if $|C \cap M| \ge |C|/10$.
If we have $|\wt{M}| \le |V_0|/\gamma^3$ then
we remove all cycles that significantly intersect $M$ from $F$ and
add all vertices in their clusters to the exceptional set.%
\COMMENT{We only use non-significant intersection to choose a cluster at distance~2 from $M$ when defining
$\Z_L\cup \Z_R$ (so a 4 instead of 10 would be enough). But this distance 2 property is needed since it
ensures that no cluster in $\Z_L\cup \Z_R$ is nearly 4-good. So we cannot just add all cycles lying
entirely in~$M$ to~$V_0$.}
Since $d \ll \gamma$ we still have the inequality
\begin{equation} \label{boundV0}
|V_0| \le 11 d^{1/2} n +10\cdot 11 d^{1/2} n /\gamma^3 \le d^{1/4} n.
\end{equation}
Later we will distinguish the following two cases according to the size of $\wt{M}$.
\begin{itemize}
\item[($\star$)] $|\wt{M}| \le |V_0|/\gamma^3$. Moreover, no cycle of $F$ significantly intersects $M$.
\item[($\star \star$)] $|\wt{M}| \ge |V_0|/\gamma^3>0$.
\end{itemize}
\COMMENT{1. $V_0 \ne \emptyset$ as we deleted vertices e.g. in
making pairs super-regular. 2. $\gamma^2$ would probably be ok too,
but we didn't want to risk changing it as it doesn't really make
things prettier} (The proof would be considerably simpler if we
could remove all the cycles which significantly intersect/lie in
$M$ in the case ($\star \star$), but this would make $|V_0|$ too
large.) Since any cycle in $F$ has equal intersection sizes with
$M_H$ and $M_V$ we still have $|M_H|=|M_V|$ of size at most $17 \eta
n$. We still denote the remaining subset of $R$ by $R$, and
similarly for all the other sets $B,L,M_V$ etc.

\subsection{Properties of the shifted components.}

We start by justifying the name `shifted components'. The following lemma shows that
we have decomposed most of the digraph into two pieces of roughly equal size,
where in each piece we have the high connectivity that enabled us to establish
the result in the previous section.

\begin{lemma} \label{propRL} $ $
\begin{enumerate}
\item[(i)] $H[L]$ and $H[R]$ are strongly $\eta' k/2$-connected.
\item[(ii)] $|\wt{L}|,|\wt{R}|,|\wt{B}|,|\wt{T}|=n/2 \pm 19 \eta n$.
\end{enumerate}
\end{lemma}
\proof
To prove (i), recall that before the removal of the cycles we knew that $H[L]$ and $H[R]$ were
strongly $\eta' k$-connected. By~(\ref{boundV0}), the number of clusters removed in case ($\star$)
is at most $d^{1/4} n/m  \le \eta' k/2$. Since we only removed entire $F$-cycles,
we did not delete any edges from $H$ other than those incident to the clusters that
were deleted. Thus for each cluster removed the connectivity only decreases by at most one.

For (ii), we recall that $|C| = k/2 \pm 2\eta k$ (Lemma~\ref{CDsize})
and $L$ was obtained from $C$ by removing $|C_{\rm small}| \le 8\eta k$ clusters,
adding at most $|S'| \le 17\eta k$ clusters and removing at most $d^{1/4} n/m  \le \eta k$ clusters.
The argument for $R$ is the same. The other two bounds follow since $|B|=|R|$ and $|T|=|L|$.
\endproof

Next we define a partition of $M_V$ into $M_V^{LR}$ and $M_V^{RL}$ as follows.
A cluster $X \in M_V$ belongs to $M_V^{LR}$
if $|N_H^{+}(X) \cap C'| < \eta' k$ and $|N_H^{-}(X) \cap D'| < \eta' k$.
A cluster $X \in M_V$ belongs to $M_V^{RL}$
if $|N_H^{+}(X) \cap D'| < \eta' k$ and $|N_H^{-}(X) \cap C'| < \eta' k$.
The definition of $L$ and $R$ and the fact that $H$ has minimum semidegree
at least $\beta k/2$ imply that this is indeed a partition of $M_V$.
Since $|M_V| \le 17 \eta k$ and $\delta^{\pm}(H)\ge \beta k/2$
we have the following properties.

\begin{lemma}\label{LRandRLmiddle} $ $
\begin{itemize}
\item[(i)] All $V\in M_V^{LR}$ satisfy $|N_H^{+}(V) \cap L|,|N_H^{-}(V) \cap R| < 2\eta' k$
and $|N_H^{+}(V) \cap R|,|N_H^{-}(V) \cap L| >\beta k/3$.
\item[(ii)] All $V\in M_V^{RL}$ satisfy $|N_H^{+}(V) \cap R|,|N_H^{-}(V) \cap L| < 2\eta' k$
and $|N_H^{+}(V) \cap L|,|N_H^{-}(V) \cap R|>\beta k/3$.
\end{itemize}
\end{lemma}

Let $M_H^{LR}$ the set of clusters whose
successor in $F$ belongs to $M_V^{LR}$ and define $M_H^{RL}$ similarly.
Note that this yields a partition of $M_H$.

It will be helpful later to note that if $M_V^{LR}\neq \emptyset$ then we can use clusters
in $M_V^{LR}$ to obtain shifted walks from $L$ to $R$. Similarly, any clusters in $M_V^{RL}$
can be used to obtain shifted walks from $R$ to $L$. This will use the following lemma.

\begin{lemma}\label{middleedges} $  $
\begin{enumerate}
\item[(i)] For all $x \in \wt{M}_H^{LR}$, we have $|N^+_G(x) \cap \wt{L}| \le 3 \eta' n$
and $|N^+_G(x) \cap \wt{R}| \ge \beta n/2$. Also, at most $12\eta' n$ vertices in $\wt{L}$
have more than $|\wt{M}^{LR}_H|/4$ inneighbours in $\wt{M}^{LR}_H$.
\item[(ii)] For all $x \in \wt{M}_V^{LR}$, we have $|N^-_G(x) \cap \wt{B}| \le 3 \eta' n$
and $|N^-_G(x) \cap \wt{T}| \ge \beta n/2$. Also, at most $12\eta' n$ vertices in $\wt{B}$
have more than $|\wt{M}^{LR}_V|/4$ outneighbours in $\wt{M}^{LR}_V$.
\item[(iii)] For all $x \in \wt{M}_H^{RL}$, we have $|N^+_G(x) \cap \wt{R}| \le 3 \eta' n$
and $|N^+_G(x) \cap \wt{L}| \ge \beta n/2$. Also, at most $12\eta' n$ vertices in $\wt{R}$
have more than $|\wt{M}^{RL}_H|/4$ inneighbours in $\wt{M}^{RL}_H$.
\item[(iv)] For all $x \in \wt{M}_V^{RL}$, we have $|N^-_G(x) \cap \wt{T}| \le 3 \eta' n$
and $|N^-_G(x) \cap \wt{B}| \ge \beta n/2$. Also, at most $12\eta' n$ vertices in $\wt{T}$
have more than $|\wt{M}^{RL}_V|/4$ outneighbours in $\wt{M}^{RL}_V$.
\end{enumerate}
\end{lemma}
\proof
For (i), suppose $x \in \wt{M}_H^{LR}$ satisfies $|N^+_G(x) \cap \wt{L}| \ge 3 \eta' n$.
Note that Lemma~\ref{Degree sequence of R_{G'} - Modified} implies that $x$ is typical.
Using the definition of `typical' and accounting for vertices added to $V_0$
we still have $|N^+_{G''}(x) \cap \wt{L}| \ge 2 \eta' n$.
Then the cluster $U$ containing $x$ must have (in $R_{G''}$) at least
$2 \eta' k$ outneighbours in $L$.
The definition of $M_H^{LR}$ implies that the successor $U^+$ of $U$ lies in $M_V^{LR}$.
Then $|N^+_H(U^+) \cap L| = |N^+_{R_{G''}}(U) \cap L| \ge 2 \eta' k$,
contradicting Lemma~\ref{LRandRLmiddle}(i). We deduce that $|N^+_G(x) \cap \wt{L}| \le 3 \eta' n$.
It follows that there are at most $3\eta' n |\wt{M}_H^{LR}|$ edges from $\wt{M}_H^{LR}$ to $\wt{L}$,
so the final assertion of (i) holds. For the second bound in (i), we note that
$$
|N^+_{G}(x) \cap \wt{R}| \ge \delta^+(G)- |N^+_G(x) \cap \wt{L}| - |\wt{M}_V| -|V_0|\ge
\beta n- 3 \eta' n - 17 \eta n - d^{1/4} n
\ge \beta n/2.
$$
For (ii), suppose $x \in \wt{M}_V^{LR}$ satisfies $|N^-_G(x) \cap \wt{B}| \ge 3 \eta' n$.
Then the cluster $U\in M_V^{LR}$ containing $x$ must have (in $R_{G''}$) at least
$2 \eta' k$ inneighbours in $B$. Thus in $H$ it has at least $2 \eta' k$ inneighbours in $R$,
contradicting Lemma~\ref{LRandRLmiddle}(i). The remainder of (ii) follows as for (i).
The proof of (iii) is very similar to that of (i) and the proof of (iv) to that of (ii).
\endproof

If $X$ and $Y$ are clusters in $L$, then there are many shifted walks (with respect to~$R_{G''}$ and~$F$)
from $X$ to $Y$. Later we will need that paths corresponding to such walks can be found in $G$,
even if a large number of vertices
in clusters lying on these paths have already been used for other purposes.
This will follow from the following lemma.

\begin{lemma}\label{twinnbs}
Suppose $U$ is a cluster, $u \in U$ and write $s=\eta' k/4$.
\begin{itemize}
\item[(i)] If $U\in R\cup M_V^{RL}$ then there are clusters
$V_1,\dots,V_{s} \in B$ such that $V_iU\in E(R_{G''})$ and
$u$ has at least $d'm/4$ inneighbours in $V_i$ for $1 \le i \le s$.
\item[(ii)] If $U\in T\cup M_H^{RL}$ then there are  clusters
$V_1,\dots,V_{s} \in L$ such that $UV_i\in E(R_{G''})$ and
$u$ has at least $d'm/4$ outneighbours in $V_i$ for $1 \le i \le s$.
\item[(iii)] If $U\in L\cup M_V^{LR}$ then there are clusters
$V_1,\dots,V_{s} \in T$ such that $V_iU\in E(R_{G''})$ and
$u$ has at least $d'm/4$ inneighbours in $V_i$ for $1 \le i \le s$.
\item[(iv)] If $U\in B\cup M_H^{LR}$ then there are clusters
$V_1,\dots,V_{s} \in R$ such that $UV_i\in E(R_{G''})$ and
$u$ has at least $d'm/4$ outneighbours in $V_i$ for $1 \le i \le s$.
\end{itemize}
\end{lemma}
\proof
To prove~(i) recall from Lemma~\ref{propRL} that $H[R]$ is strongly $\eta' k/2$-connected,
and so has minimum indegree at least $\eta' k/2$. Thus any $U\in R$ has inneighbours
$V_1,\dots,V_{2s}$ in $R_{G''}$ such that $V_i \in B$. This also holds for $U\in M_V^{RL}$
by Lemma~\ref{LRandRLmiddle}(ii), since $\beta \gg \eta'$.
In both cases we remove all the~$V_i$ for which $u$ does not have at least $d'm/4$ inneighbours in $V_i$. Then, since $u$ is typical (this was defined before
Lemma~\ref{Degree sequence of R_{G'} - Modified}), we are left with
$2s-\eps^{1/2}k \ge s$ clusters where $u$ has at least $d'm/4$ inneighbours.
Statements~(ii)--(iv) are proved similarly.
\endproof


\subsection{Transitions}\label{transitions}

As in the highly connected case, our general strategy is to find a suitable shifted
walk $W$ and transform it into a Hamilton cycle. We will be able to move easily within $\wt{L}$,
and also within $\wt{R}$, using the same arguments as in the highly connected case.
However, we need other methods to move between $\wt{L}$ and $\wt{R}$, which we will
now discuss. To avoid excessive notation we will just describe how to move from $\wt{R}$ to $\wt{L}$,
as our arguments will be symmetric under the exchange $R \lra L$ (and so $B \lra T$).
To move from $\wt{R}$ to $\wt{L}$ we use two types of `transitions' from $\wt{B}$ to $\wt{L}$.
The first of these is a set of edges $\match_{BL}$ from $\wt{B}$ to $\wt{L}$,
which will `almost' be a matching, and will have certain desirable properties defined as follows.

Given matchings $\match'$ and $\match''$ in $G$ from $\wt{B}$ to $\wt{L}$, we call a
cluster $V$ \emph{full} (with respect to~$\match'\cup \match''$) if it contains
at least $\gamma m$ endvertices of edges in $\match'\cup \match''$.
Given a number $\ell$, we say $V$ is \emph{$\ell$-fair} (with respect to~$\match'\cup \match''$)
if no cluster with distance at most~$\ell$ from~$V$ in~$F$ is full.
A cluster $V$ is \emph{$\ell$-excellent} if it is $\ell$-fair and no cluster with distance at
most~$\ell$ from~$V$ in~$F$ lies in~$M=M_V\cup M_H$ (the `middle').%
     \COMMENT{We use $\ell$-fair for  $\ell=2,3,4,5$ and $\ell$-excellent for $\ell=4,5$.}
We call $\match'\cup \match''$ a \emph{pseudo-matching from $\wt{B}$ to $\wt{L}$}
if the following properties are satisfied:
\begin{itemize}
\item $\match'\cup \match''$ is a vertex-disjoint union of `components',
each of which is either a single edge or a directed path of length 2.
\item Every single edge component has at least one endvertex
in a $4$-excellent cluster, and every directed path of length 2
has both endvertices in $4$-excellent clusters.
\end{itemize}
Given matchings $\match'$ and $\match''$ from $\wt{T}$ to $\wt{R}$, we say that
$\match'\cup \match''$ is a \emph{pseudo-matching from~$\wt{T}$ to~$\wt{R}$} if it
satisfies the analogous properties.%
     \COMMENT{We allow matching edges whose endvertices lie in the same cluster.
But this doesn't seem to create problems.}
As we shall see later, each edge of a pseudo-matching from $\wt{B}$ to $\wt{L}$
allows us to move from $\wt{R}$ to $\wt{L}$. Note that this applies even to the
two edges in any directed paths of length 2: these will enable us to move twice
from  $\wt{R}$ to $\wt{L}$, using the rerouting procedure described later.
Similarly, each edge of a pseudo-matching from $\wt{T}$ to $\wt{R}$
allows us to move from $\wt{L}$ to $\wt{R}$.
We consider pseudo-matchings rather than matchings because in general
$\wt{B} \cap \wt{L} \neq \emptyset$, so the largest matching we can guarantee
is only half as large as the largest pseudo-matching. This would not provide all the
edges we need to move from $\wt{R}$ to $\wt{L}$.

We now choose pseudo-matchings $\match_{BL}$ from $\wt{B}$ to $\wt{L}$
and $\match_{TR}$ from $\wt{T}$ to $\wt{R}$, each of which is maximal subject to the condition
\begin{itemize}
\item $|\match_{BL}|, |\match_{TR}| \le \gamma^2 n$.
\end{itemize}
(Here $|\match_{BL}|$ denotes the number of edges in~$\match_{BL}$.)
Note that $\match_{BL}$ and $\match_{TR}$ may have common vertices.
Recalling that $|M| \le 34\eta k$ by~(\ref{eq:middle}),
\textno at most $2|\match_{BL}|/\gamma m \le 3\gamma k$ clusters are full with respect to~$\match_{BL}$,
and at most $11(3\gamma k+|M|)\le 400\eta k$ clusters are not $5$-excellent with respect to~$\match_{BL}$.
A similar statement holds for~$\match_{TR}$.
&(\clubsuit)

From now on, whenever we refer to a fair or excellent cluster it will be with respect to the pseudomatching $\match_{BL}$. 

As in the highly connected case, we will identify `entries' and `exits'
for edges of the cycle that do not lie in a pair corresponding to an edge of $F$.
For $\match_{BL}$, the \emph{exits} is the set $\exit_{BL}$ of all initial vertices
of edges in~$\match_{BL}$, and the \emph{entries} is the set $\entry_{BL}$
of all final vertices of edges in~$\match_{BL}$. (We will define further exits
and entries in due course.)

At this stage, we do not know how many of the matching edges we
actually will need in $W$, as this depends on a partition of the
exceptional set $V_0$ to be defined in the next subsection.
So, given a cluster $V$, we want to ensure if e.g.~we only use some of the vertices
in $V \cap \exit_{BL}$, then the unused remainder of $V$ and $V^+$ still forms a super-regular pair.
We may not be able to achieve this for any $V$, but if $V$ is $2$-fair,
we know that none of $V^-$, $V$, $V^+$ is full, which gives us the flexibility we need.
We say that a cluster $V$ is {\em nearly $2$-fair} if $V$ is either $2$-fair or at distance 1 on $F$
from a $2$-fair cluster. In the following lemma we choose partitions of the nearly $2$-fair clusters
which allow us to avoid any `interference' between exits/entries and the exceptional set.
Figure~2 illustrates these partitions, and also some additional sets that will be defined in
Subsection~\ref{twins} (`twins' of exits/entries and an ideal to preserve super-regularity).

\begin{figure}\label{fig:split}
\includegraphics[scale=0.5]{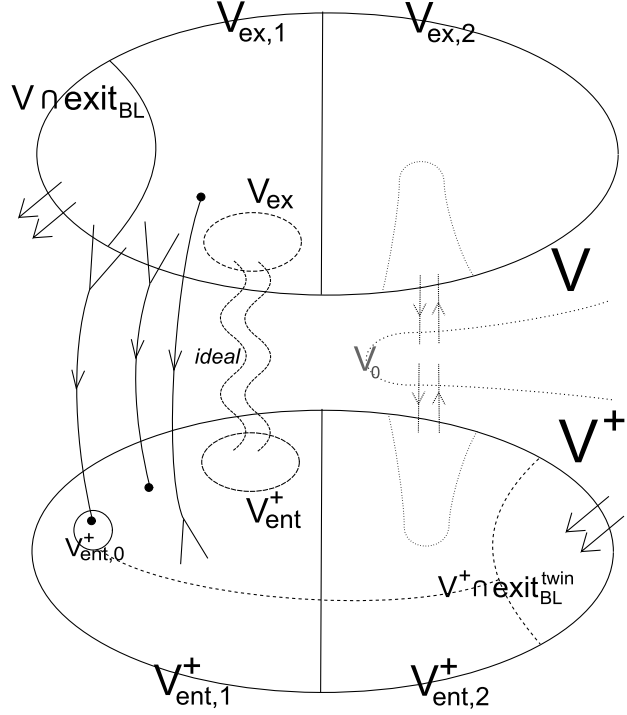}
\caption{Partitions avoiding interference between exits/entries and $V_0$.}
\end{figure}

We define our partitions of the nearly $2$-fair clusters as follows.
For every $2$-fair cluster $V$ with $V \cap \exit_{BL} \ne \emptyset$
we will choose a partition $V_{ex,1},V_{ex,2}$ of $V$ and a partition $V^+_{ent,1}, V^+_{ent,2}$ of $V^+$.
Also, for every $2$-fair cluster $V$ with $V \cap \entry_{BL} \ne \emptyset$
we choose a partition $V_{ent,1},V_{ent,2}$ of $V$ and a partition $V^-_{ex,1}, V^-_{ex,2}$ of $V^-$.
There is no conflict in our notation, i.e.~we will not e.g.~define $V_{ex,1}$ twice,
since when $V \cap \exit_{BL} \ne \emptyset$ we must have $V \in B$,
whereas when $V^+ \cap \entry_{BL} \ne \emptyset$ we must have $V^+ \in L$, so $V \in T$,
and these cannot occur simultaneously.
We also define the analogous partitions with respect to $\match_{TR}$,
although for simplicity we will not explicitly introduce notation for them,
as we will mainly focus on the case when only $\match_{BL}$
is needed for the argument. So for each cluster $V$ we will choose at most~$4$ partitions.
We let $V_{\nd}$ be the intersection of all the second parts of the at most 4 partitions
defined for~$V$. (So if all 4 partitions are defined, then $V_\nd$ is the intersection
of the sets $V_{ex,2},V_{ent,2}$ defined with respect to $\match_{BL}$ and the 2 analogous sets
defined with respect to $\match_{TR}$. If only 3 partitions are defined for $V$, then $V_\nd$ is the
intersection of only~3 sets etc. If no partition is defined for $V$, then $V_\nd=V$.)
We let $X_\nd$ be the union of $V_\nd$ over all clusters~$V$.
We choose these partitions to satisfy the following lemma.%
\COMMENT{this lemma is now much more complicated than in the 0704 version because the previous lemma
wasn't strong enough to prove Lemma 33 in the case where $V$ intersects $\exit_{BL}$
but $|V \cap \Exit_{BL}| \le 10 \eps m$. In this case the vertices
$V^+_{ent,0} \setminus \Exit_{BL}^{twins}$ didn't
have sufficiently high degree in the random set $V_{ex}$ (as defined in that version)
in order to guarantee superregularity.}

\begin{lemma}\label{split}
The partitions $V_{ex,1}, V_{ex,2}$ and $V_{ent,1}, V_{ent,2}$
can be chosen with the following properties (when they are defined).
\begin{itemize}
\item[(i)]  $|V_{ex,1}|=m/2$, $|V_{ent,1}|=m/2$.
\item[(ii)] For any $2$-fair cluster $V$ with $V \cap \exit_{BL} \ne \emptyset$ we have
$V_{ex,2}\cap \exit_{BL}= \emptyset$. Moreover, there is
a set $V^+_{ent,0}\subseteq V^+_{ent,1}$ of size at most $10\eps m$ such that:
\begin{itemize}
\item[$\bullet$] Each vertex in $V^+\sm V^+_{ent,0}$ has at least $dm/40$ inneighbours in $V_{ex,1}\sm \exit_{BL}$.
\item[$\bullet$] Each vertex in $V^+_{ent,0}$ has at least $dm/8$ inneighbours in $V_{ex,1}$.
\item[$\bullet$] Each vertex in $V$ has at least $dm/20$ outneighbours in $V^+_{ent,1}$.
\end{itemize}
\item[(iii)] For any $2$-fair cluster $V$ with $V \cap \entry_{BL} \ne \emptyset$ we have
$V_{ent,2}\cap \entry_{BL}= \emptyset$. Moreover, there is
a set $V^-_{ex,0}\subseteq V^-_{ex,1}$ of size at most $10\eps m$ such that:
\begin{itemize}
\item[$\bullet$] Each vertex in $V^-\sm V^-_{ex,0}$ has at least $dm/40$ outneighbours in $V_{ent,1}\sm \entry_{BL}$.
\item[$\bullet$] Each vertex in $V^-_{ex,0}$ has at least $dm/8$ outneighbours in $V_{ent,1}$.
\item[$\bullet$] Each vertex in $V$ has at least $dm/20$ inneighbours in $V^-_{ex,1}$.
\end{itemize}
\end{itemize}
Also, the analogues of statements (i)--(iii) for $\match_{TR}$ hold. Moreover,
\begin{itemize}
\item[(iv)] Every vertex in~$V_0$ has at least $\beta n/20$ inneighbours and at least $\beta n/20$ outneighbours
in $X_\nd$.
\item[(v)] If $d^+_{(1-\beta)n/2}(G)\ge n/2$ then there are sets $S'_B\subseteq \wt{B}\cap X_\nd$ and
$S'_T\subseteq \wt{T}\cap X_\nd$ such that $|S'_B|,|S'_T|\ge \beta n/80$ and such that every vertex
in $S'_B\cup S'_T$ has outdegree at least $n/2$ in~$G$.
\item[(vi)] If $d^-_{(1-\beta)n/2}(G)\ge n/2$ then there are sets $S'_L\subseteq \wt{L}\cap X_\nd$ and
$S'_R\subseteq \wt{R}\cap X_\nd$ such that $|S'_L|,|S'_R|\ge \beta n/80$ and such that every vertex
in $S'_L\cup S'_R$ has indegree at least $n/2$ in~$G$.
\end{itemize}
\end{lemma}

\proof
Consider a $2$-fair cluster $V$ with $V \cap \exit_{BL} \ne \emptyset$.
If $|V \cap \exit_{BL}| < 20\eps m$ we set $V^+_{ent,0}=\emptyset$.
Otherwise, if $|V \cap \exit_{BL}| \ge 20\eps m$ we define $V^+_{ent,0}$
to be the set of vertices in $V^+$ that have less than $\frac{d}{8}|V \sm \exit_{BL}|$
inneighbours in $V \sm \exit_{BL}$. Recall that $(V,V^+)_{G'}$ is $(10\eps,d/4)$-super-regular.
Since $V$ is $2$-fair we deduce that $|V^+_{ent,0}|\le 10\eps m$.

Now consider constructing a partition of $V$ into $V_{ex,1}$ and $V_{ex,2}$ as follows.
Include $V\cap\exit_{BL}$ into $V_{ex,1}$ and distribute the remaining vertices of $V$
between $V_{ex,1}$ and $V_{ex,2}$ so that $|V_{ex,1}|=m/2$ (recall that $m$ is even),
choosing uniformly at random between all possibilities.
Note that since $V$ is $2$-fair the probability that a vertex of $V \setminus \exit_{BL}$ is included in $V_{ex,1}$
is at least $1/3$. Then by the Chernoff bound for the hypergeometric distribution
(Proposition~\ref{chernoff}), with high probability
each vertex in $V^+\sm V^+_{ent,0}$ has at least $\frac{1}{4}\frac{d}{8}|V \sm \exit_{BL}| \ge dm/40$
inneighbours in $V_{ex,1}\sm \exit_{BL}$.
Also, by definition of $V_{ent,0}^+$ and super-regularity, each vertex in $V^+_{ent,0}$ has at least
$dm/4-\frac{d}{8}|V \sm \exit_{BL}| \ge dm/8$ inneighbours in $V_{ex,1}$.%
\COMMENT{In fact these inneighbours lie in $V \cap \exit_{BL}$, so $V^+_{ent,0}$ can only be non-empty
when $|V \cap \exit_{BL}| \ge dm/8$.}
Next, consider similarly constructing a partition of $V^+$ into $V_{ent,1}^+$ and $V_{ent,2}^+$ as follows.
Include $V_{ent,0}^+$ into $V_{ent,1}^+$ and distribute the remaining vertices of $V^+$
uniformly at random between $V_{ent,1}^+$ and $V_{ent,2}^+$ so that $|V_{ent,1}^+|=m/2$.
Note that any vertex in $V$ has outdegree at least $dm/4-|V^+_{ent,0}| \ge dm/5$ in $V^+ \sm V^+_{ent,0}$.
Again, the probability that a given vertex from $V^+ \sm V^+_{ent,0}$ is included in $V_{ent,1}^+$
is at least $1/3$, so with high probability
each vertex in $V$ has at least $(dm/5)/4 = dm/20$ outneighbours in $V^+_{ent,1}$.
This shows the existence of the partitions required for (ii).
The existence of partitions satisfying (iii) is proven in the same way.

For each vertex $v$ in a cluster $V$ which does not lie in
$\exit_{BL} \cup V_{ent,0} \cup \entry_{BL} \cup V_{ex,0}$
or the analogous set defined with respect to $\match_{TR}$,%
\COMMENT{I.e. $\exit_{TR} \cup V_{ent,0} \cup \entry_{TR} \cup V_{ex,0}$,
where $V_{ent,0}$ and $V_{ex,0}$ are with respect to  $\match_{TR}$.}
the probability that it lies in the second part of each of the (up to) 4
partitions defined on $V$ (and thus lies in $X_\nd$) is at least
$(1/2)^4$. Since $\delta^0(G) \ge \beta n$, a Chernoff bound
(Proposition~\ref{chernoff}) implies that we can also choose the
partitions to satisfy (iv).

Now suppose that  $d^+_{(1-\beta)n/2}(G)\ge n/2$.
Then $G$ contains at least $(1+\beta)n/2$ vertices of outdegree at least $n/2$.
So Lemma~\ref{propRL}(ii) implies that $\wt{B}$ contains a set $\wt{B}'_{\rm large}$ of
at least $\beta n/3$ vertices whose outdegree in $G$ is at least $n/2$.
$\wt{B}'_{\rm large}$ in turn contains a set $\wt{B}_{\rm large}$ of at least $\beta n/4$ vertices
which do not lie in $\exit_{BL} \cup V_{ent,0} \cup \entry_{BL} \cup V_{ex,0}$
(for any cluster~$V$) or in the analogous set defined with respect to $\match_{TR}$.
Similarly as for (iv), with high probability we have $|\wt{B}_{\rm large} \cap X_\nd| \ge (\beta n/4)/20$.
Similar arguments applied to $\wt{T}$, $\wt{L}$ and $\wt{R}$
show that we can choose the partitions to satisfy (v) and (vi).
\endproof

\subsection{The exceptional set}

Next we will assign each vertex $x$ in the exceptional set $V_0$ an {\em in-type} which is one of $T$ or $B$
and an {\em out-type} which is one of $L$ or $R$. Combining these two types together we will
say each vertex of $V_0$ gets a type of the form $TR,TL,BR$ or $BL$.
We will also abuse notion and think of $TL$ as the set of all vertices of $V_0$ of
in-type $T$ and out-type $L$, etc.  We write $\wt{T}^*$ for the set of all
those vertices which belong to the set $X_\nd$ defined in the previous subsection
as well as to clusters of $T$ which are $5$-excellent with respect to both
$\match_{BL}$ and~$\match_{TR}$. The other sets $\wt{B}^*$ etc.~are defined similarly.

\begin{lemma}\label{excvs}
We can assign each vertex $x\in V_0$ an in- and an out-type such that the following
conditions are satisfied.
\begin{itemize}
\item[(i)] There is a matching $\match_T$ from $\wt{T}^*$ to the set of vertices of in-type $T$.
\item[(ii)] There is a matching $\match_B$ from $\wt{B}^*$ to the set of vertices of in-type $B$.
\item[(iii)] There is a matching $\match_L$ from the set of vertices of out-type $L$ to $\wt{L}^*$.
\item[(iv)] There is a matching $\match_R$ from the set of vertices of out-type $R$ to $\wt{R}^*$.
\item[(v)] The endvertices of the matchings $\match_T$, $\match_B$, $\match_L$, $\match_R$
in $V(G) \sm V_0$ are all distinct. Let $V^*_0$ denote the set of all these endvertices.
\item[(vi)] No cluster of $R_{G'}$ contains more than $\gamma m$ vertices of $V_0^*$.%
   \COMMENT{It is important to consider the endpoints collectively, so that we can change
types of vertices without violating this property.}
\item[(vii)] Subject to the above conditions, $\lvert |TR| - |BL| \rvert$ is minimal.
\end{itemize}
\end{lemma}
\proof
To show that such a choice is possible, we claim that we can proceed sequentially
through the vertices of $V_0$, assigning in-types and out-types and greedily extending
the appropriate matchings. Since $|V_0| \le d^{1/4}n$ by~(\ref{boundV0}), at any stage in this process
we have constructed at most $2d^{1/4}n$ edges of the matchings
$\match_T$, $\match_B$, $\match_L$, $\match_R$, and so there are at most
$2d^{1/4}n/\gamma\le d^{1/5} n$ vertices
belonging to clusters which contain at least $\gamma m$ endpoints of the matchings. In addition,
we have to avoid all the at most $800\eta n$
vertices lying in clusters which are not $5$-excellent with respect to both $\match_{BL}$
and $\match_{TR}$. So in total we have to avoid at most $801\eta n$ vertices in each step.
But by Lemma~\ref{split}(iv) each exceptional vertex has in- and outdegree at least $\beta n/20$ in~$X_\nd$,
so Lemma~\ref{propRL}(ii) implies that any vertex has at least $\beta n/50$ inneighbours in $\wt{T} \cap X_\nd$
or at least $\beta n/50$ inneighbours in $\wt{B} \cap X_\nd$.
A similar statement holds for outneighbours in $\wt{L} \cap X_\nd$ or $\wt{R} \cap X_\nd$,
Thus a greedy procedure can satisfy conditions (i)--(vi),
and then we can choose an assignment to satisfy~(vii).
\endproof

Note that one advantage of choosing $V_0^*$ in $X_\nd$ is that $V_0^*$ will be disjoint
from the sets $\Entry_{BL}$ etc.
The strategy of the remaining proof depends on the value of $|TR| - |BL|$.
We will only consider the case $|TR|-|BL| \ge 0$,
as the argument for $|TR|-|BL| \le 0$ is identical,
under the symmetry $L \lra R$, $T \lra B$.
When $|TR|>|BL|$ only $\match_{BL}$ is needed for the argument.
When $|TR|=|BL|$ we do not need either pseudo-matching,
although the case $|TR|=|BL|=0$ has additional complications.

\subsection{Twins}\label{twins}

When $|TR|>|BL|$, we obtain one type of transitions from $\wt{B}$ to $\wt{L}$
by fixing a pseudo-matching $\match'_{BL} \sub \match_{BL}$.
The other type of transitions uses a set $\Entry_{RL} \sub \wt{M}_V^{RL}$, as explained below.
We define exits $\Exit_{BL} \sub \exit_{BL}$ and entries $\Entry_{BL} \sub \entry_{BL}$
of $\match'_{BL}$ as for $\match_{BL}$.
Lemma~\ref{MBL2cor}(i) below implies that $|\match_{BL}|+|\wt{M}_V^{RL}| \ge |TR|-|BL|.$
Thus we can fix sets $\match'_{BL}$ and $\Entry_{RL}$ to satisfy%
  \COMMENT{This may not hold in case $(\star\star)$, but it's easier for the reader if
we don't say this yet.}
$$|\match'_{BL}|+|\Entry_{RL}| = |TR|-|BL|.$$

For each edge $xy \in \match'_{BL}$ we will choose
`twins' $x^{twin}$ and $y^{twin}$ of its endpoints. To use the edge $xy$ in our shifted walk
$W$, we will enter the cycle of $F$ containing $x$ at $x^{twin}$, wind around the cycle to $x$,
use the edge $xy$, wind around the cycle containing $y$, and then leave it at $y^{twin}$.
A vertex that is the midpoint of a directed path of length 2 in $\match'_{BL}$ will actually have two twins,
but we will not complicate the notation to reflect this, as it will be clear from the context
which twin is intended. Thus we obtain two `twin maps' $x \mapsto x^{twin}$ and $y \mapsto y^{twin}$.
We also use the notation $S^{twin}=\{x^{twin}: x\in S\}$ when $S$ is a set of vertices.
The twin maps will be injective on $\Exit_{BL}$ and on $\Entry_{BL}$, in that
$|\Exit_{BL}|=|\Exit_{BL}^{twin}|$, $|\Entry_{BL}|=|\Entry_{BL}^{twin}|$,
and moreover $|V\cap \Exit_{BL}|=|V^+ \cap \Exit_{BL}^{twin}|$,
$|V\cap \Entry_{BL}|=|V^- \cap \Entry_{BL}^{twin}|$.

Our choice of $x^{twin}$ depends on whether the cluster $V$ containing $x$ is $2$-fair
with respect to $\match_{BL}$. If $V$ is not $2$-fair then we fix arbitrary perfect matchings
in $G'$ from $V^-$ to $V$ and from $V$ to $V^+$ (using Lemma~\ref{regmatch}).
Then for every $x \in V \cap \Exit_{BL}$ we let $x^{twin}$ be the vertex $x$ is matched to in $V^+$
and for every $x \in V \cap \Entry_{BL}$ we let $x^{twin}$ be the vertex in $V^-$ matched to $x$.

On the other hand, if $V$ is $2$-fair then we make use of the partitions defined in Lemma~\ref{split}.
If $V \cap \Exit_{BL} \ne \emptyset$ then we choose twins for vertices in $V \cap \Exit_{BL}$
within $(V^+_{ent,2}\cup V^+_{ent,0})\sm V^*_0$, arbitrarily subject to the condition that if $|V \cap \Exit_{BL}| > 20\eps m$
then $(V \cap \Exit_{BL})^{twin}$ contains $V^+_{ent,0}$. (Recall that $V^*_0$ was defined in Lemma~\ref{excvs}(v).) If $V$ is $2$-fair, we will also choose some ideal of $(V\sm \Exit_{BL},V^+\sm \Exit_{BL}^{twin})_{G'}$ to create
flexibility when selecting further sets while preserving super-regularity.
To do this, recall that $(V,V^+)_{G'}$ was
$(10\eps,d/4)$-super-regular. Together with Lemma~\ref{split}(i),(ii) this implies that
$(V_{ex,1}\sm \Exit_{BL},V^+_{ent,1}\sm \Exit_{BL}^{twin})_{G'}$ is $(30\eps,d/40)$-super-regular.
Next we randomly choose sets $V_{ex}\subseteq V_{ex,1}\sm \Exit_{BL}$ and $V^+_{ent}\subseteq V^+_{ent,1}\sm \Exit_{BL}^{twin}$
of size $80dm$. Lemma~\ref{ideal} (applied with $\theta=160d$ and $n=m/2$) implies that with high probability
$(V_{ex},V^+_{ent})$ is an $(\sqrt{\eps},d^2)$-ideal for $(V_{ex,1}\sm \Exit_{BL},V^+_{ent,1}\sm \Exit_{BL}^{twin})_{G'}$.
Moreover, Lemma~\ref{split}(ii) and the Chernoff bound (Proposition~\ref{chernoff}) together imply that with
high probability
every vertex in $V_{ex,2}$ has at least $d^2m$ outneighbours in $V^+_{ent}$ while every vertex in
$V^+_{ent,2}$ has at least $d^2m$ inneighbours in $V_{ex}$. Altogether this shows that we can choose $(V_{ex},V^+_{ent})$
to be a $(\sqrt{\eps},d^2)$-ideal for $(V\sm \Exit_{BL},V^+\sm \Exit_{BL}^{twin})_{G'}$.

Similarly, if $V \cap \Entry_{BL} \ne \emptyset$ then we choose
twins for vertices in $V \cap \Entry_{BL}$ in $(V^-_{ex,2}\cup
V^-_{ex,0}) \setminus V_0^*$, arbitrarily subject to the condition
that if $|V \cap \Entry_{BL}| > 20\eps m$ then $(V \cap
\Entry_{BL})^{twin}$ contains $V^-_{ex,0}$. We also choose a
$(\sqrt{\eps},d^2)$-ideal $(V^-_{ex},V_{ent})$ for $(V^-\sm
\Entry_{BL}^{twin},V\sm \Entry_{BL})_{G'}$. Then we define $X_{BL}$
to be the union of the sets $V_{ex}$ and $V_{ent}$ defined using
$\match'_{BL}$ over all nearly 2-fair clusters~$V$. Note that these
sets will play a similar role to the sets $V^*$ used in the highly
connected case, in that they preserve super-regularity even if when
some vertices are deleted. We let
$$X^*_{BL}:=X_{BL}\cup\Exit_{BL}\cup\Entry_{BL}\cup\Exit^{twin}_{BL}\cup\Entry^{twin}_{BL}.$$
Note that $X^*_{BL}\cap V^*_0=\emptyset$. Define $X_{TR}$ and
$X^*_{TR}$ similarly using the matching $\match'_{TR}$.

We will also choose twins for vertices in $\Entry_{RL}$ such that if $x \in V \in M_V^{RL}$
then $x^{twin} \in V^- \in M_H^{RL}$. Lemma~\ref{middleedges}(iii),(iv) implies that
each $x\in \Entry_{RL}$ has many inneighbours in~$\wt{B}$
while~$x^{twin}$ has many outneighbours in~$\wt{L}$.
Writing $C$ for the cycle of $F$ containing the cluster containing $x$,
we get a transition from~$\wt{B}$ to $\wt{L}$ by entering $C$ at $x$
from an inneighbour in~$\wt{B}$, traversing $C$,
then exiting $C$ at $x^{twin}$ to an outneighbour in~$\wt{L}$.

Now we describe how to choose twins for $\Entry_{RL}$,
and also some ideals to create flexibility while preserving super-regularity.
Call a cluster~$V$ \emph{$M^{RL}$-full} if it contains at least $\gamma m$ vertices in
$\Entry_{RL}$. Say $V$ is \emph{$\ell$-good}
(with respect to $\match_{BL}$ and $\Entry_{RL}$)
if $V$ is $\ell$-fair with respect to~$\match_{BL}$ and no cluster with
distance at most~$\ell$ from~$V$ on $F$ is $M^{RL}$-full.
Since $|\Entry_{RL}| \le |V_0| \le d^{1/4}n$ the number
of $M^{RL}$-full clusters is at most $\gamma^{-1}d^{1/4}n/m$.
Recalling that by~($\clubsuit$) at most $3\gamma k$ clusters are full,%
   \COMMENT{Formerly used excellent to bound good, but incorrectly used $\gamma$ for $\eta$,
which is too weak.}
\textno at most $9(3\gamma k+\gamma^{-1}d^{1/4}n/m) \le 30\gamma k$ clusters are not $4$-good.
&(\diamondsuit)

Consider a cluster $V \in M_V^{RL}$ with $V \cap \Entry_{RL} \ne
\emptyset$. If $V$ is not $2$-good then we choose a perfect matching
in~$G'$ from $V^-$ to $V$ (using Lemma~\ref{regmatch}), and for each
$x \in V$ let $x^{twin}$ be the vertex in $V^-$ that is matched to
$x$. Now suppose that $V$ is $2$-good. Then none of $V^-$, $V$ and
$V^+$ is full with respect to $\match_{BL}$ or $M^{RL}$-full. Since
$(V^-,V)_{G'}$ is $(10\eps,d/4)$-super-regular and $|V \cap
\Entry_{RL}| < \gamma m$ we can apply Lemma~\ref{tw} to obtain a set
$Y \sub V^-$ with $|Y|=|V \cap \Entry_{RL}|$ such that $(V^- \sm Y,
V \sm \Entry_{RL})_{G'}$ is $(20\eps,d/8)$-super-regular. Then we
let the twin map be an arbitrary bijection from $V \cap \Entry_{RL}$
to $Y$. Next we apply Lemma~\ref{ideal} with $\theta=32d$ to obtain
a $(\sqrt{\eps},d^2)$-ideal for $(V^- \sm \Entry_{RL}^{twin}, V \sm
\Entry_{RL})_{G'}$, which we will call $(V^-_{ex},V_{ent})$.
(Similarly to the earlier argument  and the one in the next
paragraph, the partitions of $V(G) \sm V_0$ into $\wt{L}$, $\wt{R}$
and $\wt{M}_V$ and into $\wt{T}$, $\wt{B}$ and $\wt{M}_H$ guarantee
that there is no conflict
with our previous notation.) Now we define%
   \COMMENT{1. Should we include $\Entry_{RL}$ and twins? (cf $X^*_{BL}$) - done this
2. Changed definition of $X^*$.}
$$X^*_{RL} :=  \bigcup_{V \in M_V^{RL}} V_{ent} \cup\bigcup_{V \in M_H^{RL}} V_{ex}
\cup \Entry_{RL} \cup \Entry_{RL}^{twin},$$
\begin{equation} \label{Xstar}
X^*:=X^*_{BL}\cup X^*_{RL}.
\end{equation}
Then $|X^*| \le \gamma n$ and $X^*\cap V^*_0=\emptyset$. (The latter follows since the vertices
in $V^*_0$ lie in 5-excellent clusters and so $X^*_{RL}\cap V^*_0=\emptyset$.) We also define
$$V_{entry}:=V\cap (\Entry_{BL}\cup\Exit_{BL}^{twin}\cup\Entry_{RL}),$$
$$V_{exit}:= V\cap (\Exit_{BL}\cup\Entry_{BL}^{twin}\cup\Entry_{RL}^{twin}).$$
Note that $\Entry_{BL} \sub \wt{L}$, $\Exit_{BL}^{twin} \sub \wt{R}$
and $\Entry_{RL} \sub \wt{M}_V^{RL}$. Since $\wt{L}$, $\wt{R}$ and $\wt{M}_V$
partition $V(G) \sm V_0$, any vertex will be used at most once to enter a cluster.
In particular,
\[ V_{entry} = \begin{cases}
V \cap \Entry_{BL} & \text{if } V \in L; \\
V \cap \Exit_{BL}^{twin} & \text{if } V \in R; \\
V \cap \Entry_{RL} & \text{if } V \in M_V^{RL}.
\end{cases}
\]
Similarly, $\Exit_{BL} \sub \wt{B}$, $\Entry_{BL}^{twin} \sub \wt{T}$
and $\Entry_{RL}^{twin}\sub \wt{M}_H^{RL}$. Since $\wt{T}$, $\wt{B}$ and $\wt{M}_H$
also partition $V(G) \sm V_0$, any vertex will be used at most once to exit a cluster
and
\[ V_{exit} = \begin{cases}
V \cap \Exit_{BL} & \text{if } V \in B; \\
V \cap \Entry_{BL}^{twin} & \text{if } V \in T; \\
V \cap \Entry_{RL}^{twin} & \text{if } V \in M_H^{RL}.
\end{cases}
\]
Some vertices may be used for both exits and entrances, and they will have two twins.
We summarise the properties of twins with the following lemma.

\begin{lemma}\label{twinprops}
Suppose that $|TR|>|BL|$. Then
\begin{itemize}
\item[(i)] $|\match'_{BL}|+|\Entry_{RL}|=|TR|-|BL|$.
\item[(ii)] Every cluster intersects at most one of $\Entry_{BL}$, $\Exit_{BL}^{twin}$, $\Entry_{RL}$.
Similarly, every cluster intersects at most one of $\Exit_{BL}$, $\Entry_{BL}^{twin}$, $\Entry_{RL}^{twin}$.
\item[(iii)] There exists a perfect matching from $V\sm V_{exit}$ to $V^+\sm V^+_{entry}$.
\item[(iv)] Suppose $V$ is $3$-good with respect to $\match_{BL}$ and $\Entry_{RL}$. Then
\begin{itemize}
\item[$\bullet$]
For all sets
$X'$ and $Y'$ with $(V\cap X^*)\sm V_{exit} \sub X'\sub V\sm V_{exit}$
and $(V^+\cap X^*)\sm V^+_{entry} \sub Y' \sub V^+\sm V^+_{entry}$
the pair $(X',Y')_{G'}$ is $(\sqrt{\eps},d^2)$-super-regular.
\item[$\bullet$]
For all sets $X'$ and $Y'$ with
$(V\cap X^*)\sm V_{entry}\sub X'\sub V\sm V_{entry}$
and $(V^-\cap X^*)\sm V^-_{exit} \sub Y' \sub V^-\sm V^-_{exit}$
the pair $(Y',X')_{G'}$ is $(\sqrt{\eps},d^2)$-super-regular.
\end{itemize}
\end{itemize}
\end{lemma}
\proof
As discussed at the beginning of the subsection, Lemma~\ref{MBL2cor}(i) will allow us to choose
$\match'_{BL}$ and $\Entry_{RL}$ of the size required in (i). Property (ii) was discussed above.
We will just consider the first point of property (iv), as the second is similar.
Suppose $V$ is $3$-good and consider sets $X'$ and $Y'$ as in the statement.
We need to show that $(X',Y')_{G'}$  is $(\sqrt{\eps},d^2)$-super-regular.
If $V \cap \Entry_{RL}^{twin} \ne \emptyset$ this holds by definition of $X^*_{RL}$
since $V_{ex} \cap V_{exit} = \emptyset$ and $V_{ent}^+ \cap V_{entry}^+= \emptyset$ by~(ii)
(and so $V_{ex}\subseteq X'$ and $V^+_{ent}\subseteq Y'$), and since $V^+$ is $2$-good.
If $V \cap \Exit_{BL}  \ne \emptyset$ this holds by definition of $X^*_{BL}$,
since $V$ is $2$-fair, and similarly, if $V \cap \Entry_{BL}^{twin} \ne \emptyset$
this holds again by definition of $X^*_{BL}$, since $V^+$ is $2$-fair.
It remains to prove property (iii).
Suppose first that $V \in M_H^{RL}$. If $V^+$ is not $2$-good then the required matching
exists by the way we defined twins for $\Entry_{RL}$ in this case. On the other hand,
if $V^+$ is $2$-good then we can apply the super-regularity property (iv) just
established (which only used the fact that $V^+$ is $2$-good) and Lemma~\ref{regmatch}.
Next suppose that $V \in B$. Then $V_{exit}=V \cap \Exit_{BL}$ and
$V^+_{entry}=V^+ \cap \Exit_{BL}^{twin}$.
Thus if $V$ is not $2$-fair then the required matching
exists by the way we defined twins for $\match'_{BL}$ in this case.
On the other hand, if $V$ is $2$-fair then
we can apply the first point of property (iv), which only used the fact that $V$ is $2$-fair.
Similarly, when $V \in T$, then  $V_{exit}=V \cap \Entry_{BL}^{twin}$ and
$V^+_{entry}=V^+ \cap \Entry_{BL}$. If $V^+$ is not $2$-fair the required matching exists
by the construction in this case, whereas if $V^+$ is $2$-fair then
we can apply the second point of property (iv) with $(V,V^+)$ playing the role of
$(V^-,V)$,
which only uses the fact that $V^+$ (playing the role of $V$ in the second point) is $2$-fair.
\endproof

\subsection{Summary}

The auxiliary graph $H$ is decomposed into shifted components $L$ `left' and $R$ `right'
of size $k/2 \pm 19\eta k$ and a set $M_V$ of size $|M_V| < 17\eta k$.
This corresponds to a partition of $V(G) \sm V_0 = \wt{L} \cup \wt{R} \cup \wt{M}_V$.
The $1$-factor $F$ defines a partition $V(H) = T \cup B \cup M_H$, where
a cluster $V$ belongs to $T,B,M_H$ if and only if its successor $V^+$ belongs to $L,R,M_V$ respectively.
The shifted walk $W$ will use two types of transitions from $\wt{B}$ to $\wt{L}$.
One type is a pseudo-matching $\match'_{BL}$ from $\wt{B}$ to $\wt{L}$, matching $\Exit_{BL}$ to $\Entry_{BL}$.
The other type is a set $\Entry_{RL}$ of vertices in $\wt{M}_V^{RL} \sub \wt{M}_V$,
with the property that if $V \in M_V^{RL}$ then any $x \in V$ has
many inneighbours in $\wt{B}$ and any $y \in V^-$ has many outneighbours in $\wt{L}$.
We did not discuss transitions from $\wt{T}$ to $\wt{R}$, but these are obtained similarly
under the transformation $L \lra R$, $B \lra T$, etc.
Each vertex in these sets has a twin (or possibly two twins) that will be used
when $W$ traverses the cycle of $F$ containing it.
For any cluster $V$, the set of exit points from $V$ is $V_{exit}$ and
the set of entry points to $V$ is $V_{entry}$.
There exists a perfect matching from $V \sm V_{exit}$ to $V^+ \sm V^+_{entry}$.
The exceptional set $V_0$ is decomposed into $4$ parts $TR$, $TL$, $BR$ and $BL$,
where the first letter gives the in-type of a vertex and the second letter the out-type:
there is a matching $\match_T$ from $\wt{T}^*$ to vertices of in-type $T$ (and so on).
Technical complications are created by the possibility that a cluster may be full
(contain at least $\gamma m$ endpoints of $\match_{BL}$) or $M^{RL}$-full
(contain at least  $\gamma m$ endpoints of $\Entry_{RL}$).
A cluster $V$ is $\ell$-fair if no cluster at distance at most $\ell$ from $V$ is full,
$\ell$-excellent if no cluster at distance at most $\ell$ from $V$ is full or in $M=M_V\cup M_H$,
and $\ell$-good if no cluster at distance at most $\ell$ from $V$ is full or $M^{RL}$-full.
We have a set $X^* = X^*_{BL} \cup X^*_{RL}$ such that
whenever $V$ is $3$-good, we have flexibility to use any vertices avoiding these sets
in $V^-$, $V$ and $V^+$ (as well as avoiding the exits and entries already chosen),
while preserving super-regularity of the corresponding pairs in $F$.
Finally, the set $V^*_0$ of endpoints in $V(G) \sm V_0$ of the matchings $\match_T$ etc. only
uses $5$-excellent clusters and avoids $X^*$.

\section{The size of the pseudo-matching}

Our aim in this section is to prove the following lower bound on the size of our pseudo-matchings
$\match_{BL}$ and $\match_{TR}$.%
   \COMMENT{It seems at first as if this could appear earlier, but the argument uses
the exceptional set matchings $\match_T$ etc., which must be chosen after the pseudo-matchings.}

\begin{lemma}\label{MBL2cor}$ $
\begin{itemize}
\item[(i)] $|\match_{BL}|\ge \min\{|\wt{M}_V^{LR}|/2,\gamma^4 n\}-|\wt{M}_V^{RL}|-|V_0|$.
Moreover, if $|TR|>|BL|$ then $|\match_{BL}|\ge |TR|-|BL|-|\wt{M}_V^{RL}|+
\min\{|\wt{M}_V^{LR}|/2,\gamma^4 n\}$.
\item[(ii)] $|\match_{TR}|\ge \min\{|\wt{M}_V^{RL}|/2,\gamma^4 n\}-|\wt{M}_V^{LR}|-|V_0|$.
Moreover, if $|BL|>|TR|$ then $|\match_{TR}|\ge |BL|-|TR|-|\wt{M}_V^{LR}|+
\min\{|\wt{M}_V^{RL}|/2,\gamma^4 n\}$.
\end{itemize}
\end{lemma}

To prove this we first show that there are large sets $S_B \sub \wt{B}$
with many outneighbours in $\wt{L}$ and $S_L \sub \wt{L}$ with many inneighbours in $\wt{B}$.
Note that part (v) in the following lemma is not used in the proof of Lemma~\ref{MBL2cor}
but will be needed in the final section of the paper.

\begin{lemma}\label{MBL1}$ $
\begin{itemize}
\item[(i)] If $|TR|>|BL|$ there is $S_B \sub \wt{B}$ with $|S_B| \ge \beta n/100$,
such that every $x \in S_B$ satisfies
$$|N^+_G(x) \cap \wt{L}| \ge \deg_L := \frac{n}{2} - (|BL| + |BR| + |\wt{R}|)-|\wt{M}^{RL}_V|-|\wt{M}^{LR}_V|/4.$$
Furthermore, in any case, $\wt{B}$ contains a set $S^*_B$ of size $|S^*_B| \ge \beta n/100$,
such that every $x \in S^*_B$ satisfies
$|N^+_G(x) \cap \wt{L}| \ge \frac{n}{2} - |V_0|-|\wt{R}|-|\wt{M}^{RL}_V|-|\wt{M}^{LR}_V|/4$.
\item[(ii)] If $|TR|>|BL|$ there is $S_L \sub \wt{L}$ with $|S_L| \ge \beta n/100$,
such that every $x \in S_L$ satisfies
$$|N^-_G(x) \cap \wt{B}| \ge \deg_B := \frac{n}{2} - (|TL| + |BL| + |\wt{T}|)-|\wt{M}^{RL}_H|-|\wt{M}^{LR}_H|/4.$$
Furthermore, in any case, $\wt{L}$ contains a set $S^*_L$ of size $|S^*_L| \ge \beta n/100$,
such that every $x \in S^*_L$ satisfies
$|N^-_G(x) \cap \wt{B}| \ge \frac{n}{2} - |V_0|- |\wt{T}|-|\wt{M}^{RL}_H|-|\wt{M}^{LR}_H|/4$.
\item[(iii)] If $|BL|>|TR|$ there is $S_T \sub \wt{T}$ with $|S_T| \ge \beta n/100$,
such that every $x \in S_T$ satisfies
$$|N^+_G(x) \cap \wt{R}| \ge \deg_R := \frac{n}{2} - (|TL| + |TR| + |\wt{L}|)-|\wt{M}^{LR}_V|-|\wt{M}^{RL}_V|/4.$$
Furthermore, in any case, $\wt{T}$ contains a set $S^*_T$ of size $|S^*_T| \ge \beta n/100$,
such that every $x \in S^*_T$ satisfies
$|N^+_G(x) \cap \wt{R}| \ge \frac{n}{2} - |V_0|- |\wt{L}|-|\wt{M}^{LR}_V|-|\wt{M}^{RL}_V|/4$.
\item[(iv)] If $|BL|>|TR|$ there is $S_R \sub \wt{R}$ with $|S_R| \ge \beta n/100$,
such that every $x \in S_R$ satisfies
$$|N^-_G(x) \cap \wt{T}| \ge \deg_T := \frac{n}{2} - (|TR| + |BR| + |\wt{B}|)-|\wt{M}^{LR}_H|-|\wt{M}^{RL}_H|/4.$$
Furthermore, in any case, $\wt{R}$ contains a set $S^*_R$ of size $|S^*_R| \ge \beta n/100$,
such that every $x \in S^*_R$ satisfies
$|N^-_G(x) \cap \wt{T}| \ge \frac{n}{2} - |V_0|- |\wt{B}|-|\wt{M}^{LR}_H|-|\wt{M}^{RL}_H|/4$.
\item[(v)] Finally, suppose that $M_V^{RL}$, $TR$ and $BL$ are all empty.
\begin{itemize}
\item[$\bullet$] If $|\wt{L} \cup TL| \ge |\wt{B} \cup BR|$, then $\wt{B}$ contains a set $S_B$
of at least $\beta n/100$ vertices, each having at least $|\wt{M}_V^{LR}|/4$ outneighbours in $\wt{L} \cup TL$.
\item[$\bullet$] If $|\wt{L} \cup TL| \le |\wt{B} \cup BR|$, then $\wt{L}$ contains a set $S_L$
of at least $\beta n/100$ vertices, each having at least $|\wt{M}_V^{LR}|/4$ inneighbours in $\wt{B} \cup BR$.
\end{itemize}
\end{itemize}
\end{lemma}
\proof
Suppose that $|TR|>|BL|$. To prove (i), we first consider the case when $d^+_{(1 - \beta)n/2}(G) \geq n/2$.
Let $S'_B$ be as defined in Lemma~\ref{split}(v). Let $S_B$ be the set obtained from $S'_B$
by deleting the following vertices.%
  \COMMENT{** Is this all necessary? We choose the pseudo-matching before anything else,
so why do we take them into account when bounding its size? It seems that we don't need these
properties in the statement of the lemma...Answer: Later we argue that no $x \in S_B$
has an outneighbour in $TL \cup TR$ (since otherwise we can decrease $|TR|-|BL|$).
But for this we need that $x$ could be chosen as an endpoint of an edge in $\match_B$,
ie $x$ has to satisfy the properties in Lemma 32}
\begin{itemize}
\item The set $V_0^*$ of $2|V_0|\le 2d^{1/4} n$ endvertices in $V(G) \sm V_0$ of edges in
$\match_T$, $\match_B$, $\match_L$, $\match_R$.
\item All the at most $800\eta n$ vertices which lie in clusters
that are not $5$-excellent with respect to~$\match_{BL}$ or~$\match_{TR}$.
\item All the at most $2|V_0|k/(\gamma m/2)\le d^{1/5} n$ vertices which lie in clusters
containing at least $\gamma m/2$ vertices of $V_0^*$.
\item All the at most $12\eta' n$ vertices in $\wt{B}$ having more than
$|\wt{M}^{LR}_V|/4$ outneighbours in $\wt{M}^{LR}_V$ (see Lemma~\ref{middleedges}(ii)).
\end{itemize}
Thus $|S_B|\ge \beta n/100$. Now we make the following key use of the minimality of $|TR|-|BL|>0$.
We claim that any vertex $x\in S_B$ has outdegree at most $|BL| + |BR|$ in $V_0$.
Otherwise, there would be some edge $xy$ with $y \in TL \cup TR$.
But then we can change the in-type of $y$ to $B$ by deleting the edge in $\match_T$ incident to $y$
and adding the edge $xy$ to $\match_B$. Conditions (v) and (vi) in Lemma~\ref{excvs} will still hold,
since $S_B$ is disjoint from $V_0^*$ and only contains vertices in clusters
containing at most $\gamma m/2$ vertices of $V_0^*$.
Condition (ii) holds since $x \in \wt{B}^*$ by definition of $S_B$.
This reduces $||TR|-|BL||$, which contradicts the minimality condition in Lemma~\ref{excvs}(vii).
Therefore the claim holds.
Now recall that $\wt{R} \cup \wt{L} \cup \wt{M}^{LR}_V\cup \wt{M}^{RL}_V \cup V_0$ is a partition of $V(G)$.
Any $x \in S_B$ has at least $n/2$ outneighbours, of which at most $|\wt{R}|+|\wt{M}^{RL}_V|$ belong to
$\wt{R} \cup \wt{M}^{RL}_V$, at most $|\wt{M}^{LR}_V|/4$ belong to $\wt{M}^{LR}_V$
and at most $|BL| + |BR|$ belong to $V_0$. This shows that $S_B$ is a set as required in~(i).

Now consider the case when $d^+_{(1 - \beta)n/2}(G) < n/2$, and so
$d^-_{(1 - \beta)n/2}(G) \geq (1 +\beta)n/2$ by our degree assumptions.
Then $G$ has at least $(1 + \beta)n/2$ vertices of indegree at least $(1 +\beta)n/2$,
and by Lemma~\ref{propRL} at least
$(1 + \beta)n/2-|\wt{R}|-|V_0|-|\wt{M}_V|> \beta n/3$ of these belong to $\wt{L}$.
Let $A \sub \wt{L}$ be a set of $\beta n/3$ vertices with indegree
at least $(1 +\beta)n/2$. Note that every vertex in $A$ has indegree at least
$(1 +\beta)n/2+|\wt{B}|-n > \beta n/3$ in $\wt{B}$. Then we must have a set $S_B$ of at least
$\beta n/100$ vertices in $\wt{B}$, each having outdegree at least $\beta^3 n$ in $A$,
or we would have $\beta n/3 \cdot |A|\le E(\wt{B},A) \le \beta n/100 \cdot |A| + \beta^3 n |\wt{B}|$,
a contradiction. Then every vertex in $S_B$ has at least $\beta^3 n\ge \frac{n}{2} - |\wt{R}|$
outneighbours in $A\sub \wt{L}$, as required. This completes the proof of~(i) when $|TR|>|BL|$.

The argument for (i) when we do not have $|TR|>|BL|$ is the same, except that we no longer have
the minimality argument for $||TR|-|BL||$, so vertices in~$S^*_B$ may have all of~$V_0$ as outneighbours.
The arguments for~(ii)--(iv) are analogous, so we omit them.

Finally, suppose that $M_V^{RL}$, $TR$ and $BL$ are all empty, so that
$\wt{R} \cup \wt{L} \cup \wt{M}^{LR}_V \cup TL \cup BR$ is a partition of $V(G)$.
For the first point in (v), suppose that $|\wt{L} \cup TL| \ge |\wt{B} \cup BR|$.
Since $|\wt{B}|=|\wt{R}|$, every vertex $x$ in $S_B$ (defined as above)
has at least $n/2-|BR|-|\wt{B}|-|\wt{M}_V^{LR}|/4$ outneighbours in $\wt{L} \cup TL$.
By assumption, we have $|BR|+|\wt{B}|+|\wt{M}_V^{LR}|/2 \le n/2$, so
the number of outneighbours of $x$ in $\wt{L} \cup TL$ is at least $|\wt{M}_V^{LR}|/4$, as required.
The second point follows in the same way.\endproof

\noindent
{\em Proof of Lemma~\ref{MBL2cor}.}
Observe that all stated lower bounds are at most $\gamma^2 n$, so it is enough
to prove the existence of pseudo-matchings satisfying these bounds.
We will suppose that $|TR|>|BL|$ and prove the `moreover'
statement of~(i); the arguments for the other assertions are similar.
Define an auxiliary bipartite graph whose vertex classes are $\wt{B}$ and $\wt{L}$ by joining
a vertex $x\in \wt{B}$ to a vertex $y\in \wt{L}$ if $xy$ is an edge of~$G$.
Let~$J$ be the graph obtained from this bipartite graph by deleting all the edges whose
endvertices both lie in clusters having distance at most~4 in $F$ from~$M$.
Let $Q$ be the largest matching in~$J$.

\medskip

\noindent
\textbf{Case~1.} $|Q|\ge \gamma^2 n$.

\smallskip

\noindent Let~$Q'$ be a matching obtained from $Q$ by deleting  as
few edges as possible so as to ensure that every vertex of~$G$
belongs to at most one edge from~$Q'$. Note that every vertex of $G$
has indegree at most $1$ and outdegree at most $1$ in $Q$, so $Q$
considered as a subdigraph of $G$ is
a vertex-disjoint union of directed paths and cycles. Thus we can
retain at least $1/3$ of the edges of $Q$ in $Q'$ (with equality for
a disjoint union of directed triangles). By deleting further edges
if necessary we may assume that $|Q'|=\gamma^2 n/3$.

We claim that there is a submatching $Q''$ of $Q'$ of size at least $\gamma|Q'|/3$ such that no cluster is
full with respect to~$Q''$ (i.e.~every cluster contains at most $\gamma m$ endvertices of $Q''$).
To see that such a $Q''$ exists, consider the submatching $Q''$ obtained from~$Q'$ by
retaining every edge of~$Q'$ with probability $\gamma/2$ in $Q''$, independently of all
other edges of~$Q'$. Then for any cluster $V$, the expected number of endvertices of $Q''$ in $V$
is at most $\gamma m/2$,
and the expected size of $Q''$ is  $\gamma|Q'|/2$.
By Chernoff bounds we see that with high probability $Q''$ has the claimed properties.%
  \COMMENT{Actually, if there are matching edges with both endpoints in $V$, we cannot directly use
the binomial distribution. But for such edges, we fix one of its endvertices and then apply Chernoff
separately (i) to the fixed endvertices of these edges (ii) to those vertices which have exactly
one endpoint in the cluster (of course one could also use a single application of Azuma).
People who notice this subtlety will know how to fix it, so it's probably better not to go into this..) }

Note that $Q''$ is a pseudo-matching from~$\wt{B}$ to $\wt{L}$, as
by construction it is a matching, and by definition of $J$ every
edge in $Q''$ has an endvertex in a $4$-excellent cluster. Also,
since $|V_0| \le d^{1/4}n$ we have $|Q''|\ge \gamma^3 n/9 \ge
|V_0|+\gamma^4 n\ge |TR|-|BL|-|\wt{M}_V^{RL}|+\gamma^4 n$, as
required.

\medskip

\noindent
\textbf{Case~2.} $|Q|\le \gamma^2 n$.

\smallskip

\noindent
Let $A$ be a minimum vertex cover of~$J$. Then $|A|\le \gamma^2 n$ by K\"onig's theorem
(Proposition~\ref{konig}). Write $A_B:=A\cap \wt{B}$ and $A_L:=A\cap \wt{L}$.
We say that a cluster~$V$ is \emph{$A$-full} if it contains at least $\gamma m/3$ vertices from~$A$.
We say that $V$ is \emph{$A$-excellent} if no cluster of distance at most~4 from~$V$ on~$F$ is $A$-full
or lies in~$M$. Note that at most $\gamma^2 n/(\gamma m/3)=3\gamma n/m$ clusters are
$A$-full and thus by~(\ref{eq:middle}) all but at most $9(3\gamma n+|M|m)\le 350 \eta n$ vertices lie
in $A$-excellent clusters. Since $|TR|>|BL|$ we can construct the sets~$S_B$ and $S_L$
given by Lemma~\ref{MBL1}(i) and~(ii).
Let $S'_B$ be the set of all those vertices in~$S_B\sm A$ which lie in $A$-excellent clusters.
Thus $|S'_B| \ge |S_B|-\gamma^2 n-350\eta n\ge \beta n/101$.
Moreover, $N^+_G(S'_B)\cap \wt{L}\sub A_L$,
since none of the edges deleted in the construction of $J$ were incident to $S'_B$.

Now we greedily choose a matching $\match_1$ from $S'_B$ to $A_L\sub \wt{L}$
of size $\deg_L$ (defined in Lemma~\ref{MBL1}(i)) in such a way that every cluster
contains at most $\gamma m/3$ vertices on the $\wt{B}$-side of $J$.
To see that this is possible, note that at any stage in the process we have
excluded at most $|A|/(\gamma m/3) < \beta n/101 \le |S'_B|$ vertices in $S'_B$,
so we can always pick a suitable vertex $x$ in $S'_B$. Then, since we have chosen less
than $\deg_L$ vertices in $A_L$, we can choose an unused outneighbour of $x$ in $\wt{L}$
(which lies in the cover $A$, so in $A_L$).

Let $S'_L$ be the set of all those vertices in~$S_L\sm A$ which lie in $A$-excellent clusters
and are not endvertices of edges in~$\match_1$.
Then $|S'_L|\ge |S_L|-\gamma^2 n-350\eta n-2|\match_1|\ge \beta n/101$
and $N^-_G(S'_L)\cap \wt{B}\sub A_B$. As before, we can greedily choose a matching
$\match_2$ from $A_B\sub \wt{B}$ to $S'_L$ of size $\deg_B$ in such a way that
every cluster contains at most $\gamma m/3$ vertices on the $\wt{L}$-side of $J$.

Note that every $A$-excellent cluster is $4$-excellent with respect to $\match_1\cup \match_2$,
as it contains at most $\gamma m$ endvertices of edges from~$\match_1\cup \match_2$
(it is not $A$-full), and so is not full with respect to $\match_1\cup \match_2$.
Also, any edge $e$ in $\match_1\cup \match_2$ has one endvertex in~$A$ and one endvertex outside~$A$.
The endvertex outside $A$ is that in $S'_B$ (if $e\in \match_1$) or $S'_L$ (if $e\in \match_2$).
So the endvertex outside $A$ is not an endvertex of another edge from $\match_1\cup \match_2$ and
lies in a cluster which is $4$-excellent with respect to $\match_1\cup \match_2$.
We deduce that $\match_1\cup \match_2$ is a disjoint union of edges and directed paths of length $2$
satisfying the definition of a pseudo-matching from $\wt{B}$ to~$\wt{L}$.
Moreover, since $|\wt{M}^{RL}_V|=|\wt{M}^{RL}_H|$ and $|\wt{M}^{LR}_V|=|\wt{M}^{LR}_H|$, we have
\begin{align*}
|\match_1 \cup \match_2| & = \deg_L + \deg_B \\
& = n-(2|BL| + |BR| + |TL|+ |\wt{R}|+|\wt{T}|)-2|\wt{M}^{RL}_V|-2|\wt{M}^{LR}_V|/4\\
& = |TR|-|BL| + (n-|V_0|-|\wt{R}|-|\wt{L}|) -2|\wt{M}^{RL}_V|-|\wt{M}^{LR}_V|/2\\
& = |TR|-|BL|+|\wt{M}_V|-2|\wt{M}^{RL}_V|-|\wt{M}^{LR}_V|/2\\
& = |TR|-|BL|-|\wt{M}^{RL}_V|+|\wt{M}^{LR}_V|/2,
\end{align*}
as required.
The proof of the first statement of (i) is the same, except that
we use the `furthermore' statements of Lemma~\ref{MBL1}(i) and~(ii) in the final calculation
instead of working with $\deg_L$ and $\deg_B$.
\hfill{$\square$}


\section{Proof of Theorem~\ref{CKKO - Approximate Chvatal}}

In this section we use the matchings and sets constructed in Section~\ref{structure2} to
prove Theorem~\ref{CKKO - Approximate Chvatal}. We will assume that $H$ is not strongly
$\eta k$-connected, as we have already covered this case in Section~\ref{highconnect}
(although it could also be deduced from the arguments in this section).
Our strategy will depend on the value of $|TR|-|BL|$, and also on
the size of middle, as described by the cases $(\star)$ or $(\star\star)$ above.
We divide the proof into three subsections: the first covers the case
when $|TR| \ne |BL|$ and~$(\star)$ holds, the second when $(\star\star)$ holds,
and the third when $|TR|=|BL|$ and~$(\star)$ holds.

\subsection{The case when $|TR| \ne |BL|$ and~$(\star)$ holds.}\label{subsec:BLstar}
We will just give the argument for the case when $|TR|>|BL|$, as the other case is similar.
We recall that $(\star)$ is the case when $|\wt{M}| \le |V_0|/\gamma^3$ and $|C\cap M|<|C|/10$ for every
cycle $C$ of $F$. In this case we will use the pseudo-matching~$\match'_{BL}$ as well as the additional transitions
from~$\wt{B}$ to $\wt{L}$ which we get from $\Entry_{RL}\cup \Entry_{RL}^{twin}$.
We want to construct a walk $W$ with the same properties as in the proof of the
case when $H$ is highly connected.

Recall that both $H[L]$ and $H[R]$ are strongly $\eta'k/2$-connected (see Lemma~\ref{propRL}).
Then by arguing as in the proof of Lemma~\ref{disjointpaths} for the graph $H[L]$ instead of $H$ we deduce that for any two clusters $V,V'\in L$ we can find $\eta^{\prime 2} k/64$
shifted walks (with respect to~$R_{G''}$ and~$F$) from $V$ to $V'$ such that each walk traverses at
most $4/\eta'$ cycles from~$F$ and every cluster is internally used by at most one of these walks.
A similar statement holds for any two clusters in~$R$.
For any two clusters $V$ and~$V'$, we call a shifted walk (with respect to~$R_{G''}$ and~$F$) from $V$ to~$V'$
\emph{useful} if it traverses at most $4/\eta'$ cycles from~$F$ and if every cluster which is
internally used by the walk is $4$-excellent (note that $4$-excellence is defined with respect to
pseudo-matching~$\match'_{BL}$ in this case, and the pseudo-matching $\match_{TR}$ is irrelevant).
Since all but at most $400\eta k$ clusters are $4$-excellent (by ($\clubsuit$))
we have the following property.
 \textno Whenever $\V$ is a set of at most $\eta^{\prime 2}k/100$ clusters and $V,V'\in L$ there exists
a useful shifted walk from $V$ to $V'$ that does not internally use clusters in $\V$.
A similar statement holds for any two clusters in~$R$. &(\heartsuit)

We incorporate the vertices of the exceptional set $V_0$ using whichever edges
in the matchings $\match_T$, $\match_B$, $\match_L$, $\match_R$ correspond to their in-types and out-types.
Suppose for example that we have just visited a vertex $x$ of out-type $L$, arriving
via some edge in $\match_T$ or $\match_B$ (depending on the in-type of $x$)
and leaving to its outneighbour $x^+$ in the matching $\match_L$.
Then $x^+$ belongs to some cluster $U$ in $L$. Since $H[L]$ is highly connected
we can proceed to incorporate any vertex $y$ of in-type $T$ as follows.
Let $y^-$ be the inneighbour of $y$ in $\match_T$. Then $y^-$ belongs to some cluster $V$ of $T$,
and so by definition of $T$ the successor $V^+$ of $V$ on $F$ is a cluster of $L$.
Let $W'_{xy} = X_1 C_1 X_1^- X_2 C_2 X_2^- \ldots X_t C_t X_t^- X_{t+1}$ be a useful shifted walk
with $X_1 = U$ and $X_{t+1} = V^+$. Let $C_{t+1}$ be the cycle of $F$ containing $X_{t+1}=V^+$
and form $W_{xy}$ by appending the path in $C_{t+1}$ from $V^+$ to $V$. So for any cycle $C$
of $F$ the clusters of $C$ are visited equally often by $W_{xy}$. Then in the construction
of $W$ we can use the walk $W_{xy}$ to move from $x$ to $y$. Note that since we chose
$\match_T$, $\match_B$, $\match_L$, $\match_R$ to use at most $\gamma m$ vertices from any cluster and
since $|V_0|/(\gamma m) \ll \eta^{\prime 2} k$,
($\heartsuit$) implies that we can avoid using any cluster more than $3\gamma m$ times
(although we may visit a cluster more often).%
\COMMENT{we do need a 3 here -- 1 directly from the matching 1 for twins from the predecessor
1 for twins from the successor}

Thus we see that the structure of $H$ allows us to follow any vertex of out-type $L$ with
any vertex of in-type $T$, and similarly we can follow any vertex of out-type $R$ with
any vertex of in-type $B$. In particular, we can incorporate all vertices of type $TL$
sequentially, all vertices of type $BR$ sequentially, and vertices of type $BL$ or $TR$
can be incorporated in an alternating sequence, while there remain vertices of both types.
This explains the purpose of condition (vii) in Lemma~\ref{excvs}: choosing
$\lvert |TR| - |BL| \rvert$ to be minimal.

We order the vertices of $V_0$ as follows.  First, we list all vertices of type $TL$ (if any exist).
These will be followed by an arbitrary vertex of type $TR$ (which must exist as $|TR|>|BL| \geq 0$).
Then list all vertices of type $BR$ (if any exist). Then we
alternately list vertices of type $BL$ and $TR$ until all vertices of
type $BL$ are exhausted. Finally, we list all vertices of type $TR$
(if any remain). So the list by type has the form:
\[
TL, \ldots , TL\, |\, TR\, |\, BR , \ldots , BR\, |\, BL , TR ,
\ldots , BL , TR\, |\, TR, \ldots , TR.
\]
We can follow the procedure described above to incorporate all vertices in the list
apart from the final block of $|TR|-|BL|-1$ vertices of type $TR$.
At this point the above procedure would require a shifted walk from $R$ to $L$, which need not
exist. For these remaining vertices we will use the $|TR|-|BL|$ transitions from $\wt{B}$ to
$\wt{L}$ formed by the matching $\match'_{BL}$ and the vertices in $\Entry_{RL}\cup \Entry_{RL}^{twin}$.
(We need $|TR|-|BL|$ transitions rather than $|TR|-|BL|-1$ since we need to close the walk~$W$ after
incorporating the last exceptional vertex.)
Suppose we have just visited an exceptional vertex $a$ of type $TR$,
leaving to its outneighbour $a^+$ in the matching $\match_R$,
and we want to visit another vertex $b$ of type $TR$, with inneighbour $b^-$ in the matching $\match_T$.
Let $U$ be the cluster of $R$ containing $a^+$ and $V$ the cluster of $T$ containing $b^-$.
We pick an unused edge $xy$ of $\match'_{BL}$, where $x$ belongs to a cluster $X$ of $B$
and $y$ to a cluster $Y$ of $L$. Recall from Subsection~\ref{transitions}
that $x\in \exit_{BL}$ and $y\in \entry_{BL}$ have twins
$x^{twin}\in X^+\in R$ and $y^{twin}\in Y^-\in T$. By Lemma~\ref{twinnbs} there
are $X'\in B$ and $Y'\in L$ such that $x^{twin}$ has at least $d' m/4$
inneighbours in~$X'$, whereas $y^{twin}$ has at least $d'm/4$ outneighbours in~$Y'$.
We can also choose $X'$ and $Y'$ to be $4$-excellent, since
by Lemma~\ref{twinnbs} there are at least $\eta' k/4$ choices for both $X'$ and $Y'$,
and at most $400\eta k\ll \eta' k/4$ clusters are not $4$-excellent (by ($\clubsuit$)).
Choose a useful shifted walk $W_1$ from $U$ to the $F$-successor $(X')^+$ of~$X'$ and
a useful shifted walk $W_2$ from $Y'$ to the $F$-successor $V^+$ of $V$.
$W_1$ and $W_2$ exist by~($\heartsuit$), since $(X')^+\in R$ as $X'\in B$ and
$V^+\in L$ as $V\in T$. Thus, as illustrated in Figure~3,
we can form a segment of the walk $W$ linking $a$ to $b$ by first following
$W_1$ to $(X')^+$, then the path in $F$ from $(X')^+$ to $X'$, then the edge $X'X^+$,
then the path in $F$ from $X^+$ to $X$, then the edge $xy$, then the path in $F$ from $Y$ to $Y^-$,
then the edge $Y^-Y'$, then $W_2$ to~$V^+$, and finally the path in $F$ from $V^+$ to $V$.
When we are transforming our walk $W$ into a Hamilton cycle we will replace $X'X^+$ with an edge of~$G$
from some vertex in~$X'$ to~$x^{twin}$ and replace $Y^-Y'$ with an edge from~$y^{twin}$ to some vertex in~$Y'$.
So we say that $x^{twin}$ is a \emph{prescribed endvertex} for this particular occurrence of $X'X^+$
on~$W$ and that $y^{twin}$ is a \emph{prescribed endvertex} for this particular occurrence of $Y^-Y'$
on~$W$. The vertices $x$ and $y$ will be \emph{prescribed endvertices} for the edge $xy$ on~$W$.
(We will also define other prescribed endvertices on $W$, and if they are not endpoints of $\match'_{BL}$
then they will always be such that they have at least $d' m/4$ inneighbours in the previous cluster on $W$
or at least $d' m/4$ outneighbours in the next cluster on $W$. Note our eventual Hamilton cycle may
not follow the route connecting $x^{twin}$, $x$, $y$ and $y^{twin}$ used here since we will use
a rerouting procedure which is similar to that in the case when $H$ is highly connected.)

\begin{figure}\label{transition}
\includegraphics[scale=0.8]{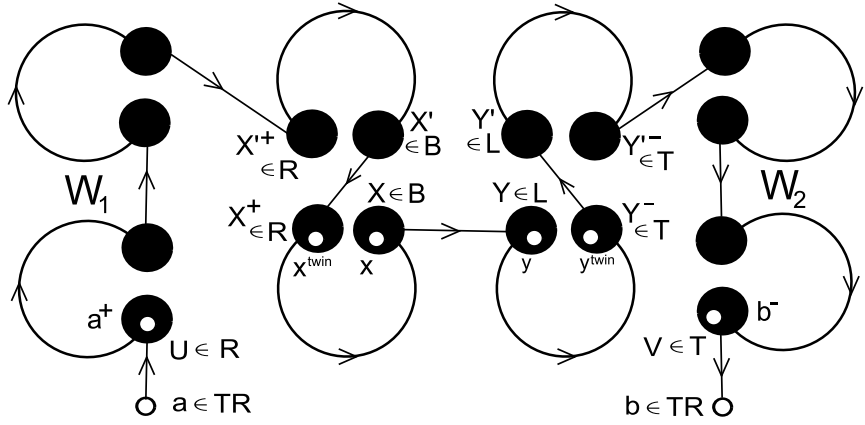}
\caption{Transitions using $\match_{BL}$}
\end{figure}

We use different matching edges from~$\match'_{BL}$ for different vertices of type $TR$.
After having used all of~$\match'_{BL}$, we use the $|TR|-|BL|-|\match'_{BL}|$ transitions
from~$\wt{B}$ to~$\wt{L}$ which we get from $\Entry_{RL}\cup \Entry_{RL}^{twin}$ instead.
If $X\in M_V^{RL}$ is the cluster containing $x\in \Entry_{RL}$ and so $X^-\in M_H^{RL}$ is the cluster
containing its twin $x^{twin}\in \Entry_{RL}^{twin}$, then, using Lemma~\ref{twinnbs} again,
we can choose $4$-excellent clusters $X'\in B$ and $X''\in L$
such that $x$ has at least $d' m/4$ inneighbours in~$X'$ whereas $x^{twin}$ has at least $d'm/4$
outneighbours in~$X''$. We then take~$W_1$ to be a useful walk from~$U$ to the $F$-successor $(X')^+$
of $X'$ and~$W_2$ to be a useful walk from $X''$ to~$V^+$.
Then when we are transforming $W$ into a Hamilton cycle we will replace $X'X$ with an edge of~$G$
from some vertex in~$X'$ to $x$ and
replace $X^-X''$ with an edge from~$x^{twin}$ to some vertex in~$X''$.
So we say that $x$ is a \emph{prescribed endvertex} for this particular occurrence of $X'X$
on~$W$ and that $x^{twin}$ is a \emph{prescribed endvertex} for this particular occurrence of $X^-X''$
on~$W$.

At the moment we have constructed a walk~$W$ which starts in the cluster~$U^*\in T$ containing the inneighbour
of the first exceptional vertex in our list, then goes into that vertex and then joins up all the exceptional vertices. After
visiting the last exceptional vertex of type $TR$, $W$ follows our last transition from~$\wt{B}$ to~$\wt{L}$
and ends in some $4$-excellent cluster $V^*\in L$.
(Using the same notation as above, if this last transition was a matching edge $xy\in \match'_{BL}$ then $V^*=Y'$,
and if it was a transition formed by a vertex $x\in \Entry_{RL}$ and its twin~$x^{twin}$ then $V^*=X''$.)
Say that a cluster $V$ is {\em nearly $4$-good} if $V$ is either $4$-good or at distance $1$ on $F$
from a $4$-good cluster. (A nearly $4$-good cluster is $3$-good, but not conversely.)
Note that since all walks above were useful, the walk $W$ constructed only {\em uses} nearly $4$-good clusters,
except when it uses a prescribed endvertex.
Using $(\heartsuit)$, it is easy to check that we can choose~$W$ in such a way that
every nearly $4$-good cluster is used at most $9\gamma m$ times.%
\COMMENT{3 from the exceptional matching, 3 from matchBL, 3 from the middle transitions}

Before closing up the walk~$W$, we have to enlarge it by some special walks
$W_L^{bad}$, $W_L^{good}$, $W_R^{bad}$ and $W_R^{good}$ which will ensure that
we can actually transform~$W$ into a Hamilton cycle of~$G$ (rather than a 1-factor).
We start by defining $W_L^{good}$. List the $4$-good clusters in $L$ as $V_1,\dots,V_s$,
for some $s$, where $V_1=V^*$. Choose useful shifted walks $W_i$ from $V_i$ to $V_{i+1}$,
for $i=1,\dots,s$, where $V_{s+1}:=V^*$. Let $W^{good}_L:=W_1\dots W_s$.
Then $W^{good}_L$ is a shifted walk from~$V^*$ to itself, which uses every $4$-good cluster in~$L$
at least once, and which only uses nearly $4$-good clusters.

Call a cycle in~$F$ \emph{bad} if it does not contain a $4$-good cluster
lying in~$L\cup R$. For every bad cycle we pick a cluster whose distance from~$M$ on~$F$ is at
least~$2$. (This is possible since we are in case $(\star)$, when no $F$-cycle significantly intersects~$M$.)
We let $\Z_L$ and $\Z_R$ be the sets of clusters in~$L$ and $R$ thus obtained.
Then no cluster $Z\in\Z_L \cup \Z_R$ is nearly $4$-good, since $Z$ has distance at least~2
from~$M$ on~$F$, and so the neighbours of~$Z$ on~$F$ cannot be $4$-good by definition of `bad'.
In particular, no cluster in $\Z_L \cup \Z_R$ is $4$-good,
so $|\Z_L|,|\Z_R|\le 30\gamma k$ by $(\diamondsuit)$.

The purpose of the walk $W_L^{bad}$ is to `fill up' each cluster
in~$\Z_L$: $W_L^{bad}$ will ensure that $W$ enters each such cluster precisely $m$ times.%
   \COMMENT{We could have focussed on exits instead.}
(Recall that this notion was defined in Section~\ref{techniques}.)
List the clusters in~$\Z_L$ that are not already entered $m$ times as $Z^1,\dots,Z^t$,
and let $a_i:=m-|Z^i_{entry}|$, where $Z^i_{entry}$ is as defined before the statement of Lemma~\ref{twinprops}.
(Recall that $Z^i$ is not nearly $4$-good and so~$W$ only enters $Z^i$ in vertices which are prescribed.)
Let $U^i\in T$ be the $F$-predecessor of~$Z^i$. Let $z^i_1,\dots,z^i_{a_i}$ be the vertices in
$Z^i\sm Z^i_{entry}$ and $u^i_1,\dots,u^i_{a_i}$ be the vertices in
$U^i\sm U^i_{exit}$. (Lemma~\ref{twinprops}(iii) implies that $|Z^i_{entry}|=|U^i_{exit}|$.)%
    \COMMENT{$U^i_{exit}$ is actually the set of exits in~$U^i$ of the walk~$W$ constructed so far
since $W$ is a shifted walk and so it exits $U^i$ precisely $|Z^i_{entry}|=|U^i_{exit}|$ times.}
Apply Lemma~\ref{twinnbs} to choose $4$-excellent clusters $Z^i_j\in T$ and $U^i_j\in L$ such
that $z^i_j$ has at least $d' m/4$ inneighbours in $Z^i_j$ and such that $u^i_j$ has at least $d'm/4$
outneighbours in $U^i_j$. We now find the following shifted walks:%
\begin{itemize}
\item For each $i=1,\dots t$ and each $j=1,\dots,a_i-1$ choose a useful
walk $W'_{i,j}$ from $U^i_j$ to the $F$-successor $(Z^i_{j+1})^+ \in L$ of $Z^i_{j+1}$.
\item Choose a useful walk $W''_0$ from $V^*$ to the $F$-successor $(Z^1_1)^+$ of $Z^1_1$.
\item For each $i=1,\dots,t-1$ choose a useful walk $W''_i$ from $U^i_{a_i}$ to the $F$-successor
$(Z^{i+1}_1)^+$ of $Z^{i+1}_1$.
\item Choose a useful walk $W''_t$ from $U^t_{a_t}$ to~$V^*$.%
   \COMMENT{Changed $(V^*)^+$ to $V^*$.}
\item Define the shifted walks $W''_{i,j}:=(Z^i_j)^+ C^i_j Z^i_j Z^i C^i U^i U^i_j$
for each $i=1,\dots t$ and each $j=1,\dots,a_i$, where
$C^i_j$ is the $F$-cycle containing $Z^i_j$ and where $C^i$ is the $F$-cycle containing~$Z^i$.
\end{itemize}
Then, as illustrated in Figure~4, we define
$$
W_L^{bad}:=W''_0 W''_{1,1} W'_{1,1} W''_{1,2} W'_{1,2}\dots
W'_{1,a_1-1}W''_{1,a_1}W''_1W''_{2,1} W'_{2,1}\dots W'_{2,a_2-1}\dots W'_{t,a_t-1}W''_{t,a_t}W''_t.
$$
So $W_L^{bad}$ is a shifted walk from~$V^*$ to itself.
When transforming our walk~$W$ into a Hamilton cycle of~$G$, for each $i=0,\dots,t-1$ we will replace
the edge $Z^i_jZ^i$ on $W''_{i,j}$ by an edge of~$G$ entering $z^i_j$ and the edge $U^iU^i_j$
on $W''_{i,j}$ by an edge leaving $u^i_j$. So we say that $z^i_j$, $u^i_j$ are \emph{prescribed endvertices}
for these particular occurrences of $Z^i_jZ^i$, $U^iU^i_j$.

\begin{figure}\label{bad}
\includegraphics[scale=0.6]{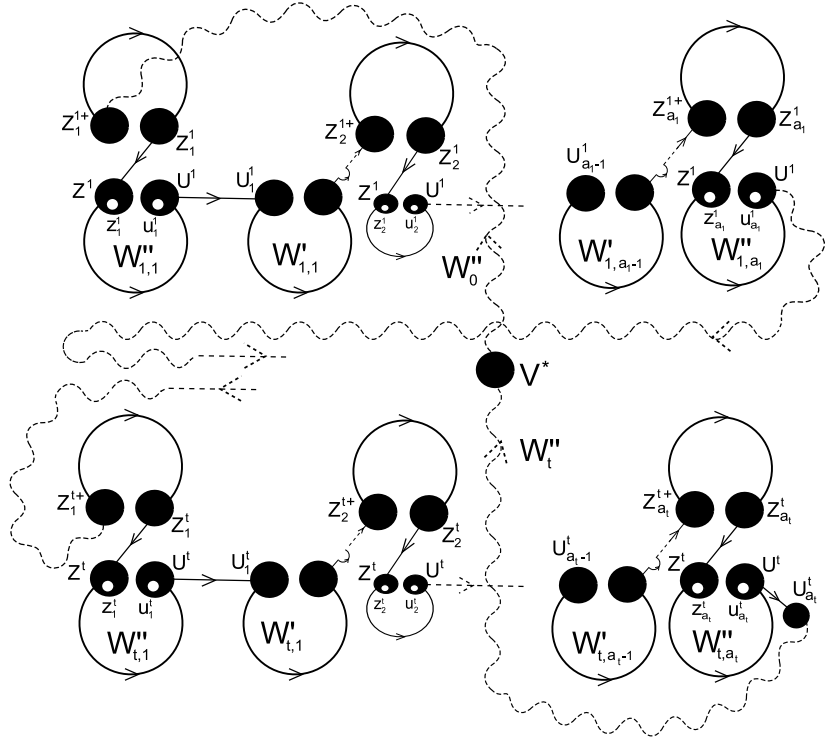}
\caption{A bad walk}
\end{figure}

Note that $W_L^{bad}$ is composed of $1 + \sum_{i=1}^t 2a_i \le 3tm \le 90\gamma n$ walks,
each using at most $8/\eta'$ clusters (by definition of useful walks). Also $W_L^{good}$
is composed of at most $|L| \le k$ further such walks. Using ($\heartsuit$), we can choose
$W^{good}_L$ and $W_L^{bad}$ such that the number of times they use every cluster
outside $\Z_L\cup \Z^-_L$ is at most
$\frac{(90\gamma n + k)(8/\eta')}{\eta^{\prime 2} k/100} < \gamma^{1/2}m$ (say).
Here $\Z^-_L$ is the set of predecessors of $\Z_L$ on $F$.

Let $V^{**}\in R$ be the cluster which contains the neighbour in $\match_R$ of the first
exceptional vertex of type $TR$ in our list (this exists, since $|TR|>|BL|$).
Then $V^{**}$ is $4$-excellent by Lemma~\ref{excvs}. Define walks $W_R^{good}$ and $W_R^{bad}$
similarly to $W_L^{good}$ and $W_L^{bad}$, where $V^{**}$ now plays the role of~$V^*$.
(So both $W_R^{good}$ and $W_R^{bad}$ are walks from~$V^{**}$ to itself.)
Now we construct our final walk, which we will also call $W$, as follows.
We start with our previous walk $W$ joining~$U^*$ to~$V^*$, then we add $W_L^{good}W_L^{bad}$
and replace the occurrence of~$V^{**}$ mentioned above with $W_R^{good}W_R^{bad}$.
We close up~$W$ by adding a useful walk from~$V^*$ to~$(U^*)^+$,
and then following the path in~$F$ from~$(U^*)^+$ to~$U^*$.
Our final walk $W$ has the properties listed below.

\begin{itemize}
\item[(a)] For each cycle $C$ of $F$, $W$ visits every cluster of $C$ the same number of times, say $m_C$.
\item[(b$''$)]
\begin{itemize}
\item[$\bullet$] $W$ enters every cluster of $R_{G''}$ at most $m$ times,
and thus $W$ exits every cluster at most~$m$ times.
\item[$\bullet$] If $V\in \Z_L\cup \Z_R$, then $W$ enters~$V$ precisely $m$ times and all the vertices of~$V$
are prescribed endvertices for these~$m$ entering edges of~$W$.
If $V\in \Z^-_L\cup \Z^-_R$, then $W$ exits~$V$ precisely $m$ times and all the vertices of~$V$
are prescribed endvertices for these~$m$ exiting edges of~$W$.
\item[$\bullet$]
If $V\notin \Z_L\cup \Z_R$ is not nearly $4$-good then $W$ enters~$V$ precisely $|V_{entry}|$ times and
the set $V_{entry}$ is the set of prescribed endvertices for all these $|V_{entry}|$ entering edges of~$W$.
Similarly, if $V\notin \Z^-_L\cup \Z^-_R$ is not nearly $4$-good then $W$ exits~$V$ precisely $|V_{exit}|$ times
and the set $V_{exit}$ is the set of prescribed endvertices for all these $|V_{exit}|$ exiting edges of~$W$.
\item[$\bullet$]
If $V$ is nearly $4$-good then $W$ enters~$V$ between $|V_{entry}|$ and $|V_{entry}|+2\gamma^{1/2} m$ times,
the set $V_{entry}$ is the set of prescribed endvertices for $|V_{entry}|$ of these entering edges of~$F$
and no vertex in~$V$ is a prescribed endvertex for the other entering edges of~$W$.
The analogue holds for the exits of~$W$ at~$V$.
\end{itemize}
\item[(c)] $W$ visits every vertex of $V_0$ exactly once.
\item[(d)] For each $x_i \in V_0$ we can choose an inneighbour $x_i^-$ in the cluster
preceding $x_i$ on $W$ and an outneighbour $x_i^+$ in the cluster following $x_i$ on $W$,
so that as $x_i$ ranges over $V_0$ all vertices $x_i^+$, $x_i^-$ are distinct.
\end{itemize}
Recall that $V_0^*$ denotes the set of all endvertices of the matching edges in
$\match_B\cup \match_T\cup \match_L\cup \match_R$ outside of~$V_0$.
Our aim now is to transform~$W$ into a Hamilton cycle of~$G$.
We start by fixing edges in~$G$ corresponding to all
those edges of~$W$ that lie in~$R_{G''}$ but not in~$F$. We first do this for all those occurrences
$VU\in E(R_{G''})\sm E(F)$ of edges
on~$W$ for which there is no prescribed endvertex. Note that the second and third conditions in (b$''$)
together imply that in this case both $V$ and $U$ must be nearly $4$-good.
Then, applying Lemma~\ref{regmatch} as in Section~\ref{highconnect}, we can
replace each such occurrence $VU$ by an edge from $V\sm (X^*\cup V_0^*)$
to $U\sm (X^*\cup V_0^*)$ in~$G$, so that all the edges of $G$ obtained in this
way are disjoint. We denote the set of edges obtained by $\cE_1$.
Next we choose the edge in~$G$ for all those occurrences
$VU\in E(R_{G''})\sm E(F)$ of edges on~$W$ which have a prescribed endvertex.
This can be achieved by the following greedy procedure.
Suppose that we have assigned the endvertex $u\in U$ to $VU$. Then $V$ will be $4$-excellent,
so by the last condition in~(b$''$) we have chosen at most $3\gamma^{1/2} m$ endvertices in~$V$
for edges constructed in previous steps. But $u$ has at least $d'm/4$ inneighbours in $V$,
where $d' \gg \gamma$, and $|V\cap (X^*\cup V_0^*)| \le 2\gamma m$,
so we can replace $VU$ by $vu$ for some $v\in V\sm (X^*\cup V_0^*)$
which is distinct from all the vertices chosen before.
(This is the point where we need to work with $R_{G''}$ instead of $R_{G'}$
-- we have $d \ll \gamma$, and so the above argument would fail for $R_{G'}$.)
We denote the set of edges obtained by $\cE_2$.

Let $\cE = \cE_1 \cup \cE_2 \cup \cE_3 \cup \cE_4$, where $\cE_3 = \match'_{BL}$
and $\cE_4 = \match_B\cup \match_T\cup \match_L\cup \match_R$.
(Note that $W$ used each edge in $\cE_3 \cup \cE_4$ precisely once.)
For each cluster $V$ let $V_{Exit}\sub V$ be the subset of all initial vertices of edges in~$\cE$
and let $V_{Entry}\sub V$ be the subset of all final vertices of edges in~$\cE$.
Then $V_{exit}\sub V_{Exit}$ and $V_{entry}\sub V_{Entry}$.
The following lemma provides useful properties of these fixed edges.

\begin{lemma}\label{exitentry}$ $
\begin{itemize}
\item[(i)] $\cE$ is a vertex-disjoint union of directed paths,
each having at least one endvertex in a $4$-excellent cluster.
\end{itemize}
Moreover, every cluster $V$ satisfies the following.
\begin{itemize}
\item[(ii)] $|V_{Exit}|=|V^+_{Entry}|$.
\item[(iii)] If $V$ is nearly $4$-good then $|V_{Exit}|,|V_{Entry}|\le 4\gamma^{1/2} m$,
$(V\cap X^*)\sm V_{entry}\sub V\sm V_{Entry}$
and $(V\cap X^*)\sm V_{exit}\sub V\sm V_{Exit}$.
Moreover, $V_{Exit}\cap V_{Entry}= V_{exit}\cap V_{entry}$.
\item[(iv)] If $V$ is nearly $4$-good then
the pairs $(V\sm V_{Exit},V^+\sm V^+_{Entry})_{G'}$
and $(V^-\sm V^-_{Exit},V\sm V_{Entry})_{G'}$ are $(\sqrt{\eps},d^2)$-super-regular.
\item[(v)] There is a perfect matching from $V\sm V_{Exit}$ to $V^+\sm V^+_{Entry}$.
\end{itemize}
\end{lemma}
\proof
By construction every vertex is the initial vertex of at most one edge in $\cE$
and the final vertex of at most one edge in $\cE$,
so $\cE$ is a disjoint union of directed paths and cycles.
To prove statement (i), we note that $\cE_1$ forms an independent set of edges in~$\cE$ and
each edge in $\cE_1$ has both endvertices in $4$-excellent clusters. Moreover, every edge in~$\cE_2$ has a
prescribed endvertex, and if $u \in U$ was prescribed for an edge $VU$ or $UV$ in $W$,
then $V$ is a $4$-excellent cluster and we chose $v \in V$
so that $uv$ is the only edge of $\cE$ containing $v$.
Thus any component of $\cE$ containing an edge from $\cE_1\cup \cE_2$ is a directed path
having at least one endvertex in a $4$-excellent cluster.
Also, $\cE_3$ and $\cE_4$ are vertex-disjoint, so any component of $\cE$
not containing an edge from $\cE_1\cup \cE_2$ is either a component of $\cE_3$,
which has the required property by definition  of `pseudo-matching',
or a directed path consisting of two edges of $\cE_4$,
which has the required property by Lemma~\ref{excvs}. Thus statement (i) holds.

Condition~(ii) follows immediately from our construction of~$W$. The first part of~(iii) follows
from the last part of~(b$''$) and the definition of $4$-good clusters. To check the remainder of~(iii),
note that the last part of~(b$''$) implies that the vertices in $V_{Exit}\sm V_{exit}$ and in
$V_{Entry}\sm V_{entry}$ are endvertices of edges in $\cE_1\cup\cE_4$ or non-prescribed endvertices of
edges in~$\cE_2$. We chose the endvertices of edges in~$\cE_1$ and the non-prescribed endvertices of edges
in $\cE_2$ to be disjoint from each other and from $X^*\cup V_0^*$.
Also, $V_0^*\cap X^*=\emptyset$ by definition of~$X^*$ in~(\ref{Xstar}) -- see Subsection~\ref{twins}.
Altogether, this implies the remainder of~(iii).

To prove the first part of~(iv), note that~(iii) and Lemma~\ref{twinprops}(iv) applied with
$X':=V\sm V_{Exit}$ and $Y':=V^+\sm V^+_{Entry}$
together imply that $(V\sm V_{Exit},V^+\sm V^+_{Entry})_{G'}$ is
$(\sqrt{\eps},d^2)$-super-regular. The second part of~(iv) can be proved similarly.
It remains to prove~(v). If $V$ is nearly $4$-good then~(v) follows
from~(ii) and~(iv) and Lemma~\ref{regmatch}.
If $V\in \Z^-_L\cup \Z^-_R$ then~(v) is trivial since $V_{Exit}=V$ and $V^+_{Entry}=V^+$
by the second condition in~(b$''$). In all other cases the third condition of~(b$''$)
implies that $V_{Exit}=V_{exit}$ and so $V^+_{Entry}=V^+_{entry}$ by~(ii) and the fact that
$|V_{exit}|=|V^+_{entry}|$ by Lemma~\ref{twinprops}(iii). Thus in these cases~(v) follows from
Lemma~\ref{twinprops}(iii).
\endproof

Let $\C$ denote a spanning subdigraph of~$G$ whose edge set consists of~$\cE$ together with
a perfect matching from $V\sm V_{Exit}$ to $V^+\sm V^+_{Entry}$ for every cluster~$V$.
Then $\C$ is a 1-factor of~$G$. We will show that by choosing a different perfect matching
from $V\sm V_{Exit}$ to $V^+\sm V^+_{Entry}$ for some clusters~$V$ if necessary, we can
ensure that $\C$ consists of only one cycle, i.e.~that $\C$ is a Hamilton cycle of~$G$.
First we show that if $U$ is $4$-good then we can merge most vertices of $U^-$, $U$ and $U^+$
into a single cycle of~$\C$.

\begin{lemma}\label{4goodcycles}
We can choose the perfect matchings from $V\sm V_{Exit}$ to $V^+\sm V^+_{Entry}$
(for all clusters~$V$) so that the following holds for every cluster~$U$.
\begin{itemize}
\item[(i)] If $U$ is nearly $4$-good then
all vertices in $U\sm (U_{Exit}\cap U_{Entry})$ lie on a common cycle $C_U$ of~$\C$.
\item[(ii)] If $U$ is $4$-good then $C_{U^-}=C_U=C_{U^+}$.
\end{itemize}
\end{lemma}
\proof
Recall that if $U$ is nearly $4$-good then $|U_{Exit}|,|U_{Entry}|\le 4\gamma^{1/2} m$ by Lemma~\ref{exitentry}(iii).
Then it is clear that (i) implies (ii), so it suffices to prove~(i).
We will consider each nearly $4$-good cluster~$U$ in turn and show that we can
change the perfect matchings from $U^-\sm U^-_{Exit}$ to $U\sm U_{Entry}$ and from
$U\sm U_{Exit}$ to $U^+\sm U^+_{Entry}$ to ensure that $U$ satisfies the conclusions of the lemma,
and moreover, every $V$ satisfying the conclusions continues to satisfy them.
First we choose the perfect matching from $U^-\sm U^-_{Exit}$ to
$U\sm U_{Entry}$ to achieve the following.
\begin{enumerate}
\item All vertices which were on a common cycle in~$\C$ are still on a common cycle.
\item All vertices in $(U^-\sm U^-_{Exit})\cup (U\sm U_{Entry})$ lie on a
common cycle in~$\C$.
\end{enumerate}
Since $(U^-\sm U^-_{Exit},U\sm U_{Entry})_{G'}$ is
$(\sqrt{\eps},d^2)$-super-regular by Lemma~\ref{exitentry}(iv),
this can be achieved by the same argument used to prove statement ($\dagger$)
at the end of Section~\ref{highconnect}.
Next we apply the same argument to merge all the cycles in (the new)~$\C$
meeting $U^+\sm U^+_{Entry}$ into a single cycle, which then also contains all the vertices
in $U\sm U_{Exit}$. Since $|U_{Exit}|,|U_{Entry}|\le 4\gamma^{1/2} m$ by Lemma~\ref{exitentry}(iii),
we have that $U_{Exit}\cup U_{Entry}\neq U$. Thus the 1-factor~$\C$ obtained in this way satisfies~(1) and~(i)
for the cluster~$U$. Continuing in this way for all $4$-good clusters and their $F$-neighbours
yields a 1-factor~$\C$ as required in~(i).
\endproof

We will now show that any 1-factor~$\C$ as in Lemma~\ref{4goodcycles} must consist of a single cycle
(and thus must be a Hamilton cycle of~$G$).

\begin{lemma}\label{singlecycle}
Every $1$-factor~$\C$ of~$G$ as in Lemma~\ref{4goodcycles} satisfies the following conditions.
\begin{itemize}
\item[(i)] For every cycle $C$ of~$\C$ there exists a $4$-good cluster $U\in L\cup R$
such that $C=C_U$, i.e. $C$ contains all vertices in $U\sm (U_{Exit}\cap U_{Entry})$.
\item[(ii)] There is one cycle $\C_L$ in $\C$ which contains $U\sm (U_{Exit}\cap U_{Entry})$
for every $4$-good cluster $U\in L$. Similarly, there is some cycle $\C_R$ in $\C$ which contains
$U\sm (U_{Exit}\cap U_{Entry})$ for every $4$-good cluster $U\in R$.
\item[(iii)] $\C=\C_L=\C_R$ (and thus $\C$ is a Hamilton cycle of~$G$).
\end{itemize}
\end{lemma}
\proof
First consider the case when $C$ contains at least one edge in~$\cE$,
and let $P$ be the longest subpath of~$C$ which only consists of edges in~$\cE$.
By Lemma~\ref{exitentry}(i) $P$ has an endvertex $x$ lying in a $4$-excellent (and so $4$-good)
cluster~$U$. But $x$ cannot lie in both $U_{Exit}$ and $U_{Entry}$ by the maximality of the path.
Therefore $x \in C_U$, and so $C=C_U$.
Now suppose that $C$ does not contain any edges of~$\cE$.
This means that there is some cycle $C'\in F$ such that $C$ `winds around'~$C'$,
i.e. it only uses clusters in $C'$ and in each step moves from $V$ to $V^+$ for some $V \in C'$.
We claim that $C'$ cannot be bad. Otherwise, it would contain a cluster~$V\in \Z_L\cup \Z_R$.
But then the second part of (b$''$) implies that $V_{Entry}=V$, so $C$ cannot `wind around'~$C'$.
(The purpose of the walks $W_L^{bad}$ and $W_R^{bad}$ was to exclude this possibility.)
Thus $C'$ is not bad, so contains a $4$-good cluster~$U\in L\cup R$,
which $C$ must meet in at least one vertex, $u$ say.
But $u\notin U_{Exit}\cup U_{Entry}$ (otherwise one edge at $u$ on $C$ would lie in~$\cE$).
Now Lemma~\ref{4goodcycles}(i) implies that $C$ contains all vertices in
$U\sm (U_{Exit}\cap U_{Entry})$, as required.

To prove the first part of~(ii), consider the walk $W_L^{good}$ which connected all $4$-good
clusters in~$L$. Let $W_L^{good}= X_1 C_1 X_1^- X_2 C_2 X_2^- \ldots X_t C_t X^-_tX_{t+1}$,
where $X_1=X_{t+1}=V^*$. Then each $4$-good cluster in~$L$ appears at least once in $X_1,\dots,X_{t+1}$,
and $W_L^{good}$ only uses nearly $4$-good clusters.
Let $x^-_ix_{i+1}$ be the edge in~$\cE$ that we have chosen for the edge $X^-_iX_{i+1}$ on~$W_L^{good}$.
As neither $x^-_i$ nor $x_{i+1}$ was a prescribed endvertex for $X^-_iX_{i+1}$,
we have $x^-_i\in (X^-_i)_{Exit}\sm (X^-_i)_{exit}$ and $x_{i}\in (X_{i})_{Entry}\sm (X_i)_{entry}$.
Thus $x^-_i\notin (X^-_i)_{Entry}$ and  $x_{i}\notin (X_{i})_{Exit}$ by Lemma~\ref{exitentry}(iii).
So Lemma~\ref{4goodcycles}(ii) implies that for each $i=2,\dots, t$ the%
   \COMMENT{we really need to start at 2 here, at least that's much easier to check...}
vertices $x^-_i$ and $x_{i}$ lie on the same cycle of~$\C$.
Trivially, $x_i^-$ and $x_{i+1}$ also lie on the same cycle for each $i=1,\dots, t$.
This means that all of
$x_2,\dots,x_{t+1},x^-_1,\dots,x^-_t$ lie on the same cycle of~$\C$, which we will call $\C_L$.
This in turn implies that~$\C_L$ contains $U\sm (U_{Exit}\cap U_{Entry})$ for every $4$-good
cluster $U\in L$. Indeed, $U=X_i$ for some $i=2,\dots,t+1$ and $x_i\in U_{Entry}\sm U_{Exit}
\sub U\sm (U_{Exit}\cap U_{Entry})$. As $x_i\in \C_L$, Lemma~\ref{4goodcycles}(i)
implies that $\C_L$ contains all vertices in $U\sm (U_{Exit}\cap U_{Entry})$.
A similar argument for $W_R^{good}$ establishes the existence of~$\C_R$.

To verify~(iii), consider an exceptional vertex $x$ of type $TR$ (which exists since
we are assuming that $|TR|>|BL|$). Let $x^-$ and $x^+$ be the neighbours of $x$ on the cycle
$C\in\C$ which contains $x$. Let $X$ be the cluster containing $x^-$ and $X'$ be the cluster
containing $x^+$. By Lemma~\ref{excvs}, $X\in T$, $X'\in R$ and both $X$ and $X'$ are $5$-excellent
(and thus $4$-good). Since $x^+\in X'_{Entry}\sm X'_{entry}$ and so
$x^+\in X'\sm (X'_{Exit}\cap X'_{Entry})$ by Lemma~\ref{exitentry}(iii),
we must have $C=\C_R$. But on the other hand,
$x^-\in X\sm (X_{Exit}\cap X_{Entry})$ and $X^+ \in L$ is $4$-good (since $X$ is $5$-excellent).
Together with Lemma~\ref{4goodcycles}(ii) this implies that $C$ contains
$X\sm (X_{Exit}\cap X_{Entry})$, i.e.~$C=\C_L=\C_R$. Together with~(i) this now implies
that $C=\C$, as required.
\endproof


\subsection{The case when~$(\star\star)$ holds.}\label{subsec:starstar}
Recall that $(\star\star)$ is the case when $|\wt{M}|\ge |V_0|/\gamma^3$.
We only consider the case when $|TR|\ge |BL|$, as the argument for the other case is similar.
Let $F_{RL}$ denote the set of all those cycles in~$F$ which avoid all the clusters in~$L\cup R$ and
contain more clusters from $M_V^{RL}$ than from $M_V^{LR}$.
Let $F_{LR}$ denote the set of all other cycles in~$F$ which avoid all the clusters in~$L\cup R$.

We divide the argument in this subsection into two cases.
The main case is when $|TR|-|BL|> |F_{RL}|+|F_{LR}|$.
We start by showing that we have at least $\frac{|TR|-|BL|}{10\gamma^3}$ transitions from~$\wt{B}$ to~$\wt{L}$.
Note that $|\wt{M}_V|\ge |\wt{M}|/2 \ge |V_0|/(2\gamma^3)$.
If $|\wt{M}_V^{RL}|\ge |V_0|/(4\gamma^3)$ then, since $|TR|-|BL| \le |V_0|$,
we can use vertices in $\wt{M}_V^{RL}\cup \wt{M}_H^{RL}$ as in Subsection~\ref{twins}
to obtain the required transitions. On the other hand, if $|\wt{M}_V^{RL}|\le |V_0|/(4\gamma^3)$
then $|\wt{M}_V^{LR}|\ge |V_0|/(4\gamma^3)$. Applying Lemma~\ref{MBL2cor}(i)
and recalling $|V_0| \le d^{1/4}n \ll \gamma^4n$ we obtain
$$
|\match_{BL}|\ge \min\{|\wt{M}_V^{LR}|/2,\gamma^4n\}-|\wt{M}_V^{RL}|-|V_0|
\ge \frac{|V_0|}{8\gamma^3}-|\wt{M}_V^{RL}|-|V_0|
\ge \frac{|TR|-|BL|}{10\gamma^3}-|\wt{M}_V^{RL}|.
$$
Thus the pseudo-matching $\match_{BL}$ and the vertices in $\wt{M}_V^{RL}\cup \wt{M}_H^{RL}$ together
give at least $\frac{|TR|-|BL|}{10\gamma^3}$ transitions from~$\wt{B}$ to~$\wt{L}$.

We claim we can choose a sub-pseudo-matching~$\match'_{BL}$ of $\match_{BL}$
and a set $\Entry_{RL}\sub \wt{M}_V^{RL}$ with the following properties.
\begin{itemize}
\item[(i)] $|\match'_{BL}|+|\Entry_{RL}|=|TR|-|BL|+|F_{LR}|$,
\item[(ii)] No cluster contains more than $\gamma m/2$ endpoints of $\match'_{BL}$
or more than $\gamma m/2$ vertices in $\Entry_{RL}$.
\item[(iii)] Every cycle in $F_{RL}$ contains at least one vertex of $\Entry_{RL}$.
\end{itemize}
To see this, we first choose a vertex in $M_V^{RL}$ on every cycle in $F_{RL}$ to include in $\Entry_{RL}$,
which is possible since $|TR|-|BL| \ge |F_{RL}|$.
Next we arbitrarily discard one edge from each $2$-edge path in $\match_{BL}$
to obtain a matching, and then consider a random submatching in which
every edge is retained with probability $\gamma/4$.
As in Case~1 of the proof of Lemma~\ref{MBL2cor}, with high probability
we obtain a submatching of size at least $\frac{\gamma}{9}|\match_{BL}|$
with at most $\gamma m/2$ endpoints in any cluster.
We also arbitrarily choose $\gamma m/2$ vertices in each cluster of $M_V^{RL}$.
Then we still have at least $\gamma^{-1}(|TR|-|BL|)$ transitions.
Since $|TR|-|BL| \ge |F_{LR}|$, we can arbitrarily choose some
of these transitions so that $|\match'_{BL}|+|\Entry_{RL}|=|TR|-|BL|+|F_{LR}|$.

Note that there are no clusters which are $M^{RL}$-full or full with respect to~$\match'_{BL}$.
In particular, every cluster is $4$-good.
Next we choose twins as in Subsection~\ref{twins} so that
the properties in Lemma~\ref{twinprops}(ii)--(iv) hold.
Thus we obtain sets $\Exit_{BL}$, $\Entry_{BL}$, $\Exit_{BL}^{twin}$, $\Entry_{BL}^{twin}$,
$\Entry_{RL}$, $\Entry_{RL}^{twin}$ as before.

We now proceed similarly as in Subsection~\ref{subsec:BLstar}, forming a walk $W$
that incorporates all the exceptional vertices and uses $|TR|-|BL|$ transitions,
ending in some $4$-excellent cluster $V^*\in L$.
Since all clusters are $4$-good, the bad cycles are precisely those in~$F_{LR}\cup F_{RL}$.
Now we cannot construct the walks $W_L^{bad}$ and $W_R^{bad}$ as before,
since the bad cycles avoid $L\cup R$. Instead, we enlarge $W$ by including a walk $W_{LR}$
which `connects' all the cycles in~$F_{LR}$.
Suppose that $C_1,\dots,C_t$ are the cycles in~$F_{LR}$ and choose a cluster $V_i\in M_H^{LR}$
on each~$C_i$. Lemma~\ref{LRandRLmiddle}(i) implies that, in~$R_{G''}$, each $V^+_i$ has many
(at least $\beta k/4$) $4$-excellent inneighbours in~$T$,
while each $V_i$ has many $4$-excellent outneighbours in~$R$.
We pick $4$-excellent inneighbours $X_i\in T$ of $V^+_i$
and $4$-excellent outneighbours $Y_i \in R$ of $V_i$ for each~$i$.
Now we construct $W_{LR}$ as follows.
We start at $V^*\in L$, follow a useful shifted walk to $X^+_1 \in L$,
then the path in~$F$ from~$X^+_1$ to~$X_1\in T$ and then use the edge $X_1V^+_1$.
Next we wind around $C_1$ to $V_1$, use the edge $V_1Y_1$ and
follow a useful shifted walk from $Y_1$ to one of the $|F_{LR}|$ transitions from~$\wt{B}$ to~$\wt{L}$
that we have not yet used to move back to~$L$. We continue in this way until we have `connected'
all the~$C_i$. Finally, we close $W_{LR}$ by following a useful shifted walk back to~$V^*$.

As in Subsection~\ref{subsec:BLstar}, we fix the edges $\mc{E}$ and choose matchings
from $V \sm V_{Exit}$ to $V^+ \sm V^+_{Entry}$ for each cluster $V$ to obtain a $1$-factor $\mc{C}$.
Note that by construction every vertex outside of $V_0$ is incident to at most one edge of $\mc{E}$,
so $V_{Exit}\cap V_{Entry}=\emptyset$ for each cluster~$V$.
Lemmas~\ref{exitentry} and~\ref{4goodcycles} still hold, but instead of Lemma~\ref{singlecycle}(i)
we now only have that for every cycle $C$ of~$\C$ there exists a cluster $U$
(which is automatically $4$-good) such that $C$ contains all the vertices in $U$.
We then deduce that $\mc{C}$ has a cycle $\mc{C}_L$ containing all vertices in clusters of $L$
and a cycle $\mc{C}_R$ containing all vertices in clusters of $R$.
Moreover, since we use at least one transition from $\wt{B}$ to $\wt{L}$ we have
$\mc{C}_L=\mc{C}_R:=\mc{C}'$. Lemma~\ref{4goodcycles} now implies that
for every cycle $D$ in~$F$ there is a cycle $C$ in~$\C$ such that $C$ contains all
vertices belonging to clusters in~$D$. In particular, $\mc{C}'$ contains all vertices
belonging to clusters which lie on an $F$-cycle that intersects~$L\cup R$.

Moreover, if $C\neq \mc{C}'$ is another cycle in~$\C$, then there must be a cycle $D$ in $F$
such that $D$ only consists of clusters from~$M_V$ and such that $C$ contains all
vertices in $U$ for all clusters $U$ on~$D$. If $D\in F_{RL}$ then our choice of $\Entry_{RL}$ implies
that some such cluster~$U$ on~$D$ contains a vertex $x\in\Entry_{RL}$. The inneighbour~$y$ of~$x$ in~$\C$ lies
in some cluster $Y\in B$ and so $Y^+$ in $R$. But $\mc{C}'$ contains all vertices belonging to clusters that
lie on an $F$-cycle which intersects $R$. So $\mc{C}'$ contains all vertices in $Y$, and thus contains~$x$.
This shows that $\mc{C}'$ contains all those vertices which belong to clusters lying on
cycles from~$F_{RL}$. On the other hand, the walk $W_{LR}$ ensures that
$\C'$ also contains all those vertices which belong to clusters lying on
cycles from~$F_{LR}$. Altogether this shows that $\mc{C}'=\C$ is a Hamilton cycle, as required.

It remains to consider the case when $0\le |TR|-|BL| \le  |F_{RL}|+|F_{LR}|$.
We claim that there are at least $m/4$ transitions from $\wt{T}$ to $\wt{R}$ and
at least $m/4$ transitions from $\wt{B}$ to $\wt{L}$.
If $M_V^{RL}\neq \emptyset$, then we can use the vertices in $\wt{M}_V^{RL}\cup \wt{M}_H^{RL}$ to get
at least $m$ transitions from $\wt{B}$ to~$\wt{L}$. On the other hand,
if $M_V^{RL}= \emptyset$ then $|\wt{M}_V^{LR}| \ge |\wt{M}|/2 \ge |V_0|/(2\gamma^3)$ by $(\star\star)$,
so Lemma~\ref{MBL2cor}(i) implies the existence of a pseudo-matching
from~$\wt{B}$ to~$\wt{L}$ of size at least
$\min\{|\wt{M}_V^{LR}|/2,\gamma^4n\}-|V_0|\ge \min\{|\wt{M}_V^{LR}|/4,\gamma^4n/2\}\ge m/4$.
Similarly, if $M_V^{LR}\neq \emptyset$ then the vertices in $\wt{M}_V^{LR}\cup \wt{M}_H^{LR}$
give at least $m$ transitions from $\wt{T}$ to~$\wt{R}$. On the other hand,
if $M_V^{LR}= \emptyset$ then Lemma~\ref{MBL2cor}(ii) implies the existence of a
pseudo-matching from~$\wt{T}$ to~$\wt{R}$ of size at least $m/4$.
Thus in all cases, we have at least $m/4$ transitions in both
directions. We can use these transitions to argue
similarly as in the first case when $|TR|-|BL| >  |F_{RL}|+|F_{LR}|$,
but this time we also include a walk $W_{RL}$ into~$W$ which `connects' all the cycles in $F_{RL}$.
Since $|TR|-|BL| \le |F_{RL}|+|F_{LR}| \le k \ll m$ there are more than enough transitions.
Moreover, if $|F_{RL}|+|F_{LR}|=0$, then we also make sure that $W$ follows at least
one transition from $\wt{B}$ to $\wt{L}$ (and thus at least one transition from $\wt{T}$ to~$\wt{R}$).
Then we find a Hamilton cycle by the same argument as above.


\subsection{The case when $|TR|=|BL|$ and~$(\star)$ holds.}
If $|TR|=|BL| \ge 1$ then we can use the same procedure as in Subsection~\ref{subsec:BLstar},
with no need to use any edges from $\match_{BL}$ or $\match_{TR}$.
For the remainder of the proof we consider the case when $|TR|=|BL|=0$.
In this case, our list of the vertices of $V_0$
has all vertices of type $TL$ followed by all vertices of type $BR$,
so we will need to make a transition from incorporating vertices of type $TL$ to type $BR$ and
then another transition back from type $BR$ to type $TL$, i.e.~we need one transition from~$\wt{T}$ to~$\wt{R}$
and another one from $\wt{B}$ to~$\wt{L}$. If $M_V^{LR}$ and $M_V^{RL}$ are both non-empty, then a similar argument as
in the second case of the previous subsection implies that there are at least $m/4$ transitions
from~$\wt{B}$ to~$\wt{L}$ and at least $m/4$ transitions from~$\wt{T}$ to~$\wt{R}$.

Thus we may suppose that at least one of $M_V^{LR},M_V^{RL}$ is empty. We only consider the case when
$M_V^{RL}=\emptyset$ (and $M_V^{LR}$ could be empty or non-empty), as the other case is similar.
Let $x_1$ and $x_2$ be the first and last
vertices of the $TL$ list and $y_1$ and $y_2$ the first and last vertices
of the $BR$ list. Let $X_2 \in L$ be the cluster containing the outneighbour of $x_2$ in $\match_L$,
$Y_1 \in B$ the cluster containing the inneighbour of $y_1$ in $\match_B$,
$Y_2 \in R$ the cluster containing the outneighbour of $y_2$ in $\match_R$, and
$X_1 \in T$ the cluster containing the inneighbour of $x_1$ in $\match_T$.
We need to construct portions of our walk $W$ that link $X_2$ to $Y_1$ and $Y_2$ to $X_1$.
(If for instance the $TL$ list is empty, we take $X_1$ to be an arbitrary cluster in $T$.)
We consider two subcases according to whether there is a cycle of
$F$ that contains both a cluster of $L$ and a cluster of $R$.

\medskip

\noindent
\textbf{Case~1.} \emph{There is a cycle $C$ of $F$ containing a cluster $X \in L$ and a cluster $Y \in R$.}

\smallskip

\noindent
In this case we can reroute along $C$ to construct the required transitions.
Let $W_1$ be a shifted walk from $Y_2$ to $Y_1^+$. This exists since $Y_2,Y_1^+\in R$.
Also, since $C$ contains $Y \in R$ we can choose the walk $W_1$ to go via $Y$,
and $C$ will be one of the cycles traversed. Similarly we can choose a shifted walk $W_2$ from $X_2$ to
$X_1^+$, and we can choose $W_2$ to go via $X$, so it also traverses $C$.
Now we construct the portions of the walk $W$ joining $x_2$ to $y_1$
and $y_2$ to $x_1$ as follows. To join $x_2$ to $y_1$ we start at $X_2$,
follow $W_2$ until it reaches $X$, then follow $C$ round to $Y^-$, and then switch to $W_1$, which
takes us to $Y_1^+$, and we end by traversing a cycle of $F$ to reach $Y_1$.
This is balanced with respect to all cycles of $F$
except for $C$, where we have only used the portion from $X$ to $Y^-$.
To join $y_2$ to $x_1$ we start at $Y_2$, follow $W_1$ until it reaches $Y$,
then follow $C$ round to $X^-$, and then switch to $W_2$, which we follow
to $X_1^+$, then traverse a cycle of $F$ to reach $X_1$.
This is balanced with respect to all cycles
of $F$ except for $C$, where it only uses the portion from $Y$ to $X^-$.
But this is exactly the missing portion from the first transition, so in
combination they are balanced with respect to all cycles of $F$.
This scenario is illustrated in Figure~5.
\begin{figure}\label{reroute}
\includegraphics[scale=0.5]{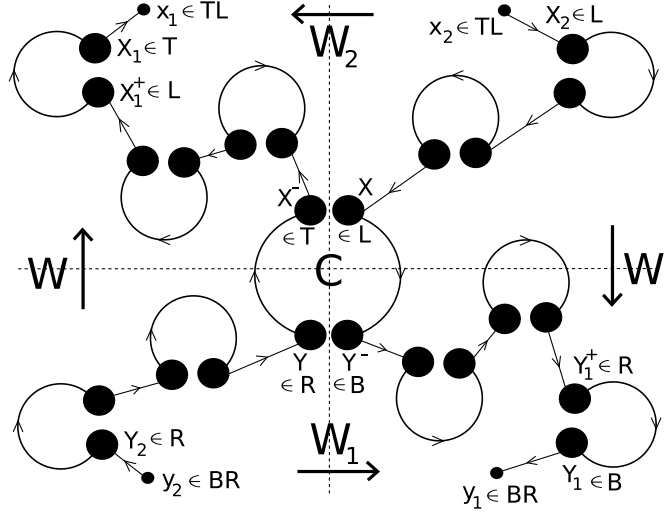}
\caption{Rerouting a cycle with both left and right clusters}
\end{figure}

\medskip

\noindent
\textbf{Case~2.} \emph{No cycle of $F$ contains clusters from both $L$ and $R$.}

\smallskip

\noindent
%
First we observe that in this case we have $B \subseteq R \cup M_V$ and $T \subseteq L \cup M_H$.
To see the first inclusion, note that if $U \in B$ then $U^+ \in R$, so we cannot have
$U \in L$ by our assumption for this case. The second inclusion is similar.
Since $|\wt{L}|,|\wt{R}|=n/2\pm 19\eta n$ by Lemma~\ref{propRL}
and since $\delta(G)\ge \beta n$, every vertex in~$V_0$ has either at least $\beta n/3$ inneighbours
in~$\wt{T}$ or at least $\beta n/3$ inneighbours in~$\wt{B}$ (and similarly for outneighbours in~$\wt{L}$
and $\wt{R}$). So by swapping the types of the exceptional vertices between $TL$ and $BR$ if necessary, we may assume that each
$v\in TL$ has either at least $\beta n/3$ inneighbours in~$\wt{T}$ or at least $\beta n/3$ outneighbours
in~$\wt{L}$ (or both), and similarly for the exceptional vertices of type~$BR$.

Recall that $M_V^{RL}$ is empty. We consider two further subcases
according to whether or not $M_V^{LR}$ is also empty.

\medskip

\noindent
\textbf{Case~2.1} $M_V^{LR}\neq \emptyset$.

\smallskip

\noindent
We start by choosing an edge $yy'$ from $\wt{B}\cup BR$ to $\wt{L}\cup TL$
such that $y\in \wt{B}\sm V_0^*$ or $y'\in \wt{L}\sm V_0^*$ (or both).
Such an edge exists by Lemma~\ref{MBL1}(v), since $|V_0^*|=2|V_0|\ll \beta n$.
If both $y\in \wt{B}\sm V_0^*$ and $y'\in \wt{L}\sm V_0^*$, then we can use $yy'$
for the transition from~$\wt{B}$ to $\wt{L}$. Together with suitable useful shifted walks in $L$
and in~$R$ this will achieve the transition from~$y_2$ to~$x_1$.
For the transition from~$\wt{T}$ to $\wt{R}$ we use one vertex in $M_V^{LR}$,
together with a twin of this vertex in $M_H^{LR}$.
Now we may suppose that $y\notin \wt{B}\sm V_0^*$ or $y'\notin \wt{L}\sm V_0^*$.
We only consider the case when the former holds, as the other case is similar.
We will still aim to use $yy'$ for the transition from~$\wt{B}$ to $\wt{L}$,
although we need to make the following adjustments according to various cases for $y$.

If $y\in BR$ then we relabel the $BR$ list so that $y_2=y$. We can then use the edge $yy'$
(together with a suitable useful shifted walk in $L$)
to obtain a transition from~$y_2$ to~$x_1$.

Suppose next that some edge in $\match_R$ joins an exceptional vertex $u\in BR$ to~$y$.%
    \COMMENT{We cannot just relabel the $BR$ list so that $y_2=u$ in this case and then use
$yy'$ for the transition since then the $F$-cycle containing~$y$ would not be balanced.}
If $u$ has an outneighbour $v\in \wt{R}\sm (V_0^*\cup \{y,y'\})$, then we replace
the edge $uy$ by $uv$ and can now use $yy'$ to obtain a transition from~$\wt{B}$ to $\wt{L}$.
If $u$ has no such outneighbour, then $u$ must have at least $\beta n/3$ inneighbours in $\wt{B}$
(by our assumption on the exceptional vertices at the beginning of Case~2) as well as
at least $\beta n/3$ outneighbours in $\wt{L}$.
Pick such an inneighbour $u^-\in \wt{B}\sm V_0^*$ and such an outneighbour $u^+\in \wt{L}\sm V^*_0$.
We can now use the path $u^-uu^+$ to obtain a transition from~$\wt{B}$ to $\wt{L}$.

Finally, suppose that some edge in $\match_B$ joins~$y$ to an exceptional vertex $u\in BR$.
If $u$ has an inneighbour $v\in \wt{B}\sm (V_0^*\cup \{y,y'\})$, then we replace
the edge $yu$ by $vu$ and can now use $yy'$ to obtain a transition from~$\wt{B}$ to $\wt{L}$.
If $u$ has no such inneighbour, then $u$ must have at least $\beta n/3$ inneighbours in $\wt{T}$.
Pick such an inneighbour $u^-\in \wt{T}\sm (V_0^*\cup \{y,y'\})$ and let $u^+\in \wt{R}$ be the outneighbour of~$u$
in the matching $\match_R$. We now use $yy'$ to obtain a transition from~$\wt{B}$ to $\wt{L}$
and the path $u^-uu^+$ to obtain a transition from~$\wt{T}$ to $\wt{R}$.


(Note that $y$ cannot be an endvertex of some edge in $\match_T$ or $\match_L$ outside~$V_0$,
since $y\in \wt{B}\cup BR$ and $B\subseteq R\cup M_V$.)%
    \COMMENT{Here is what happens in the case when $y'\notin \wt{L}\sm V_0^*$.
If $y'\in TL$ we reorder the $TL$ list to have $x_1=y'$ and then use $yy'$ for the transition.
If some edge in $\match_T$ joins $y'$ to an exc vx $u\in TL$ and we cannot find another inneighbour of $u$ then
$u$ has many inneighbours in $\wt{B}$ and many outneighbours in $\wt{L}$. So we can use a 2-path with
inner vx $u$ for the transition from~$\wt{B}$ to~$\wt{L}$. [This can't happen...
If some edge in $\match_B$ joins $y'$ to an exc vx
$u\in BR$ and we cannot find another inneighbour of $u$ then
$u$ has many inneighbours in $\wt{T}$. So we can use a 2-path with
inner vx $u$ for the transition from~$\wt{T}$ to~$\wt{R}$ and $yy'$ for the transition from~$\wt{B}$ to~$\wt{L}$.
 -- ?? The latter cannot hold since $y'\in T\cup TL$ and $B\subseteq R\cup M_V$.]
If some edge in $\match_L$ joins an exc vx $u\in TL$ to $y'$ and we cannot find another outneighbour of $u$ then
$u$ has many outneighbours in $\wt{R}$. So we can use a 2-path with
inner vx $u$ for the transition from~$\wt{T}$ to~$\wt{R}$ and $yy'$ for the transition from~$\wt{B}$ to~$\wt{L}$.}

\medskip

\noindent
\textbf{Case~2.2} $M_V^{LR}=\emptyset$.

\smallskip

\noindent
In this case we have $M_V=M_H=\emptyset$.
Thus $B \sub R$, and since $|B|=|R|$ we have $B=R$. Similarly $L=T$.

Next suppose that there are exceptional vertices $v_1\neq v_2$ such that
\textno $v_1$ has least $\beta n/3$ inneighbours in $\wt{T}$ and at least $\beta n/3$ outneighbours in $\wt{R}$. &(v_1)

\textno $v_2$ has least $\beta n/3$ inneighbours in $\wt{B}$ and at least $\beta n/3$ outneighbours in $\wt{L}$. &(v_2)

Then we use $v_1$ to get a transition from~$\wt{T}$ to $\wt{R}$,
and use $v_2$ to get a transition from~$\wt{B}$ to $\wt{L}$.
Now we may suppose that we cannot find $v_1$ and $v_2$ as above. It may even be that we can find neither
$v_1$ nor $v_2$ as above, in which case the following holds.
  \textno Every vertex of type $TL$ has at least $\beta n/3$ inneighbours in~$\wt{T}$
and at least $\beta n/3$  outneighbours in $\wt{L}$. Every vertex of type $BR$ has
at least $\beta n/3$ inneighbours~$\wt{B}$ and at least $\beta n/3$ outneighbours in $\wt{R}$. &(\spadesuit)

For the remainder of the proof we suppose that either $(v_1)$ or $(\spadesuit)$ holds,
as the case $(v_2)$ is similar to $(v_1)$. We consider the partition
$V(G) = L' \cup R'$ where $L':= \wt{L} \cup TL$ and $R' = \wt{R} \cup BR$.
When $(v_1)$ holds we add $v_1$ to either $L'$ or $R'$ such that the sets obtained in this
way have different size. We still denote these sets by $L'$ and $R'$.
The following lemma will supply the required transitions.

\begin{lemma}\label{final} $ $
\begin{itemize}
\item[(i)] When $(v_1)$ holds there is an edge $yy'$ from
$R'\sm\{v_1\}$ to $L'\sm\{v_1\}$ having at least one endvertex in $(\wt{L} \cup \wt{R}) \sm V_0^*$.
\item[(ii)] When $(\spadesuit)$ holds there are two disjoint edges $xx'$ and $yy'$ with $x,y' \in L'$, $x',y \in R'$
such that both edges have at least one endpoint in $(\wt{L} \cup \wt{R})\sm V_0^*$.
\end{itemize}
\end{lemma}

\proof
Suppose first that we have $d^+_{(1-\beta)n/2} \ge (1+\beta)n/2$.
Then we have at least $(1+\beta)n/2$ vertices of outdegree at least $(1+\beta)n/2$.
Since $|\wt{L}|,|\wt{R}|=n/2\pm 19\eta n$ (by Lemma~\ref{propRL}) and $|V_0^*|\ll \beta n$,
we can choose vertices $x\in \wt{L} \sm V_0^*$ and $y\in \wt{R} \sm V_0^*$ with outdegree
at least $(1+\beta)n/2$, and then outneighbours $x' \ne y$ in $\wt{R} \sm V_0^*$ of $x$ and $y' \ne x$ in
$\wt{L} \sm V_0^*$ of $y$. Then $xx'$ and $yy'$ are the edges required in~(ii)
and $yy'$ is the edge required in~(i). A similar argument applies
if $d^-_{(1-\beta)n/2} \ge (1+\beta)n/2$. Therefore we can assume that
$d^+_{(1-\beta)n/2} < (1+\beta)n/2$ and $d^-_{(1-\beta)n/2} < (1+\beta)n/2$.
Now our degree assumptions give $d^+_{(1-\beta)n/2} \ge n/2$ and $d^-_{(1-\beta)n/2} \ge n/2$,
so there are at least  $(1+\beta)n/2$ vertices of outdegree at least $n/2$
and at least $(1+\beta)n/2$ vertices of indegree at least $n/2$. We consider cases according to the
size of~$L'$ and~$R'$.

If $|R'|>|L'|$ then we have sets $X, Y \sub \wt{L} \sm V_0^*$ with $|X|, |Y| \ge \beta n/3$
such that every vertex in $X$ has at least $2$ outneighbours in $R'$ and
every vertex in $Y$ has at least $2$ inneighbours in $R'$.
Then we can obtain the required edges greedily: if $(\spadesuit)$ holds, choose any $x\in X$,
an outneighbour $x' \in R'$ of $x$, any $y'\in Y$ with $y' \ne x$
and any inneighbour $y \in R'$ of $y'$ with $y \ne x'$; if $(v_1)$ holds, then
we choose $y\in R'\sm\{v_1\}$ and an outneighbour $y'\in L'\sm\{v_1\}$.
A similar argument applies when $|L'|>|R'|$.

Finally we have the case $|L'|=|R'|=n/2$, in which case $(\spadesuit)$ holds by definition of $L'$ and $R'$.
Then we have sets $X \sub \wt{L} \sm V_0^*$ and $Y \sub \wt{R} \sm V_0^*$
of vertices with outdegree at least $n/2$, with $|X|, |Y| \ge \beta n/3$.
Note that each $x \in X$ has at least one outneighbour in $R'$
and each $y \in Y$ has at least one outneighbour in $L'$.
Choose some $x_0 \in X$ and an outneighbour $x_0' \in R'$ of $x_0$.
If there is any $y \in Y$, $y \ne x_0'$ with an outneighbour $y' \ne x_0$ in~$L'$ then
$x_0 x_0'$ and $yy'$ are our required edges. Otherwise, we have
the edge $yx_0$ for every $y \in Y$ with $y \ne x_0'$. So we choose some other
$x \in X$ with $x \ne x_0$, an outneighbour $x' \in R'$ of $x$ and
a vertex $y \in Y \sm \{x_0',x'\}$, and our required edges are
$xx'$ and $yx_0$.
\endproof

Now suppose that $(v_1)$ holds. Let $yy'$ be the edge provided by Lemma~\ref{final}(i).
If $y\in \wt{R}\sm V_0^*$ and $y'\in \wt{L}\sm V_0^*$, then we can use $yy'$ for the transition from~$\wt{B}=\wt{R}$
to $\wt{L}$. For the transition from~$\wt{T}=\wt{L}$ to $\wt{R}$ we use a path $v^-_1v_1v^+_1$ such
that $v^-_1\in \wt{L}\sm (V_0^*\cup \{y,y'\})$ and $v^+_1\in \wt{R}\sm (V_0^*\cup \{y,y'\})$.
(Such a path exists since $v_1$ has many inneighbours in~$\wt{L}$ and many outneighbours in~$\wt{R}$.)
Now we may suppose that either $y\notin \wt{R}\sm V_0^*$ or $y'\notin \wt{L}\sm V_0^*$.
We only consider the case when the former holds, as the other case is similar.
We still aim to use $yy'$ for the transition from~$\wt{B}$ to $\wt{L}$,
although we need to make adjustments as in Case~2.1.
For example, consider the case when some edge in $\match_B$ joins~$y$
to an exceptional vertex $u\in BR$. If $u$ has an inneighbour $v\in \wt{B}\sm (V_0^*\cup \{y,y'\})$, then we replace
the edge $yu$ by $vu$ and can now use $yy'$ to obtain a transition from~$\wt{B}$ to $\wt{L}$.
If $u$ has no such inneighbour, then $u$ must have at least $\beta n/3$ inneighbours in $\wt{T}=\wt{L}$.
Choose such an inneighbour $u^-\in \wt{T}\sm (V_0^*\cup\{y,y'\})$
and let $u^+$ be the outneighbour of $u$ in $\match_R$. Add $v_1$ to the set $TL$, $BR$
which contained it previously. We can now use $yy'$ to obtain a transition from~$\wt{B}$ to $\wt{L}$
and the path $u^-uu^+$ to obtain a transition from~$\wt{T}$ to $\wt{R}$.
The other cases are similar to those in Case 2.1.%
    \COMMENT{If eg $y\in BR$ then $y\neq v_1$ and so we can still relabel the $BR$ list to get that $y_2=y$ and
use $yy'$ for the transition from~$\wt{B}$ to $\wt{L}$ and $v_1$ for transition from~$\wt{T}$ to $\wt{R}$.
If some edge in $\match_R$ join an exceptional vertex $u\in BR$ to~$y$ and $u=v_1$ then we just choose another
outneighbour of $v_1$. The remaining cases are really similar.}

Finally, suppose that $(\spadesuit)$ holds. Let $xx'$, $yy'$ be the edges provided by Lemma~\ref{final}(ii).
Then by changing the edges in $\match_L\cup \match_R \cup \match_T\cup \match_B$ if necessary
we can ensure that $x,x',y,y'\notin V_0^*$. Now we can use $xx'$ for the transition from~$\wt{T}$
to~$\wt{R}$ and $yy'$ for the transition from~$\wt{B}$ to $\wt{L}$. This is clear if none
of $x,x',y,y'$ lies in $V_0$. But if we have $x\in V_0$ (for example), then $x\in TL$,
and relabelling the $TL$ list so that $x=x_2$ we can use $xx'$ for the transition from~$x_2$ to~$y_1$.
The other cases are similar.

\medskip

\noindent
In all of the above cases for $|TR|=|BL|=0$ we obtain a transition from $\wt{B}$ to $\wt{L}$
and a transition from~$\wt{T}$ to $\wt{R}$. Now we can complete the proof as in Subsection~\ref{subsec:BLstar}.
Here no cluster is full, and so every cluster is $4$-good. Moreover, every vertex outside~$V_0$
is an endvertex of at most one edge in~$\cE$. Thus as in Lemma~\ref{singlecycle} one can show that there
are cycles $\C_L,\C_R\in \C$ such that $\C_L$ contains all vertices belonging to clusters in~$L$,
$\C_R$ contains all vertices belonging to clusters in~$R$ and every exceptional vertex lies
in~$\C_L$ or $\C_R$. Moreover, since every cluster is
$4$-good and every vertex outside~$V_0$ is an endvertex of at most one edge
in~$\cE$, Lemma~\ref{4goodcycles} now implies that
for every cycle $D$ in~$F$ there is a cycle $C'$ in~$\C$ such that $C'$ contains all
vertices belonging to clusters in~$D$.
Now considering the transition from $\wt{B}$ to $\wt{L}$ (say) we see that $\C_L=\C_R$.
Since $(\star)$ implies that every cycle of~$F$
contains at least one cluster from $L\cup R$, we also have that all vertices in $\wt{M}$ are
contained in $\C_L=\C_R$. Thus $\C=\C_L=\C_R$ is a Hamilton cycle of~$G$.
This completes the proof of Theorem~\ref{CKKO - Approximate Chvatal}.

\section{A concluding remark}

The following example demonstrates that the degree properties used
in our main theorem cannot be substantially improved using our current method.
Let $G$ be a digraph with $V(G) = \{1,\dots,n\}$ such that $ij \in E(G)$ for
every $1 \le i<j \le n$ and also for every $1 \le j < i \le an+1$ and
$n-an \le j < i \le n$, for some $0<a<1/2$.
Then $G$ has minimum semidegree $an$ and satisfies $d^+_i, d^-_i \ge i-1$
for all $1 \le i \le n$, so if we apply the regularity lemma it
is `indistinguishable' from a digraph satisfying
the hypotheses of Conjecture~\ref{Nash-Williams Conjecture - Posa}.
However, any disjoint union of cycles in $G$ covers
at most $2an$ vertices, so the argument used in Section 6 breaks down.
We remark that our argument may well still be useful in combination with
a separate method for treating the case when the reduced digraph cannot be
nearly covered by disjoint cycles.

\bigskip

\noindent
{\footnotesize
Demetres Christofides, Daniela K\"uhn, Deryk Osthus, School of Mathematics,
University of Birmingham, Birmingham, B15 2TT, United Kingdom,
{\tt \{christod,kuehn,osthus\}@maths.bham.ac.uk}

\smallskip

\noindent Peter Keevash, School of Mathematical Sciences, Queen
Mary, University of London, Mile End Road, London, E1 4NS, United
Kingdom, {\tt p.keevash@qmul.ac.uk}
}

\end{document}